\documentclass[11pt,a4paper]{article}

\usepackage{setspace}
\usepackage[normalem]{ulem}
\usepackage[margin=1in]{geometry}
\usepackage{amsmath, amsfonts, amsthm, natbib}
\usepackage[toc,page]{appendix}
\usepackage{bbold} % to write indicator function by \mathbb{1}
\usepackage{graphicx, algorithm, enumerate, enumitem, color, hyperref}
\usepackage[noend]{algpseudocode}  % to write pseudocode
\usepackage{caption}

\newcommand{\real}{\mathbb R}
\def\st{\mathrm{\quad s.t.\quad}}
\def\half{\frac{1}{2}}
\def\cov{\mathrm{Cov}}
\def\E{\mathrm{E}}
\def\Var{\mathrm{Var}}
\def\Prob{\mathrm{P}}
\def\tr{\operatorname{tr}}
\def\vec{\operatorname{vec}}

\def\rank{\operatorname{rank}}
\def\tSig{{\Sigma}}

\def\tL{{L}}
\def\tLt{{L^{T}}}
\def\I{\mathcal{I}}

\newtheorem{theorem}{Theorem}
\newtheorem{lemma}[theorem]{Lemma} 
 
\newtheorem{remark}[theorem]{Remark}
\newtheorem{corollary}[theorem]{Corollary}

\newcommand{\vertiii}[1]
{{\left\vert\kern-0.25ex\left\vert\kern-0.25ex\left\vert #1
\right\vert\kern-0.25ex\right\vert\kern-0.25ex\right\vert}}

\newcommand{\norm}[1]{\left\lVert#1\right\rVert}

\DeclareMathOperator*{\argmin}{arg\,min}

\graphicspath{ {./art/} }

\title{Learning Local Dependence In Ordered Data}
\author{Guo Yu\thanks{Department of Statistical Science, Cornell University, 1173 Comstock Hall, Ithaca, NY 14853, \href{mailto:gy63@cornell.edu}{gy63@cornell.edu} } \and
Jacob Bien\thanks{Department of Biological Statistics and Computational Biology and Department of Statistical Science, Cornell University, 1178 Comstock Hall, Ithaca, NY 14853, \href{mailto:jbien@cornell.edu}{jbien@cornell.edu}}}
\date{}

\begin{document}

\maketitle

\begin{abstract}%   <- trailing '%' for backward compatibility of .sty file
  In many applications, data come with a natural ordering.
  This ordering can often induce local dependence among nearby variables.
  However, in complex data, the width of this dependence may vary, making simple assumptions such as a constant neighborhood size unrealistic.
  We propose a framework for learning this local dependence based on estimating the inverse of the Cholesky factor of the covariance matrix.
  Penalized maximum likelihood estimation of this matrix yields a simple regression interpretation for local dependence in which variables are predicted by their neighbors.
  Our proposed method involves solving a convex, penalized Gaussian likelihood problem with a hierarchical group lasso penalty. 
  The problem decomposes into independent subproblems which can be solved efficiently in parallel using first-order methods.  
  Our method yields a sparse, symmetric, positive definite estimator of the precision matrix,
  encoding a Gaussian graphical model. 
  We derive theoretical results not found in existing methods 
  attaining this structure. In particular, our conditions for signed support recovery and 
  estimation consistency rates in multiple norms are as mild as those in a regression problem. 
  Empirical results show our method performing favorably compared to existing methods. 
  We apply our method to genomic data to flexibly model linkage disequilibrium. 
  Our method is also applied to improve the performance of discriminant analysis in sound recording classification.
\end{abstract}

\section{Introduction} \label{sec:introduction}
% Section 1 Introduction
Estimating large inverse covariance matrices is a fundamental problem in modern multivariate statistics.
Consider a random vector $X = \left( X_1,\ldots,X_p \right)^T \in {\mathbb{R}}^p$ with mean zero and covariance
matrix $E (X X^T) = \Sigma$. 
Unlike the covariance matrix, which
captures marginal correlations among variables in $X$, the
inverse covariance matrix $\Omega = \Sigma^{-1}$ (also known as the precision matrix) 
characterizes conditional correlations
and, under a Gaussian model, $\Omega_{jk} = 0$ implies
that $X_j$ and $X_k$ are conditionally independent given all other variables.
When $p$ is large, it is common to regularize the precision matrix 
estimator by making it sparse \citep[see, e.g.,][]{pourahmadi2013high}.   
This paper focuses on the special context in which variables have a natural ordering, such as when data are
collected over time or along a genome.
In such a context, it is often reasonable to assume that random variables
that are far away in the ordering are less dependent than those that are close together.
For example, it is known that genetic mutations that occur close together on a chromosome 
are more likely to be coinherited than mutations that are located far apart.
We propose a method for estimating the 
precision matrix based on this assumption while also allowing each
random variable to have its own notion of closeness.

In general settings where variables do not necessarily have a known
ordering, two main types of convex methods 
with strong theoretical results have been developed
for introducing sparsity in $\Omega$. 
The first approach, known as the \textit{graphical lasso} \citep{yuan2007model, banerjee2008model, friedman2008sparse,
rothman2008sparse}, performs penalized maximum likelihood, solving
$\min_{\Omega \succ 0, \Omega = \Omega^T} 
\mathcal{L} \left( \Omega \right) + \lambda P\left( \Omega \right)$,
where $\mathcal{L}(\Omega)=-\log\det\Omega + n^{-1}
\sum_{i=1}^n x_i^T \Omega x_i$ is, up to constants, the
negative log-likelihood of a sample of $n$ independent Gaussian
random vectors and $P(\Omega)$ is the (vector) $\ell_1$-norm of $\Omega$.
\cite{zhang2014sparse} introduce a new convex loss function called the
\textit{D-trace loss} and propose a positive definite precision matrix estimator by minimizing an $\ell_1$-penalized version of this loss.
The second approach is through penalized pseudo-likelihood, the most
well-known of which is called \textit{neighborhood selection} \citep{meinshausen2006high}. 
Estimators in this category are usually solved by a column-by-column approach and thus are more amenable to theoretical
analysis \citep{yuan2010high,cai2011constrained,liu2012high, liu2012tiger, sun2013sparse, khare2014convex}. 
However they are not guaranteed to be positive definite and do not exploit the symmetry of $\Omega$.
\citet{peng2012partial} propose a partial correlation matrix estimator that 
develops a symmetric version of neighborhood selection; however, positive definiteness is still not guaranteed.

In the context of variables with a natural ordering, by contrast,
almost no work uses convex optimization to flexibly estimate
$\Omega$ while exploiting the ordering structure.
Sparsity is usually induced via the Cholesky decomposition of $\Sigma$,
which leads to a natural 
interpretation of sparsity.
Consider the Cholesky decomposition $\Sigma = QQ^T$, which implies $\Omega = L^T L$ for $L = Q^{-1}$ for
lower triangular matrices $Q$ and $L$ with positive diagonals.
The assumption that $X \sim N \left( 0, \Sigma \right)$
is then equivalent to a set of linear models in terms of rows of $L$, i.e., 
$L_{11}X_1=\varepsilon_1$ and
\begin{align}
  L_{rr}X_r = - \sum_{k=1}^{r-1} L_{rk} X_k + \varepsilon_r \quad r = 2,\ldots,p,
  \label{eq:regression}
\end{align}
where $\varepsilon \sim N\left( 0, I_p \right)$. 
Thus, $L_{rk} = 0$ (for $k < r$) can be interpreted as meaning that in predicting
$X_r$ from the previous random variables, one does not need to know $X_k$. This observation has motivated previous work, including 
\citet{pourahmadi1999joint,wu2003nonparametric,huang2006covariance, shojaie2010penalized, 2016arXiv161002436K}.
While these methods assume sparsity in $L$, they do not require local dependence because each variable is allowed to be
dependent on predecessors that are distant from it (compare the upper left to the upper right panel of Figure~\ref{fig:bestL}).

The assumption of ``local dependence''
can be expressed as saying
that each variable $X_r$ can be best explained by exactly its $K_r$
closest predecessors:
\begin{align}
  L_{rr}X_r = - \sum_{k=r-K_r}^{r-1} L_{rk} X_k + \varepsilon_r, \quad \text{for} \quad L_{rk} \neq 0, \quad r - K_r \leq k \leq r-1, \quad r = 2,\ldots,p.
  \label{eq:localdependence}
\end{align}
Note that this does not describe all patterns of a variable depending on its nearby variables.
For example, $X_r$ can be dependent on $X_{r-2}$ but not on $X_{r-1}$.
In this case, the dependence is still local, but would not be captured by \eqref{eq:localdependence}. 
We focus on the restricted class \eqref{eq:localdependence}
since it greatly simplifies the interpretation of the learned dependence structure by capturing the extent of this dependence in a single number $K_r$, the neighborhood size.

Another desirable property of model \eqref{eq:localdependence} is that it admits a simple connection between the sparsity
pattern of $L$ and the sparsity pattern of the precision matrix $\Omega$ in the Gaussian graphical model.
In particular, straightforward algebra shows that for $j < k$,
\begin{align}
  L_{kj} = \dots = L_{pj} = 0 \implies \Omega_{jk} = 0.
  \label{eq:sparsity}
\end{align}
Statistically, this says that if none of the variables $X_k,\dots, X_p$ depends on $X_j$ in the sense of \eqref{eq:regression},
then $X_j$ and $X_k$ are conditionally independent given all other variables.

\citet{bickel2008regularized} study theoretical properties
in the case that all bandwidths, $K_r$, are equal, in which case model \eqref{eq:localdependence} is a $K_r$-ordered 
antedependence model \citep{zimmerman2009antedependence}.
A banded estimate of $L$
then induces a banded estimate of $\Omega$. 
The \textit{nested lasso} approach of \citet{levina2008sparse} provides for ``adaptive
banding'', allowing $K_r$ to vary with $r$ \citep[which corresponds to variable-order antedependence models in][]{zimmerman2009antedependence};
however, the nested lasso
is non-convex, meaning that
the proposed algorithm does not necessarily minimize the stated
objective and theoretical properties of this estimator
have not been established.

In this paper, we propose a penalized likelihood approach that
provides the flexibility of the nested lasso but is formulated
as a convex optimization problem, which allows us to prove strong
theoretical properties and to provide an efficient, scalable algorithm
for computing the estimator.  The theoretical development of our
method allows us to make clear comparisons with known results for
the graphical lasso \citep{rothman2008sparse,ravikumar2011high} in the non-ordered
case.  Both methods are convex penalized likelihood approaches, so
this comparison highlights the similarities and differences in the ordered and non-ordered problems.

There are two key choices we make that lead to a convex formulation.
First, we express the optimization problem in terms of the Cholesky
factor $L$.  The nested lasso and other methods (starting with
\citealt{pourahmadi1999joint}) use the modified
Cholesky decomposition, $\Omega = T^T D^{-1} T$, where $T$ is
a lower-triangular matrix 
with ones on its diagonal and $D$ is a diagonal matrix with
positive entries. While 
$\mathcal{L}(\Omega)$ is convex in $\Omega$, the negative
log-likelihood $\mathcal{L}(T^T D^{-1} T)$ is not jointly
convex in $T$ and $D$.  By contrast, 
\begin{align}
  \mathcal{L}\left( L^TL \right) &= -\log \det \left( L^TL \right) + \frac{1}{n} 
  \sum_{i=1}^n {x}_i^T L^T L {x}_i= -2 \sum_{r=1}^p \log L_{rr} + \frac{1}{n} \sum_{i=1}^n 
  \norm{L {x}_i}_2^2
  \label{intro:negloglik}
\end{align}
is convex in $L$.
This parametrization is considered in
\citet{JMLR:v16:aragam15a}, \citet{khare2014convex}, and \citet{2016arXiv161002436K}.
Maximum likelihood estimation of $L$ preserves the regression interpretation by noting 
that 
\begin{align}
  \mathcal{L}\left( L^TL \right) = 
  -2 \sum_{r=1}^p \log L_{rr} + \frac{1}{n} \sum_{r=1}^p \sum_{i=1}^n
  L_{rr}^2 \left( x_{ir} + \sum_{k = 1}^{r-1} L_{rk} x_{ik} / L_{rr} \right)^2.
  \nonumber
\end{align}
This connection has motivated
previous work
with the modified Cholesky decomposition, in which
$T_{rk}=-L_{rk}/L_{rr}$ are the coefficients of a linear model in which $X_r$
is regressed on its predecessors, and $D_{rr}=L_{rr}^{-2}$ corresponds to the error
variance. The second key choice 
is our use of a hierarchical group lasso in place of the nested lasso's nonconvex penalty.

We introduce here some notation used throughout the paper.
For two sequences of constants $a(n)$ and $b(n)$, the notation $a(n) = o\left( b(n) \right)$ means that for every $\varepsilon > 0$,
there exists a constant $N > 0$ such that
$|a(n) / b(n)| \leq \varepsilon$ for all $n \geq N$. 
And the notation $a(n) = \mathcal{O} \left( b(n) \right)$ means that there exists a constant $N > 0 $ and a constant $M > 0$ such that
$|a(n) / b(n)| \leq M$ for all $n \geq N$.
For a sequence of random variables $A(n)$,
the notation $A(n) = \mathcal{O}_P\left( b(n) \right)$ means that for every $\varepsilon > 0$, there exists a constant $M > 0$ such that 
$\mathrm{P}\left( |A(n) / b(n)| > M  \right) \leq \varepsilon$ for all $n$.

For a vector $v = \left( v_1,\ldots,v_p \right) \in \real^p$, 
we define $\norm{v}_1 = \sum_{j=1}^p |v_j|$, $\norm{v}_2 = ( \sum_{j=1}^p v_j^2 )^{1/2}$ and 
$\norm{v}_\infty = \max_{j} |v_j|$. For a matrix $M \in \real^{n \times p}$, 
we define the element-wise norms by two vertical bars. Specifically,
$\norm{M}_\infty = \max_{jk} |M_{jk}|$ and Frobenius norm
$\norm{M}_F = ( \sum_{j,k} M_{jk}^2 )^{1/2}$.
For $q \geq 1$,
we define the matrix-induced (operator) $q$-norm by three vertical bars:
$\vertiii{M}_q = \max_{\norm{v}_q = 1} \norm{Mv}_q$.
Important special cases include $\vertiii{M}_2$, also known as the spectral norm, which is the largest singular value of $M$, 
as well as
$\vertiii{M}_1 = \max_k \sum_{j=1}^p |M_{jk}|$ and $\vertiii{M}_\infty = \max_j \sum_{k=1}^p |M_{jk}|$.
Note that $\vertiii{M}_1 = \vertiii{M}_\infty$ when $M$ is symmetric.

Given a $p$-vector $v$, a $p\times p$ matrix $M$, and an index set
$T$, let $v_T = ( v_i
)_{i \in T }$ be the $|T|$-subvector and $M_T$
the $p \times |T|$ submatrix with columns selected
from $T$.  Given a second index set $T'$, let $M_{TT'}$ be the $|T|\times|T'|$ submatrix 
with rows and columns of $M$ indexed by $T$ and $T'$, respectively. 
Specifically, we use $L_{r\cdot}$ to denote the $r$-th row of $L$.

% Section 2 Estimator
\section{Estimator} \label{sec:estimator}
For a given tuning parameter $\lambda \geq 0$, we define our estimator $\hat{L}$
to be a minimizer of the following penalized negative Gaussian log-likelihood 
\begin{align}
  \hat L \in \argmin_{\substack{L: L_{rr}>0 \\ L_{rk}=0\text{ for }r < k }}
  \left\{-2\sum_{r=1}^p
  \log L_{rr} + \frac{1}{n} \sum_{i=1}^n \norm{L {x}_i}_2^2 + 
  \lambda \sum_{r=2}^p P_r\left( L_{r\cdot} \right) \right\}.
  \label{est:estimator}
\end{align}
The penalty $P_r$, which is applied to the $r$-th row, is defined by
\begin{align}
  P_r\left( L_{r\cdot} \right) =  \sum_{\ell = 1}^{r-1} \norm{W^{(\ell)} \ast L_{g_{r,\ell}}}_2 = 
  \sum_{\ell = 1}^{r-1} \left( \sum_{m=1}^\ell w_{\ell m}^2 L_{rm}^2 \right)^{1/2},
  \label{est:penalty}
\end{align}
where $W^{(\ell)} = (w_{\ell 1}, \ldots, w_{\ell \ell}) \in \real^{\ell}$ is a vector of weights, $\ast$ denotes element-wise multiplication,
and $L_{g_{r, \ell}}$ denotes the vector of elements of $L$ from the group $g_{r, \ell}$, which corresponds to the first $\ell$ elements in the $r$-th row
(for $1 \leq \ell \leq r - 1$):
$$g_{r,\ell} = \left\{ (r, \ell'): \ell' \leq \ell \right\}.$$
Since $g_{r,1} \subset g_{r,2} \subset \cdots \subset g_{r, r-1}$,
each row $r$ of $L$ is penalized with a sum of $r-1$ nested, weighted $\ell_2$-norm
penalties. This is a hierarchical group lasso penalty \citep{yuan2007model,zhao2009composite,jenatton2011structured, yan2015hierarchical} with group
structure conveyed in Figure~\ref{fig:group}.

With $w_{\ell m} > 0$, this nested structure always puts more penalty on those
elements that are further away from the diagonal.
Since the group lasso has the effect of setting to zero a subset of groups,
it is apparent that
this choice of groups ensures that whenever the elements in $g_{r, \ell}$ are set to zero, elements in $g_{r, \ell'}$ are also set to zero for all $\ell' \leq \ell$.
In other words, for each row of $\hat L$, the non-zeros are those elements within some
(row-specific) distance of the diagonal.
This is in contrast to the $\ell_1$-penalty as used in \citet{2016arXiv161002436K}, which produces sparsity patterns with no particular structure
(compare the top-left and top-right panels of Figure~\ref{fig:bestL}).

The choice of weights, $w_{\ell m}$, affects both the empirical
and theoretical performance of the estimator. We focus primarily on a quadratically decaying set of weights,
\begin{align}
  w_{\ell m} = \frac{1}{\left( \ell - m + 1 \right)^2},
  \label{est:generalweight}
\end{align}
but also consider the unweighted case (in which $w_{\ell m}=1$).
The decay counteracts the fact that the elements of $L$ appear in
differing numbers of groups (for example $L_{r1}$ appears in $r-1$
groups whereas $L_{r,r-1}$ appears in just one group).  
In a related problem, \citet{bien2015convex} choose weights that
decay more slowly with $\ell-m$ than \eqref{est:generalweight}.
Our choice makes the enforcement of hierarchy weaker so that 
our penalty behaves more closely to the lasso penalty \citep{tibshirani1996regression}. 
The choice of weight sequence in \eqref{est:generalweight} is
more amenable to theoretical analysis; however, in practice the
unweighted case is more efficiently implemented and works well empirically.
\begin{figure}
  \centering
  \includegraphics[width=0.9\textwidth]{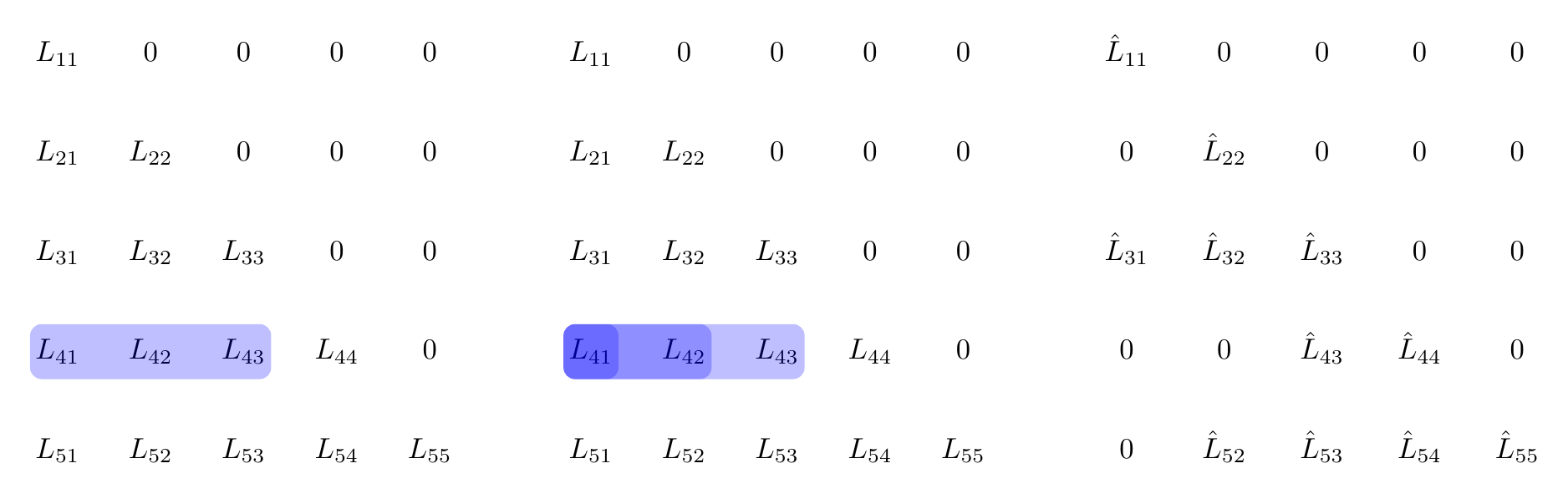}
  \caption{There are $p \choose 2$ groups used in the penalty, with each row $r$ having $r-1$ nested groups $g_{r,1} \subset g_{r,2} \subset \cdots \subset g_{r, r-1}$. 
  Left: the group $g_{4,3}$. Middle: the nested group structure $g_{4,1} \subset g_{4,2} \subset g_{4,3}$.
  Right: A possible sparsity pattern in $\hat{L}$, where elements in $g_{2,1}, g_{4,2}$ (and thus $g_{4,1}$) and $g_{5,1}$ are set to zero.}
  \label{fig:group}
\end{figure}

Problem \eqref{est:estimator} is convex in $L$. 
While $-\log\det(\cdot)$ is strictly convex, $-\sum_{r}\log(L_{rr})$ is
not strictly convex in $L$.  Thus, the $\argmin$ in
\eqref{est:estimator} may not be unique.  In Section \ref{sec:theory},
we provide sufficient conditions to ensure uniqueness with high probability.

In Appendix \ref{online:decouple}, 
we show that \eqref{est:estimator}
decouples into $p$ independent subproblems,
each of which estimates one row of $L$. More specifically,
let $\mathbf{X} \in \real^{n\times p}$ be a sample matrix with independent rows $x_i \sim N(0, \Sigma)$,
$ \hat{L}_{11} = n^{1/2}(\mathbf{X}_1^T \mathbf{X}_1)^{-1/2}$ and for $r = 2,\ldots,p$,
\begin{align}
  \hat{L}_{r,1:r} = \argmin_{\beta \in \real^r: \beta_r >0 }
  \left\{ -2\log \beta_r + \frac{1}{n}\norm{\mathbf{X}_{1:r}\beta}_2^2 
  + \lambda
  \sum_{\ell=1}^{r-1} \left( \sum_{m=1}^\ell w_{\ell m}^2 \beta_m^2 \right)^{1/2}\right\}.
  \label{est:subproblem}
\end{align}
This observation means that the computation can be
easily parallelized,
which potentially can achieve a linear speed up with the number of CPU cores. 
Theoretically, to analyze the properties
of $\hat{L}$ it is easier to start by studying an estimator of each row, i.e.,
a solution to \eqref{est:subproblem}. We will see in Section 
\ref{sec:theory} that problem \eqref{est:subproblem} has connections
to a penalized regression problem, meaning that both the assumptions and results
we can derive are better than if we were working with a penalty based on $\Omega$.

In light of the regression interpretation of \eqref{eq:regression}, $\hat L$ provides an interpretable notion of local
dependence; however, we can of course also use our estimate of $L$ to estimate $\Omega$: $\hat{\Omega} = \hat{L}^T \hat{L}$.
By construction, this estimator is both symmetric and positive definite.
Unlike a lasso penalty, which would induce unstructured sparsity in the estimate of $L$ and thus would not be guaranteed to produce a sparse estimate of $\Omega$,
the adaptively banded structure in our estimator of $L$ can yield a generally banded $\hat\Omega$ 
with sparsity pattern determined by \eqref{eq:sparsity}
(See the top-left and bottom-left panels in Figure~\ref{fig:bestL} for an example). 

\section{Computation} \label{sec:computation}
As observed above, we can compute $\hat L$ by 
solving (in parallel across $r$) problem \eqref{est:subproblem}.
Consider an alternating direction
method of multipliers (ADMM) approach that solves the equivalent problem
\begin{align}
  \min_{\beta, \gamma \in \real^r: \beta_r > 0}
  \left\{ -2 \log \beta_r + \frac{1}{n} \norm{\mathbf{X}_{1:r} \beta}_2^2 + \lambda 
  \sum_{\ell=1}^{r-1} \left( {\sum_{m=1}^\ell w_{\ell m}^2 \gamma_m^2} \right)^{1/2}
\st \beta = \gamma \right\}.
\nonumber
\end{align}
Algorithm \ref{alg:ADMM} presents the ADMM algorithm, which
repeatedly minimizes this problem's augmented Lagrangian 
over $\beta$, then over $\gamma$, and then updates the dual
variable $u \in \real^r$.
\begin{algorithm}
  \caption{ADMM algorithm to solve \eqref{est:subproblem}}
  \begin{algorithmic}[1]
    \Require $\beta^{(0)}$, $\gamma^{(0)}$, $u^{(0)}$, $\rho>0$, $t=1$.
    \Repeat
    \State
    \begin{flalign}
      \beta^{(t)} &\gets  
      \argmin_{\beta \in \real^r: \beta_r >0}\left\{ -2 \log \beta_r + \frac{1}{n} 
      \norm{\mathbf{X}_{1:r}\beta}_2^2 + \left( \beta - \gamma^{(t-1)} \right)^T u^{(t-1)}
      + \frac{\rho}{2} \norm{\beta-\gamma^{(t-1)}}_2^2 \right\} 
      \label{comp:updatebeta} &&
    \end{flalign}
    \State
    \begin{flalign}
      \gamma^{(t)} &\gets 
      \argmin_{\gamma \in \real^r} 
      \left\{ \frac{\rho}{2} \norm{\gamma - \beta^{(t)} - 
      \rho^{-1} u^{(t-1)}}_2^2 + \lambda \sum_{\ell=1}^{r-1} 
      \left( \sum_{m=1}^\ell w_{\ell m}^2 \gamma_m^2 \right)^{1/2}  \right\} 
      \label{comp:updategamma} &&
    \end{flalign}
    \State
    $u^{(t)} \gets u^{(t-1)} + \rho \left( \beta^{(t)} - 
    \gamma^{(t)} \right)$\;
    \State
    $t \gets t+1$\;
    \Until{convergence}
    \State \Return{$\gamma^{(t)}$}
  \end{algorithmic}
  \label{alg:ADMM}
\end{algorithm}
The main computational effort in the algorithm is in solving 
\eqref{comp:updatebeta} and \eqref{comp:updategamma}. 
Note that \eqref{comp:updatebeta} has a smooth objective
function. Straightforward calculus
gives the closed-form solution (see Appendix \ref{online:closeformupdatebeta} for detailed derivation),
\begin{align}
  &\beta^{(t+1)}_r = \frac{-B - \sqrt{B^2 - 8 A}}{2A} > 0
  \nonumber\\
  &\beta^{(t+1)}_{-r} = - \left( 2 S^{(r)}_{-r,-r} + 
  \rho I \right)^{-1} \left( 2 S^{(r)}_{-r,r} \beta^{(t+1)}_r 
  + u^{(t)}_{-r} - \rho \gamma^{(t)}_{-r} \right),
  \nonumber
\end{align}
where
\begin{align}
  &S^{(r)} = \frac{1}{n}\mathbf{X}_{1:r}^T \mathbf{X}_{1:r}
  \nonumber\\
  &A = 4 S^{(r)}_{r, -r} \left( 2 S^{(r)}_{-r,-r} + \rho I \right)^{-1} S^{(r)}_{-r,r} 
  - 2 S^{(r)}_{r,r} - \rho < 0
  \nonumber\\
  &B = 2 S^{(r)}_{r,-r} \left( 2 S^{(r)}_{-r,-r} + \rho I \right)^{-1}\left( u^{(t)}_{-r} 
  - \rho \gamma^{(t)}_{-r} \right) - u^{(t)}_r + \rho \gamma^{(t)}_r.
  \nonumber
\end{align}

The closed-form update above involves matrix inversion. With $\rho>0$,
the matrix $2 S^{(r)}_{-r,-r} + \rho I$ is invertible even when $r>n$. 
Since determining a good choice for the ADMM parameter $\rho$ is in
general difficult, we adapt the dynamic $\rho$ updating scheme described 
in Section 3.4.1 of \citet{boyd2011distributed}.

Solving \eqref{comp:updategamma} requires evaluating the proximal
operator of the hierarchical group lasso with general weights.  We
adopt the strategy developed in \citet{bien2015convex} (based on
a result of \citealt{jenatton2011structured}),
which solves the dual problem of \eqref{comp:updategamma} by
performing Newton's method on at most $r-1$ univariate functions.
The detailed
implementation is given in
Algorithm \ref{online:alg:BCDondual} in Appendix \ref{online:dualproblem}. 
Each application of Newton's method corresponds to performing an elliptical projection,
which is a step of blockwise coordinate ascent on the dual of \eqref{comp:updategamma}
(see Appendix \ref{online:ellipticalprojection} for details).
Finally we observe in Algorithm \ref{alg:updategammanoweight} that for the unweighted case ($w_{\ell m} = 1$), solving 
\eqref{comp:updategamma} is remarkably efficient. 
\begin{algorithm}
  \caption{Algorithm for solving \eqref{comp:updategamma} for unweighted estimator}
  \begin{algorithmic}[1]
    \Require $\beta^{(t)}, u^{(t - 1)} \in \real^r$, $\lambda, \rho > 0$.
    \State Initialize $\gamma^{(t)} = \beta^{(t)} + u^{(t-1)}/\rho$ and $\tau = \lambda/\rho$\;
    \For{$\ell = 1,\ldots,r-1$}
    \[
      \left( \gamma^{(t)} \right)_{1:\ell} \gets
      \left( 1 - \frac{\tau}{\norm{\left( \gamma^{(t)} \right)_{1:\ell}}_2} \right)_+ 
      \left( \gamma^{(t)} \right)_{1:\ell}
    \]
  \EndFor
  \State \Return{$\gamma^{(t)}$.}
\end{algorithmic}
\label{alg:updategammanoweight}
\end{algorithm}

The \texttt{R} package \texttt{varband} provides \texttt{C++} implementations of Algorithms \ref{alg:ADMM} and \ref{alg:updategammanoweight}.

\section{Statistical Properties} \label{sec:theory}
In this section we study the statistical properties of our estimator.
In what follows, we consider a lower triangular matrix $L$ having row-specific bandwidths, $K_r$.
The first $J_r = r-1-K_r$ elements of row $r$ are zero, and the band of
non-zero off-diagonals (of size $K_r$) is denoted $\I_r=\left\{ J_r+1,\dots,r-1\right\}$. 
We also denote $\I_r^c = \left\{ 1, 2, \dots, r \right\} \setminus \I_r$.
See Figure~\ref{fig:sparsityassumption} for a graphical example of $K_5, J_5, \mathcal{I}_4$, and $\mathcal{I}_4^c$.
\begin{figure}[h]
  \centering
  \includegraphics[width=0.4\textwidth]{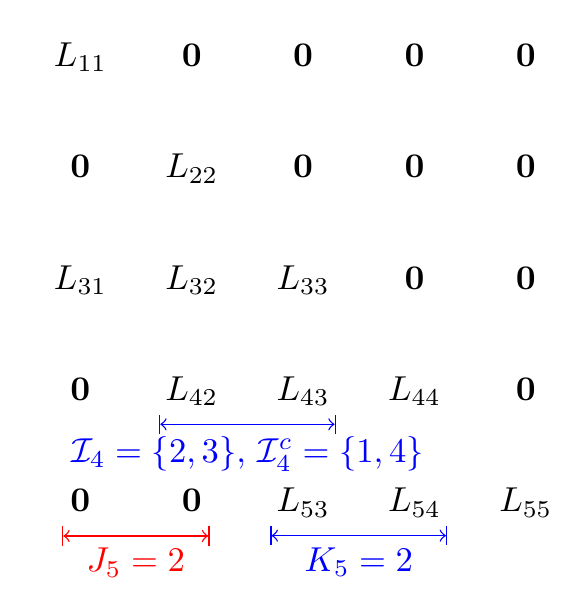}
  \caption{Schematic showing $J_r, K_r, \I_r$, and $\I_r^c$.}
  \label{fig:sparsityassumption}
\end{figure}

Our theoretical analysis is built on the following assumptions:
\begin{enumerate}[label=\textbf{A\arabic*}]
  \item \textit{Gaussian assumption}: \label{agaussian}
    The sample matrix $\mathbf{X} \in \real^{n\times p}$ has $n$ independent rows 
    with each row $x_i$ drawn from $N(\mathbf{0},\Sigma)$.
  \item \textit{Sparsity assumption}: \label{asparsity}
    The true Cholesky factor $L \in \real^{p\times p}$ is the lower triangular matrix 
    with positive diagonal elements
    such that the
    precision matrix $\Omega = \Sigma^{-1} = L^T L$.
    The matrix $L$ has row-specific bandwidths $K_r$ such that $L_{rj} = 0$ for $0 < j < r - K_r$.
  \item \textit{Irrepresentable condition}: \label{airrepre}
  There exists some $\alpha \in (0,1]$ such that
  \begin{align}
    \max_{2 \leq r \leq p} \max_{\ell \in \mathcal{I}^{c}_r } 
    \norm{\Sigma_{\ell\I_r} \left( 
    \Sigma_{\I_r\I_r}\right)^{-1}}_1 
    \leq \frac{6}{\pi^2}\left( 1-\alpha \right)
    \nonumber
  \end{align}
\item \textit{Bounded singular values}:\label{abddsval}
  There exists a constant $\kappa$ such that 
  \[
    0 < \kappa^{-1} \leq \sigma_{\min} \left( L \right) \leq 
    \sigma_{\max} \left( L \right) \leq \kappa
  \]
\end{enumerate}

When $\max_r K_r < n$, the Gaussianity assumption \ref{agaussian} implies that 
$\mathbf{X}_{\I_r}$ has full column rank for all $r$ with
probability one. 
Our analysis applies to the general high-dimensional scaling scheme where  
$K_r = K_r(n)$ and $p = p(n)$ can grow with $n$.

For $r = 2,\ldots,p$ and $\ell \in {\I}_r^c = \left\{ 1,\ldots,J_r,r \right\}$, let
\begin{align}
  \theta^{(\ell)}_{r} := \Var\left( X_\ell | X_{\mathcal{I}_r} \right)  \quad \text{and} \quad
  \theta_r := \max_{\ell \in \mathcal{I}_r^c} \theta^{(\ell)}_{r}.
  \nonumber
\end{align}
By Assumption \ref{agaussian}, 
$\theta_r^{(\ell)} =
\Sigma_{\ell\ell} - \Sigma_{\ell \mathcal{I}_r}
\left( \Sigma_{\mathcal{I}_r\mathcal{I}_r} \right)^{-1} 
\Sigma_{\mathcal{I}_r \ell}$
represents the noise variance when regressing $X_\ell$ on $X_{\mathcal{I}_r}$, i.e.,
for $\ell = 1,\ldots,J_r, r$, 
\begin{align}
  X_\ell = \Sigma_{\ell \mathcal{I}_r}( \Sigma_{\I_r \I_r} )^{-1} X_{\I_r}^T + E_\ell \qquad 
  \text{with} \qquad E_\ell \sim N\left( 0, \theta_r^{(\ell)} \right).
  \label{theory:theta_r}
\end{align}
In words,
$\theta_r^{(\ell)}$ measures the degree to which $X_\ell$ cannot be explained by the variables in the support
and $\theta_r$ is the maximum such value over all $\ell$ outside of the support $\mathcal{I}_r$ in the $r$-th row.
Intuitively, the difficulty of the estimation problem increases with $\theta_r$.
Note that for $r = 1,\ldots,p$, \eqref{eq:regression} implies $\theta^{(r)}_{r} = 1/L_{rr}^{2}$.

Assumption \ref{airrepre} (along with the $\beta_{\min}$ condition) is essentially a necessary and sufficient condition 
for support recovery of lasso-type methods
\citep[see, e.g.,][]{zhao2006model, meinshausen2006high, wainwright2009sharp, Van2009On, ravikumar2011high}.
The constant $\alpha \in (0, 1]$ is usually referred to as the irrepresentable (incoherence) constant \citep{wainwright2009sharp}.
Intuitively, the irrepresentable condition requires low correlations between signal and noise predictors, and thus
a value of $\alpha$ that is close to 1 implies that recovering the support is easier to achieve.
The constant $6\pi^{-2}$ is determined by the choice of weight
\eqref{est:generalweight} and can be eliminated by absorbing its reciprocal into the
definition of the weights $w_{\ell m}$.
Doing so, one finds that our irrepresentable condition is essentially
the same as the one 
found in the regression setting \citep{wainwright2009sharp} despite
the fact that our goal is estimating a precision matrix.

Assumption \ref{abddsval} is a bounded singular value condition.
Recalling that $\Omega = L^TL$,
\begin{align}
  0 < \kappa^{-2} \leq \sigma_{\min} \left( \Sigma \right) \leq 
  \sigma_{\max} \left( \Sigma \right) \leq \kappa^{2},
  \label{theory:reseigenvalue}
\end{align}
which is equivalent to the commonly used bounded eigenvalue condition in 
other literatures. 

\subsection{Row-Specific Results} \label{subsec:row}
We start by analyzing support recovery properties of our estimator for each row,
i.e.,
the solution to the subproblem
\eqref{est:subproblem}.
For $r>n$, the Hessian of the negative log-likelihood is not positive
definite, meaning that the objective function may not be strictly
convex in $\beta$ and 
the solution not necessarily unique. Intuitively, 
if the tuning parameter $\lambda$ is large, the resulting 
row estimate $\hat{L}_{r\cdot}$ is sparse and thus includes most variation in a
small subset of the $r$ variables. More specifically, for large $\lambda$,
$\hat{\I}_r \subseteq \I_r$ and thus by Assumption \ref{agaussian},
$\mathbf{X}_{\hat{\I}_r}$ has full rank, 
which implies that $\hat{L}_{r\cdot}$ is unique. 
The series of
technical lemmas in Appendix \ref{online:sparserow} 
precisely characterizes the solution.

The first part of the theorem below shows that with an appropriately chosen
tuning parameter $\lambda$ the solution to 
\eqref{est:subproblem} is sparse enough to be unique and that we will
not over-estimate the true bandwidth. 
Knowing that the support of the unique row estimator $\hat L_{r\cdot}$
is contained in the true support reduces the dimension of the parameter space, and thus
leads to a reasonable error bound.
Of course, if our goal were simply to establish the uniqueness of $\hat
L_{r\cdot}$ and that $\hat{K}_r \leq K_r$, we could trivially take $\lambda = \infty$ (resulting
in $\hat{K}_r=0$). 
The latter part of the theorem thus goes on to provide a choice of 
$\lambda$ that is sufficiently small to guarantee that $\hat{K}_r=K_r$
(and, furthermore, that the signs of all non-zeros are correctly recovered).

\begin{theorem}
  \label{thm:row}
  Consider the family of tuning parameters 
  \begin{align}
    \lambda = \frac{8}{\alpha}\sqrt{\frac{\theta_r \log r}{n}}
    \label{theory:lambdar}
  \end{align}
  and weights given by \eqref{est:generalweight}.
  Under Assumptions \ref{agaussian}--\ref{abddsval}, if the tuple $\left( n, J_r, K_r \right)$ satisfies
  \begin{align}
    n > \alpha^{-2}
    \left( 3\pi^2 K_r + 8 \right) \theta_r \kappa^2\log J_r,
    \label{theory:samplesizer}
  \end{align}
  then with probability 
  greater than $1- c_1 \exp\left\{ -c_2 
  \min( K_r, \log J_r )\right\} - 
  7 \exp\left( -c_3 n \right)$ for some constants $c_1,c_2,c_3$ independent of 
  $n$ and $J_r$, the following properties hold:
  \begin{enumerate}
    \item The row problem \eqref{est:subproblem} has a unique solution 
      $\hat{L}_{r\cdot}$ and
      $\hat{K}_r \leq K_r$.
    \item The estimate $\hat{L}_{r\cdot}$ satisfies the element-wise $\ell_\infty$
      bound,
      \begin{align}
        \norm{\hat{L}_{r\cdot} - L_{r\cdot}}_\infty \leq \lambda\left(4 
        \vertiii{\left(\Sigma_{\mathcal{I}_r\mathcal{I}_r}\right)^{-1}}_\infty 
        + 5\kappa^2 \right).
        \label{theory:infinityboundrow}
      \end{align}
    \item If in addition,
      \begin{align}
        \min_{j \geq J_r+1}\left|L_{rj}\right| > \lambda\left(4 
        \vertiii{\left(\Sigma_{\mathcal{I}_r\mathcal{I}_r}\right)^{-1}}_\infty 
        + 5\kappa^2 \right),
        \label{theory:betamin}
      \end{align}
      then exact signed support recovery holds: For all $j \leq
      r$, $\mathrm{sign} ( \hat{L}_{rj} ) = \mathrm{sign} ( L_{rj} )$.
  \end{enumerate}
\end{theorem}
\begin{proof}
  See Appendix \ref{online:proofofrow}.
\end{proof}
In the classical setting where the ambient dimension $r$ is fixed and the
sample size $n$ is allowed to go to infinity, $\lambda \rightarrow 0$ and
the above scaling requirement is satisfied. By \eqref{theory:infinityboundrow}
the row estimator $\hat{L}_{r\cdot}$ is consistent as is the classical maximum
likelihood estimator. Moreover, it recovers the true support since 
\eqref{theory:betamin} holds automatically.
In high-dimensional scaling, however, both $n$ and $r$ are allowed to change,
and we are interested in the case where $r$ can grow much faster than $n$. 
Theorem \ref{thm:row} shows that, if 
$\vertiii{\left(\Sigma_{\mathcal{I}_r\mathcal{I}_r}\right)^{-1}}_\infty =
\mathcal{O}(1)$ and if $n$ can grow as fast as $K_r\log J_r$, 
then the row estimator $\hat{L}_{r\cdot}$ still recovers the exact support
of $L_{r\cdot}$ when the signal is at least 
$\mathcal O (\sqrt{\frac{\log r}{n}})$ in size, 
and the estimation error $\max_j |\hat{L}_{rj} - L_{rj}|$ is
$\mathcal{O}(\sqrt{\frac{\log r}{n}})$.
Intuitively, for the row estimator to detect 
the true support, we require that the true signal be sufficiently large. The
condition \eqref{theory:betamin} imposes limitations on how fast the 
signal is allowed to decay, 
which is the analogue to the commonly known
``$\beta_{\min}$ condition'' that is assumed for establishing support recovery of the lasso.

\begin{remark} \label{rem:adaptivity}
  Both the choice of tuning parameter \eqref{theory:lambdar} and the
  error bound \eqref{theory:infinityboundrow} depend on the true covariance matrix
  via $\theta_r$. This quantity can be bounded by $\kappa^2$ as
  in \eqref{theory:reseigenvalue} using the fact that 
  $\left( \Sigma_{\mathcal{I}_r\mathcal{I}_r} \right)^{-1}$ is positive definite:
  \begin{align}
    \theta_r = \max_{\ell \in \I_r^c} \theta_r^{(\ell)} = 
    \max_{\ell \in \I_r^c}
    \left\{ \Sigma_{\ell\ell} - \Sigma_{\ell \mathcal{I}_r}
    \left( \Sigma_{\mathcal{I}_r\mathcal{I}_r} \right)^{-1} 
    \Sigma_{\mathcal{I}_r \ell} \right\}\leq
    \max_{\ell \in \I_r^c} \Sigma_{\ell\ell} \leq \kappa^2.
    \nonumber
  \end{align}
  The proof of Theorem \ref{thm:row} shows that the results in this theorem
  still hold true if we replace $\theta_r$ by $\kappa^2$. 
  This observation leads to the fact that we can select a
  tuning parameter having the properties of the theorem that does not
  depend on the unknown sparsity level $K_r$.
  Therefore, our estimator is
  adaptive to the underlying unknown bandwidths.
\end{remark}

\subsubsection{Connections to the regression setting} \label{subsubsec:regression}
In \eqref{eq:regression} we showed that estimation of the $r$-th row of $L$ can be interpreted as a
regression of $X_r$ on its predecessors.
It is thus very interesting to compare Theorem \ref{thm:row} to the standard high-dimensional regression results.
Consider the following linear model of a vector $\mathbf{y} \in \real^n$ of the form
\begin{align}
  \mathbf{y} = \mathbf{Z} \eta + \mathbf{\omega} \quad \quad \quad \omega \sim N(\mathbf{0}, \sigma^2 I_n)
  \label{est:linearmodel}
\end{align}
where $\eta \in \real^p$ is the unknown but fixed parameter to estimate, $\mathbf{Z} \in \real^{n\times p}$ is the design matrix
with each row an observation of $p$ predictors, $\sigma^2$ is the variance of the zero-mean additive noise $\mathbf{\omega}$.
A standard approach in the high-dimensional setting where $p \gg n$ is the lasso \citep{tibshirani1996regression}, which solves the
convex optimization problem,
\begin{align}
  \min_{\eta \in \real^p} \frac{1}{2n} \norm{\mathbf{y} - \mathbf{Z} \eta}_2^2 + \lambda \norm{\eta}_1,
  \label{est:lasso}
\end{align}
where $\lambda > 0$ is a regularization parameter.
In the setting where $\eta$ is assumed to be sparse, the lasso solution is known to be able to successfully recover
the signed support of the true $\eta$ with high probability when $\lambda$ is of the scale $\sigma\sqrt{\frac{\log p}{n}}$
and certain technical conditions are satisfied \citep{wainwright2009sharp}.

Despite the added complications of working with the $\log$ term in the
objective of \eqref{est:subproblem},
Theorem \ref{thm:row} gives a clear indication that, in terms
of difficulty of support recovery, the row estimate problem
\eqref{est:subproblem} is essentially the same as a lasso problem with random design, i.e., with each row $z_i \sim N(\mathbf{0}, \Sigma)$ \citep[Theorem 3,][]{wainwright2009sharp}. 
Indeed, a comparison shows that the two irrepresentable conditions are equivalent.
Moreover, $\theta_r$ plays the 
same role as \citet{wainwright2009sharp}'s $\max_{i} \left(\Sigma_{S^cS^c} - \Sigma_{S^cS}\left(\Sigma_{SS}\right)^{-1}\Sigma_{SS^c}\right)_{ii}$, a threshold constant of the conditional covariance, where $S$ is the support of the true $\eta$.

\citet{st2010l1} introduce an alternative approach to the lasso, in the context of penalized mixture regression models,
that solves the optimization problem,
\begin{align}
  (\hat{\phi}, \hat{\rho}) = \argmin_{\phi, \rho} \left\{ -2\log \rho + \frac{1}{n} \norm{\rho \mathbf{y} + \mathbf{Z} \phi}_2^2 + \lambda \norm{\phi}_1 \right\},
  \label{est:repa}
\end{align}
where $\hat{\sigma} = \hat{\rho}^{-1}$ and $\hat{\eta} = -\hat{\phi} / \hat{\rho}$.
Note that \eqref{est:repa} basically coincides with \eqref{est:subproblem} except for the penalty.

In \citet{st2010l1}, the authors study the asymptotic and non-asymptotic 
properties of the $\ell_1$-penalized estimator for the general mixture regression models where the loss functions are non-convex.
The theoretical properties of \eqref{est:repa} are studied in \citet{sun2010comments}, which partly motivates the scaled lasso \citep{Sun25092012}.

The theoretical work of \citet{sun2010comments} differs from ours both in that they study the $\ell_1$ penalty (instead of the hierarchical group lasso) and in their assumptions.
The nature of our problem requires the sample matrix to be random (as in \ref{agaussian}), while
\citet{sun2010comments} considers the fixed design setting, which does not apply in our context.
Moreover, they provide prediction consistency and a deviation bound of the regression parameters estimation in $\ell_1$ norm.
We give exact signed support recovery results for the regression parameters as well as estimation deviation bounds in various norm criteria.
Also, they take an asymptotic point of view while we give finite sample results.

\subsection{Matrix Bandwidth Recovery Result} \label{subsec:matrix}
With the properties of the row estimators in place, we are ready to
state results about estimation of the matrix $L$.  
The following
theorem gives an analogue to Theorem \ref{thm:row} in the matrix setting. Under 
similar conditions, with one particular choice of tuning parameter, the estimator
recovers the true bandwidth for all rows adaptively with high probability.
\begin{theorem}
  \label{thm:matrix}
  Let $\theta = \max_{r} \theta_r$ and
  $K = \max_{r} K_r$, and take
%  consider the tuning parameter 
  \begin{align}
    \lambda = \frac{8}{\alpha} \sqrt{\frac{2\theta \log p}{n}} 
    \label{theory:lambdaL}
  \end{align}
  and weights given by \eqref{est:generalweight}.
  Under Assumptions \ref{agaussian}--\ref{abddsval}, if $\left( n, p, K \right)$ satisfies
  \begin{align}
    n> \alpha^{-2}
    \theta \kappa^2\left( 12 \pi^2 K + 32 \right)\log p ,
    \label{theory:samplesizeL}
  \end{align}
  then with probability greater than $1- cp^{-1}$
  for some constant $c$ independent of $n$ and $p$, the following properties hold:
  \begin{enumerate}
    \item The estimator $\hat{L}$ is unique,
      and it is at least as sparse as $L$, i.e.,  $\hat{K}_r \leq K_r$ for all $r$.
    \item The estimator $\hat{L}$ satisfies the element-wise $\ell_\infty$
      bound,
      \begin{align}
        \norm{\hat{L} - L}_\infty \leq \lambda\left(4 \max_r 
        \vertiii{\left(\Sigma_{\mathcal{I}_r\mathcal{I}_r}\right)^{-1}}_\infty 
        + 5 \kappa^2 \right).
        \label{theory:infinityboundL}
      \end{align}
    \item If in addition, 
      \begin{align}
        \min_{r} \min_{j \geq J_r+1}\left|L_{rj}\right| > \lambda\left(4 \max_r
        \vertiii{\left(\Sigma_{\mathcal{I}_r\mathcal{I}_r}\right)^{-1}}_\infty 
        + 5 \kappa^2 \right),
        \label{theory:betaminL}
      \end{align}
      then exact signed support recovery holds: $\mathrm{sign}( \hat{L}_{rj}) = 
      \mathrm{sign} ( L_{rj} )$ for all $r$ and $j$.
  \end{enumerate}
\end{theorem}
\begin{proof}
  See Appendix \ref{online:proofofmatrix}.
\end{proof}
As discussed in Remark \ref{rem:adaptivity}, we can replace $\theta$ with its upper
bound $\kappa^2$, and the results remain true. This theorem shows that
one can properly estimate
the sparsity pattern across all rows exactly using only one tuning parameter chosen
without any prior knowledge of the true bandwidths. 
In Section \ref{subsubsec:regression}, we noted that
the conditions required for support recovery and
the element-wise $\ell_\infty$ error bound for estimating a row of $L$ is
similar to those of the lasso in the regression
setting. A union bound argument allows us to translate this into exact
bandwidth recovery in the matrix setting and to derive a reasonable convergence rate 
under conditions as mild as that of a lasso problem
with random design. This technique is similar in spirit to
neighborhood selection \citep{meinshausen2006high}, 
though our approach is likelihood-based.

Comparing \eqref{theory:samplesizeL} to \eqref{theory:samplesizer}, 
we see that the sample size
requirement for recovering $L$ is determined by the least sparse row.
While intuitively one would expect the matrix problem to be harder than any single row problem,
we see that in fact the
two problems are basically of the same difficulty (up to a multiplicative constant).

In the setting where variables exhibit a natural ordering, \citet{shojaie2010penalized} proposed a penalized
likelihood framework like ours to estimate the structure of directed acyclic graphs (DAGs). Their
method focuses on variables which are standardized to have unit variance.
In this special case, penalized likelihood does not involve the log-determinant term and under similar assumptions to ours, they
proved support recovery consistency. However, they use lasso and adaptive lasso \citep{zou2006adaptive} penalties,
which do not have the built-in notion of local dependence. Since these $\ell_1$-type penalties do not induce structured sparsity in the Cholesky factor, the resulting
precision matrix estimate is not necessarily sparse. 
By contrast, our method does not assume unit variances and learns an adaptively banded structure for $\hat L$ that leads 
to a sparse $\hat\Omega$ (thereby encoding conditional dependencies).

To study the difference between the ordered and non-ordered problems,
we compare our method
with \cite{ravikumar2011high}, who studied the graphical lasso estimator 
in a general setting where variables are not necessarily ordered. 
Let $\mathcal S$ index the edges of the graph specified by the sparsity pattern of $\Omega = \Sigma^{-1}$. 
The sparsity recovery result and convergence rate are established under an
irrepresentable condition imposed on $\Gamma = \Sigma \otimes \Sigma \in \real^{p^2 \times p^2}$:
\begin{align}
  \max_{e \in \mathcal S^{c}} \norm{\Gamma_{e \mathcal S} 
  \left( \Gamma_{\mathcal S \mathcal S} \right)^{-1}}_1 \leq \left(
  1-\alpha \right)
  \label{eq:glasso-irrep}
\end{align}
for some $\alpha \in (0,1]$. Our Assumption \ref{airrepre} is on each
variable through the entries of the true covariance $\Sigma$ while
\eqref{eq:glasso-irrep} imposes  
such a condition on the edge variables 
$Y_{(j,k)} = X_j X_k - \E\left( X_j X_k \right)$, resulting in a vector $\ell_1$-norm 
restriction on a much larger matrix $\Gamma$, which can be 
more restrictive for large $p$. 
More specifically, condition \eqref{eq:glasso-irrep} 
arises in \citet{ravikumar2011high} %is mainly introduced 
to tackle the analysis of 
the $\log \det \Omega$ term in the graphical lasso problem.
By contrast, in our setting the parameterization in terms of $L$ means
that the $\log\det$ term is simply a sum of $\log$ terms %comes down to
on diagonal 
elements and is thus easier to deal with, leading to the milder irrepresentable
assumption. 
Another difference is 
that they require the sample size 
$n > c \kappa_{\Gamma}^2 d^2 \log p$
for some constant $c$. The quantity $d$ measures the maximum 
number of non-zero elements in each row of the true $\Sigma$, 
which in our case is $2K+1$, and 
$\kappa_{\Gamma} = \vertiii{\left(\Gamma_{\mathcal S \mathcal S} \right)^{-1}}_\infty$
can be much larger than $\kappa^2$. Thus, comparing to 
\eqref{theory:samplesizeL}, one finds that their sample size requirement is much more restrictive.
A similar comparison could also be made with the lasso penalized D-trace estimator \citep{zhang2014sparse}, whose
irrepresentable condition involves $\Gamma = (\Sigma \otimes I + I \otimes \Sigma)/2 \in \real^{p^2 \times p^2}$.
Of course, the results in both \citet{ravikumar2011high} and \citet{zhang2014sparse}
apply to estimators invariant to permutation of variables;
additionally, the random vector only needs to satisfy an exponential-type tail condition. 

\subsection{Precision Matrix Estimation Consistency} \label{subsec:consistency}
Although our primary target of interest is $L$, the parameterization $\Omega = L^TL$
makes it natural for us to try to connect our results of estimating $L$ with the vast literature in directly estimating $\Omega$,
which is the standard estimation target when the known ordering is not available.
In this section, we consider the estimation consistency of $\Omega$ using
the results we obtained for $L$.
The following theorem gives results of
how well $\hat{\Omega} = \hat{L}^T\hat{L}$
performs in estimating the true precision matrix $\Omega=L^TL$ 
in terms of various matrix norm criteria. 
\begin{theorem} \label{thm:Omegabounds}
  Let $\theta = \max_r \theta_r$, $K = \max_r K_r$ and $s = \sum_r K_r$ 
  denote the total number of non-zero off-diagonal elements in $L$.
  Define $\zeta_\Sigma = \frac{8\sqrt{2\theta}}{\alpha}\left( 4 \max_r \vertiii{\left( \Sigma_{\mathcal I_r \mathcal I_r} \right)^{-1}}_\infty + 5\kappa^2\right)$.
  Under the assumptions in Theorem \ref{thm:matrix},
  the following deviation bounds hold with probability greater than $1 - cp^{-1}$ for some constant $c$ independent of $n$ and $p$:
  \begin{align}
    \norm{\hat{\Omega} - \Omega}_\infty &\leq
    2\zeta_\Sigma \vertiii{L}_\infty \sqrt{\frac{\log p}{n}} + \zeta_\Sigma^2 \left( K + 1 \right) \frac{\log p}{n},
    \nonumber\\
    \vertiii{\hat{\Omega} - \Omega}_\infty &\leq
    2\zeta_\Sigma \vertiii{L}_\infty \left( K + 1 \right) \sqrt{\frac{\log p}{n}} + \zeta_\Sigma^2\left( K + 1 \right)^2  \frac{\log p}{n},
    \nonumber\\
    \vertiii{\hat{\Omega} - \Omega}_2 &\leq
    2 \zeta_\Sigma \vertiii{L}_\infty \left( K + 1 \right) \sqrt{\frac{\log p}{n}} + \zeta_\Sigma^2 \left( K + 1 \right)^2 \frac{\log p}{n},
    \nonumber\\
    \norm{\hat{\Omega} - \Omega}_F &\leq
    2 \kappa \zeta_\Sigma \sqrt{\frac{\left( s + p \right)\log p}{n}} + \zeta_\Sigma^2 \left( K + 1 \right)\sqrt{s + p} \frac{\log p}{n}.
    \nonumber
  \end{align}
\end{theorem}
When the quantities $\zeta_\Gamma$, $\vertiii{L}_\infty$, and $\kappa$ are treated as constants, 
these bounds can be summarized more succinctly as follows:
\begin{proof}
  See Appendix \ref{online:proofofrates}.
\end{proof}

\begin{corollary} \label{cor:Omegabounds}
  Using the notation and conditions in Theorem \ref{thm:Omegabounds},
  if $\zeta_\Gamma$, $\vertiii{L}_\infty$, and $\kappa$ remain constant, then the scaling $(K + 1)^2 \log p = o(n)$ is
  sufficient to guarantee the following estimation error bounds:
  \begin{align}
    \norm{\hat{\Omega} - \Omega}_\infty &=
    \mathcal O_P \left(\sqrt{\frac{\log p}{n}} \right),
    \nonumber\\
    \vertiii{\hat{\Omega} - \Omega}_\infty &=
    \mathcal O_P\left( (K + 1) \sqrt{\frac{\log p}{n}} \right),
    \nonumber\\
    \vertiii{\hat{\Omega} - \Omega}_2 &=
    \mathcal O_P\left( (K + 1) \sqrt{\frac{\log p}{n}} \right),
    \nonumber\\
    \norm{\hat{\Omega} - \Omega}_F &=
    \mathcal O_P\left(
    \sqrt{\frac{	\left( s+p \right) \log p}{n}} \right).
    \nonumber
  \end{align}
\end{corollary}
The conditions for these deviation bounds to hold are those required for support recovery as in Theorem \ref{thm:matrix}.
In many cases where estimation consistency is more of interest than support recovery, we can still deliver 
the desired error rate in Frobenius norm, matching the rate derived in \cite{rothman2008sparse}.
In particular, we can drop the strong
irrepresentable assumption (\ref{airrepre}) and weaken the Gaussian assumption (\ref{agaussian}) to the following marginal sub-Gaussian
assumption:
\begin{enumerate}[resume,label=\textbf{A\arabic*}]
  \item \textit{Marginal sub-Gaussian assumption}: \label{asubgaussian}
    The sample matrix $\mathbf{X} \in \real^{n\times p}$ has $n$ independent rows 
    with each row drawn from the distribution of a zero-mean random vector $X = (X_1, \cdots, X_p)^T$ with covariance $\Sigma$ and sub-Gaussian
    marginals, i.e., 
    \begin{align}
      \mathrm{E} \exp\left( tX_j/ \sqrt{\Sigma_{jj}} \right) \leq \exp \left( C t^2 \right)
      \nonumber
    \end{align} for all $j = 1,\ldots,p$, $t \geq 0$ and for some constant $C>0$ that does not depend on $j$.
\end{enumerate}
\begin{theorem} \label{thm:Fbound}
  Under Assumption \ref{asparsity}, \ref{abddsval} and \ref{asubgaussian},
  with tuning parameter $\lambda$ of scale $\sqrt{\frac{\log p}{n}}$ and weights as in \eqref{est:generalweight}, 
  the scaling $\left( s + p \right) \log p = o (n)$ is sufficient for the following estimation error bounds in Frobenius norm
  to hold:
  \begin{align}
    \norm{\hat{L} - L}_F &=
    \mathcal O_P\left(
    \sqrt{\frac{	\left( s+p \right) \log p}{n}} \right), \nonumber\\
    \norm{\hat{\Omega} - \Omega}_F &=
    \mathcal O_P\left(
    \sqrt{\frac{\left( s+p \right) \log p}{n}} \right).\nonumber
  \end{align}
\end{theorem}
\begin{proof}
  See Appendix \ref{proof:Fbound}.
\end{proof}

The rates in Corollary \ref{cor:Omegabounds} (and Theorem \ref{thm:Fbound}) essentially match the rates obtained in methods that directly
estimate $\Omega$ (e.g., the graphical lasso estimator, studied in \citealt{rothman2008sparse}, \citealt{ravikumar2011high}, and the column-by-column
methods as in \citealt{cai2011constrained}, \citealt{liu2012tiger}, and \citealt{sun2013sparse}).
However, the exact comparison in rates with these methods is not straightforward.
First, the targets of interest are different. In the setting where the variables have a known ordering,
we are more interested in the structural information among variables that is expressed in $L$, and thus 
accurate estimation of $L$ is more important.
When such ordering is not available as considered in \citet{rothman2008sparse, cai2011constrained, liu2012tiger} and so on, 
however, the conditional dependence structure encoded by the sparsity pattern
in $\Omega$ is more of interest, and the accuracy of directly estimating $\Omega$ is the focus.
Moreover, deviation bounds of different methods are built upon assumptions that treat different quantities as constants.
Quantities that are assumed to remain constant in the analysis of one method might actually be allowed to scale with ambient dimension in a 
nontrivial manner in another method, which makes direct rate comparison among different methods complicated and less illuminating.

Our analysis can be extended
to the unweighted version
of our estimator, i.e., with weight $w_{\ell m} = 1$, but under more
restrictive conditions and with slower rates of convergence. 
Specifically, Assumption \ref{airrepre} becomes
$\max_{\ell \in \mathcal{I}^{c}_r } 
\norm{\Sigma_{\ell\I_r} \left( 
\Sigma_{\I_r\I_r}\right)^{-1}}_1 
\leq \left( 1-\alpha \right)/K_r$ for each $r = 2,\ldots,p$. With the same tuning parameter choice \eqref{theory:lambdar} 
and \eqref{theory:lambdaL}, the terms of $K_r$ and $K$ in sample size requirements \eqref{theory:samplesizer} and \eqref{theory:samplesizeL}
are replaced with $K_r^2$ and $K^2$, respectively. The estimation error bounds in all norms are multiplied by an extra factor of $K$.
All of the above indicates that in highly sparse situations (in which $K$ is very small), the
unweighted estimator has very similar theoretical performance to the weighted estimator.

\section{Simulation Study} \label{sec:simulation}
In this section we study the empirical performance of our estimators 
(both with weights as in \eqref{est:generalweight} and with no weights,
i.e., $w_{\ell m}=1$) on simulated data.
For comparison, we include
two other sparse precision matrix estimators designed for 
the ordered-variable case:
\begin{itemize}
  \item \textbf{Non-Adaptive Banding}~\citep{bickel2008regularized}:
    This method estimates $L$ as a
    lower-triangular matrix with a fixed bandwidth $K$ applying across all rows. The regularization parameter used in this method
    is the fixed bandwidth $K$.
  \item \textbf{Nested Lasso}~\citep{levina2008sparse}: This method yields an adaptive banded structure 
    by solving a set of penalized least-squares problems (both the loss function and the nested-lasso penalty are non-convex).
    The regularization parameter controls the amount of penalty and thus the sparsity level of the resulting estimate.
\end{itemize}

All simulations are run at a sample size of $n = 100$, where each sample is drawn independently from the $p$-dimensional normal distribution
$N(\mathbf{0}, (L^TL)^{-1})$.
We compare the performance of our estimators with the methods above both in terms of support recovery (in Section \ref{subsec:supportrecovery})
and in terms of how well $\hat L$ estimates $L$ (in Section \ref{subsec:dev}).
For support recovery, we consider $p = 200$ and for estimation accuracy, we consider $p = 50, 100, 200$, which corresponds to settings
where $p < n$, $p = n$, and $p > n$, respectively.

We simulate under the following models for $L$. We adapt the parameterization $L = D^{-1}T$ as in \citet{2016arXiv161002436K},
where $D$ is a diagonal matrix with diagonal elements drawn randomly from a uniform distribution on the interval $[2, 5]$,
and $T$ is a lower-triangular matrix with ones on its diagonal and off-diagonal elements defined as follows:
\begin{itemize}
  \item \textbf{Model 1}: \label{sim:model1}
    Model 1 is at one extreme of bandedness of the Cholesky factor $L$, in which we take the lower triangular matrix $L \in \real^{p \times p}$ 
    to have a strictly banded structure, with each row having the same bandwidth 
    $K_r = K = 1$ for all $r$. Specifically, we take $T_{r,r} = 1$, $T_{r,r-1} = 0.8$ and $T_{r,j} = 0$ for $j<r-1$.

  \item \textbf{Model 2}: \label{sim:model2}
    Model 2 is at the other extreme, in which we allow $K_r$ to vary with $r$. We take $T$ to be a block diagonal matrix with
    5 blocks, each of size $p / 5$. Within each block, with probability 0.5 each row $r$ is assigned with a non-zero bandwidth
    that is randomly drawn from a uniform distribution on $\left\{ 1,\ldots, r-1 \right\}$ (for $r > 1$). 
    Each non-zero element in $T$ is then drawn independently
    from a uniform distribution on the interval $[0.1, 0.4]$, and is assigned with a positive/negative sign with probability 0.5.

  \item \textbf{Model 3}: \label{sim:model3}
    Model 3 is a denser and thus more challenging version of Model 2, with $T$ a block diagonal matrix with only 2 blocks. 
    Each of the blocks is of size $p/2$ but is otherwise generated as in Model 2.

  \item \textbf{Model 4}: \label{sim:model4}
    Model 4 is a dense block diagonal model. The matrix $T$ has a completely dense lower-triangular block from the $p/4$-th row to the $3p/4$-th row
    and is zero everywhere else.
    Within this block, all off-diagonal elements are drawn uniformly from $[0.1, 0.2]$, and positive/negative signs are then assigned with probability 0.5.
\end{itemize}

Model 1 is a stationary autoregressive model of order 1. By the regression interpretation \eqref{eq:regression},
for each $r$, it can be verified that the autoregressive polynomial of the $r$-th row of Models 2, 3, and 4
has all roots outside the unit circle, which
characterizes stationary autoregressive models of orders equal to the corresponding row-wise bandwidths.
See Figure \ref{fig:truemodel} for examples of the four sparsity patterns for $p = 100$.
The non-adaptive banding method should benefit from Model 1 while the nested
lasso and our estimators are expected to perform better in the other three models where each row has its own bandwidth.
\begin{figure}
  \centering
  \includegraphics[width=\textwidth]{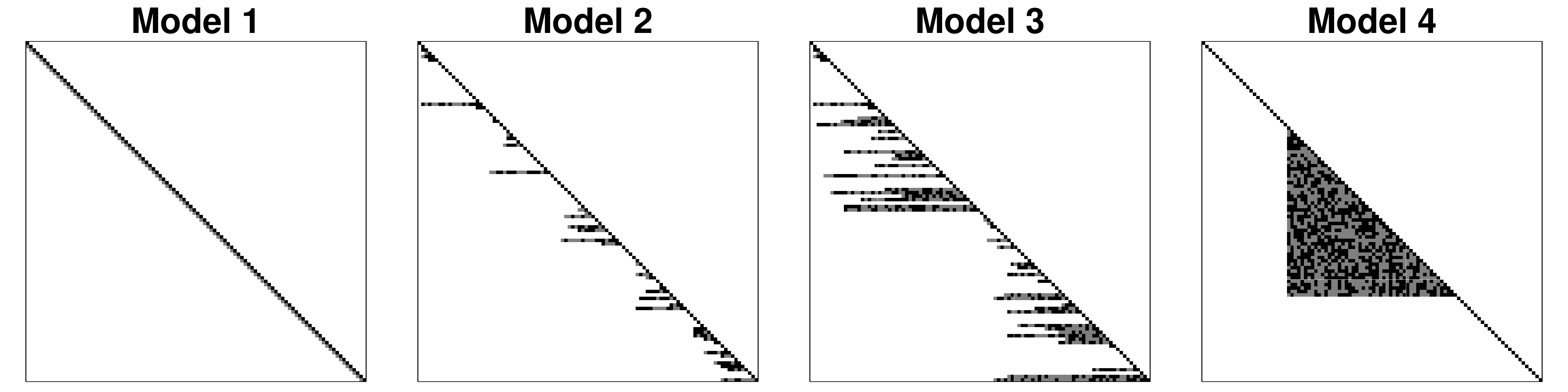}
  \caption{Schematic of four simulation scenarios with $p=100$: 
    (from left to right) Model 1 is strictly banded, Model 2 has small variable bandwidth,
    Model 3 has large variable bandwidth, and Model 4 is block-diagonal.
    Black, gray, and white stand for positive, negative, and zero entries, respectively.
  The proportion of elements that are non-zero is 4\%, 6\%, 15\%, and 26\%, respectively.}
  \label{fig:truemodel}
\end{figure}

For all four models and every value of $p$ considered,
we verified that Assumptions \ref{airrepre} and \ref{abddsval} hold and then
simulated $n = 100$ observations according to each of the four models based on Assumption \ref{agaussian}.

\subsection{Support Recovery} \label{subsec:supportrecovery}
We first study how well the different estimators identify 
zeros in the four models above.
We generate $n=100$ random samples from each model with $p = 200$.
The tuning parameter $\lambda \geq 0$ in 
\eqref{est:estimator} measures the amount of regularization and
determines the sparsity level of the estimator.  We use 100 tuning parameter values for each estimator and 
repeat the simulation 10 times.

Figure \ref{fig:ROC}
shows the sensitivity (fraction of true non-zeros that are correctly recovered) 
and specificity (fraction of true zeros that are correctly set to zero)
of each method parameterized by its tuning parameter (in
the case of non-adaptive banding, the parameter is the bandwidth
itself, ranging from $0$ to $p-1$).
Each set of 10 curves of the same color corresponds to the results of one estimator, 
and each curve within the set corresponds to the result of one draw from 10 simulations.
Curves closer to the upper-right corner indicate better classification
performance (the $x+y=1$ line corresponds to random guessing).

The sparsity level of the non-adaptive banding estimator depends only on the pre-specified bandwidth 
(which is the method's tuning parameter) and not on the data itself.
Consequently, the sensitivity-specificity curves for the non-adaptive banding do not vary across replications when 
simulating from a particular underlying model.  The sparsity levels of the nested lasso and our methods, by contrast,
hinge on the data, thus giving a different curve for each replication.

\begin{figure}[h]
  \centering
  \includegraphics[width=0.8\textwidth]{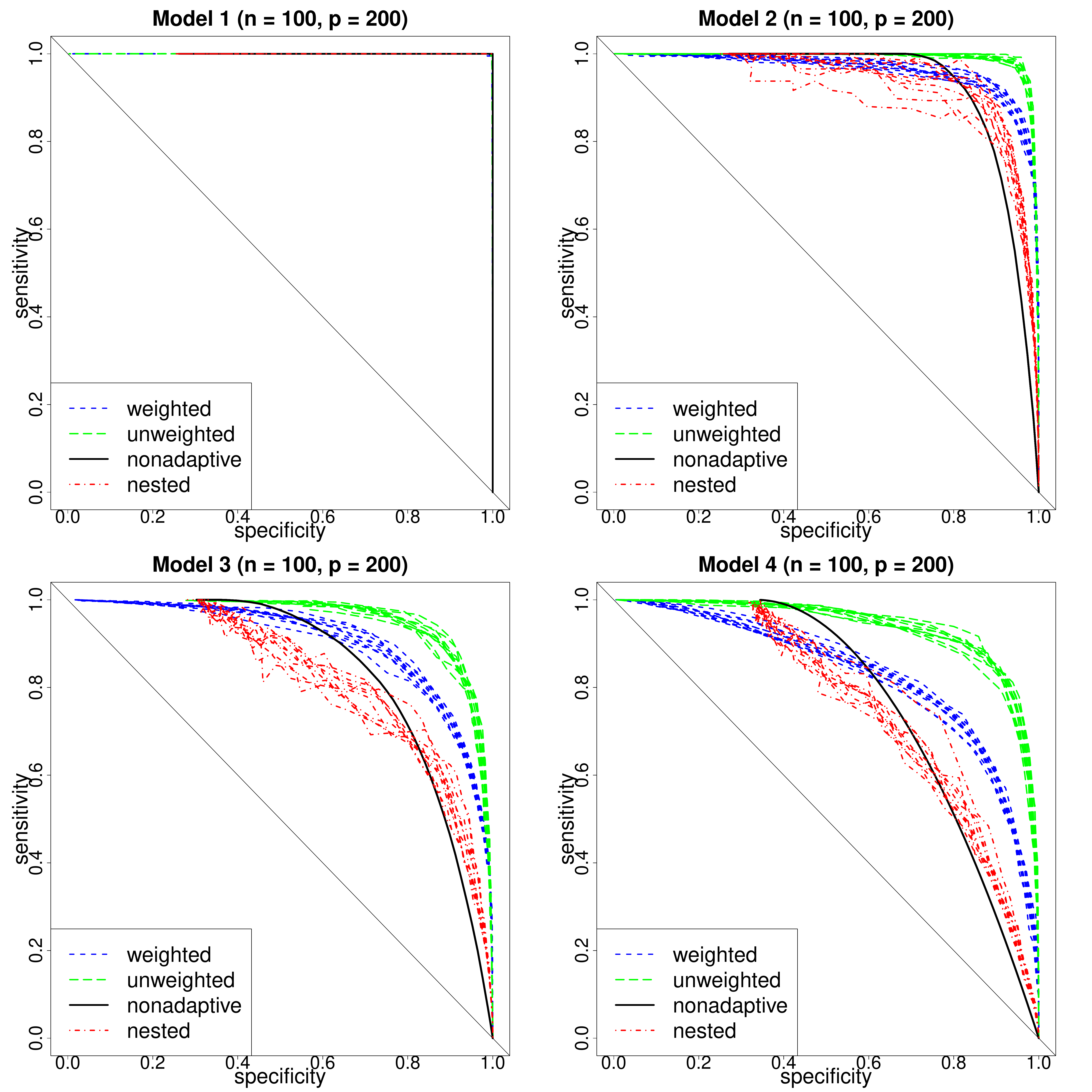}
  \caption{
    ROC curves showing support recovery when the true $L$ (top-left) is strictly banded, (top-right) has small variable bandwidth,
  (bottom-left) has large variable bandwidth, and (bottom-right) is block-diagonal, over 10 replications.}
  \label{fig:ROC}
\end{figure}

In practice, we find that our methods and the nested lasso sometimes
produce entries with very small, but non-zero, absolute values. 
To study support recovery, we set all estimates whose absolute values are
below $10^{-10}$ to zero, both in our estimators and the nested lasso.

In Model 1, we observe that all methods considered attain perfect classification
accuracy for some value of their tuning parameter. While the non-adaptive approach is guaranteed to do so in this scenario,
it is reassuring to see that the more flexible methods can still perfectly recover this sparsity pattern.

In Model 2, we observe that our two methods outperform the
nested lasso, which itself, as expected, outperforms the non-adaptive banding method.
As the model becomes more challenging (from Model 2 to Model 4), the performances of
all four methods start deteriorating. Interestingly, the nested lasso no longer retains its advantage over 
non-adaptive banding in Models 3 and 4, while the performance advantage of our methods become even more substantial.

The fact that the unweighted version of our method outperforms the
weighted version stems from the fact that all models are comparatively sparse for $p = 200$,
and so the heavier penalty on each row delivered by the unweighted approach recovers
the support more easily than the weighted version. 

\subsection{Estimation Accuracy} \label{subsec:dev}
We proceed by comparing the estimators in terms of how far $\hat L$
is from $L$. To this end, we generate $n = 100$ random samples from 
the four models with $p = 50$, $p = 100$, and $p=200$.
Each method is computed with
its tuning parameter selected to maximize the Gaussian 
likelihood on the validation data in a 5-fold cross-validation. 
For comparison, we report the estimation accuracy of each estimate in terms of
the scaled Frobenius norm $\frac{1}{p}\norm{\hat{L} - L}_F^2$, the matrix infinity norm $\vertiii{\hat{L} - L}_\infty$, the spectral norm $\vertiii{\hat{L} - L}_2$,
and the (scaled) Kullback-Leibler loss $\frac1{p}\left[\text{tr}(\Omega^{-1}\hat\Omega)-\log\det(\Omega^{-1}\hat\Omega)-p\right]$ \citep{levina2008sparse}.

The simulation is repeated 50 times, and the results are summarized in Figure~\ref{fig:Model_1} through Figure~\ref{fig:Model_4}.
Each figure corresponds to a model, and consists of a 4-by-3 panel layout. Each row corresponds to an error measure, and each column corresponds to a value of $p$.

\begin{figure}[p]
  \centering
  \includegraphics[width=0.99\linewidth]{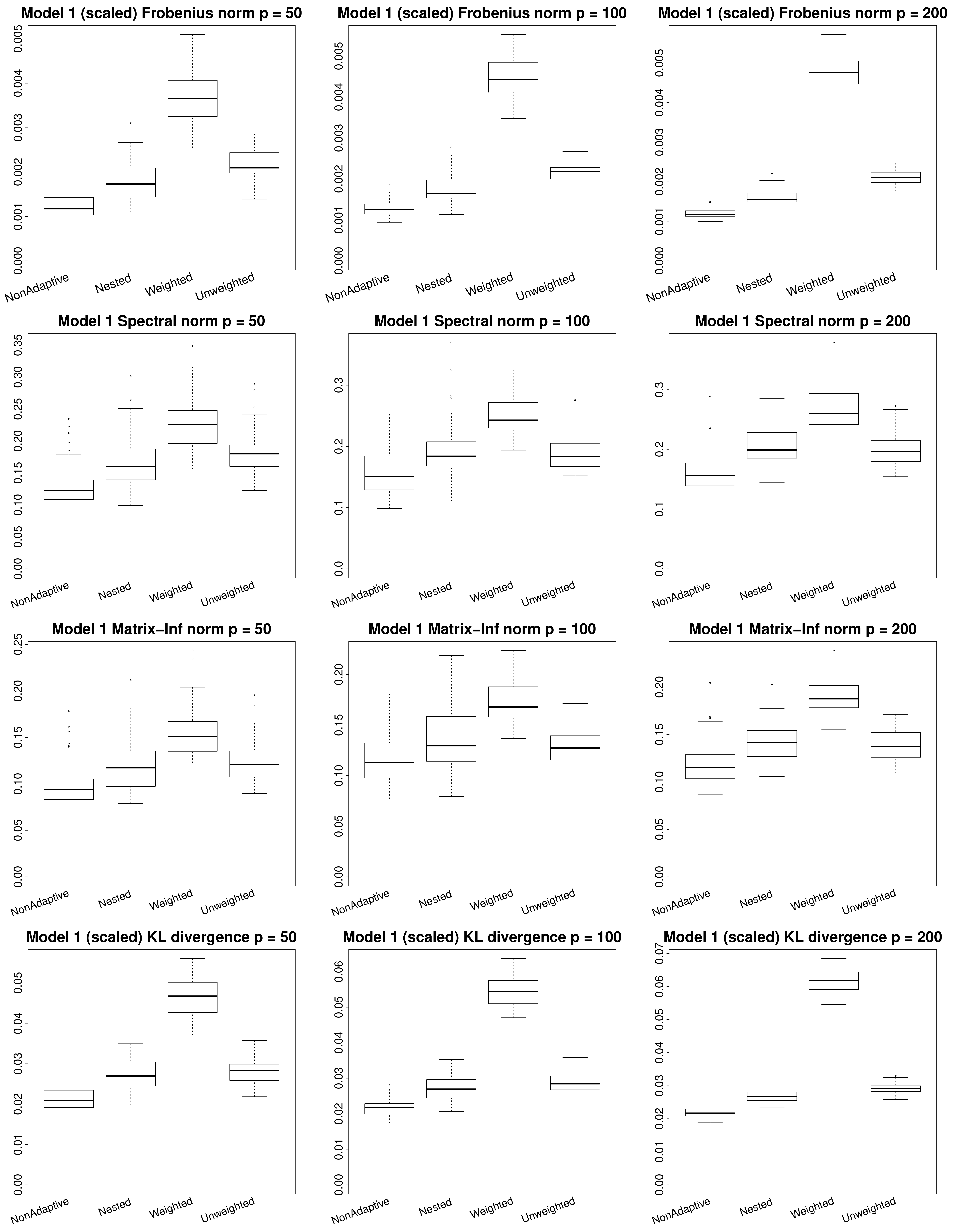}
  \caption{Estimation accuracy when data are generated from Model 1, which is strictly banded.}
  \label{fig:Model_1}
\end{figure}
\begin{figure}[p]
  \centering
  \includegraphics[width=0.99\linewidth]{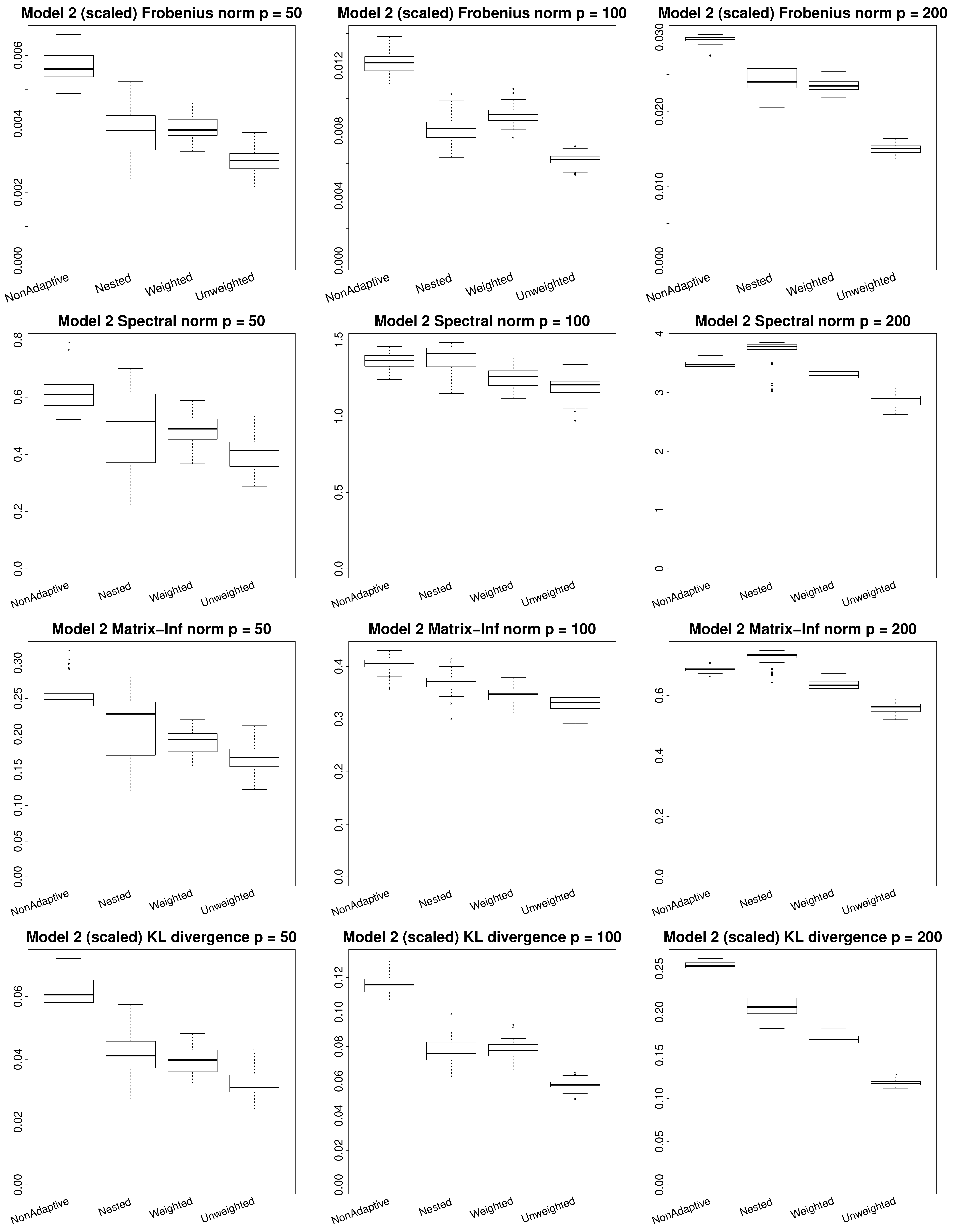}
  \caption{Estimation accuracy when data are generated from Model 2, which has small variable bandwidth.}
  \label{fig:Model_2}
\end{figure}
\begin{figure}[p]
  \centering
  \includegraphics[width=0.99\linewidth]{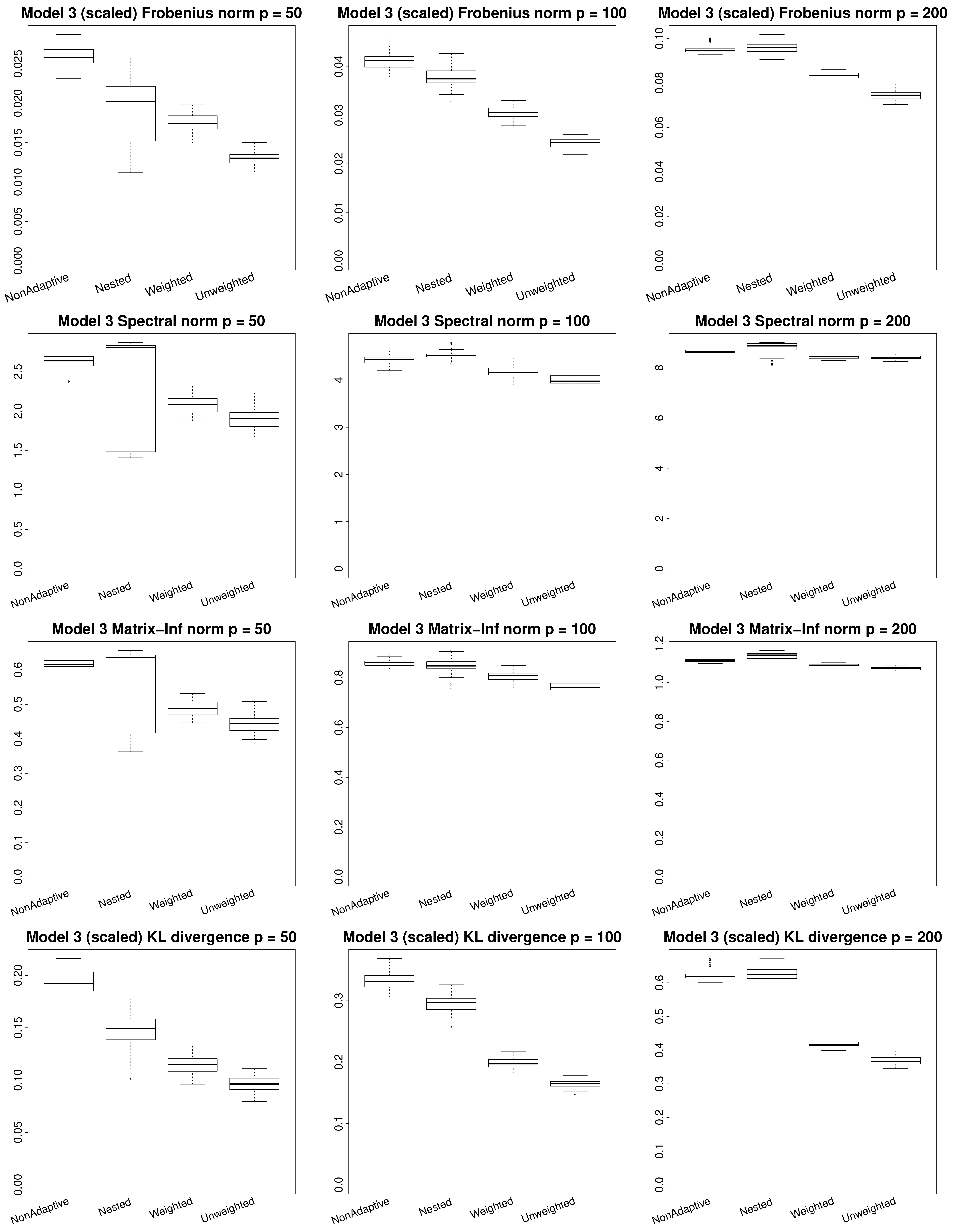}
  \caption{Estimation accuracy when data are generated from Model 3, which has large variable bandwidth.}
  \label{fig:Model_3}
\end{figure}
\begin{figure}[p]
  \centering
  \includegraphics[width=0.99\linewidth]{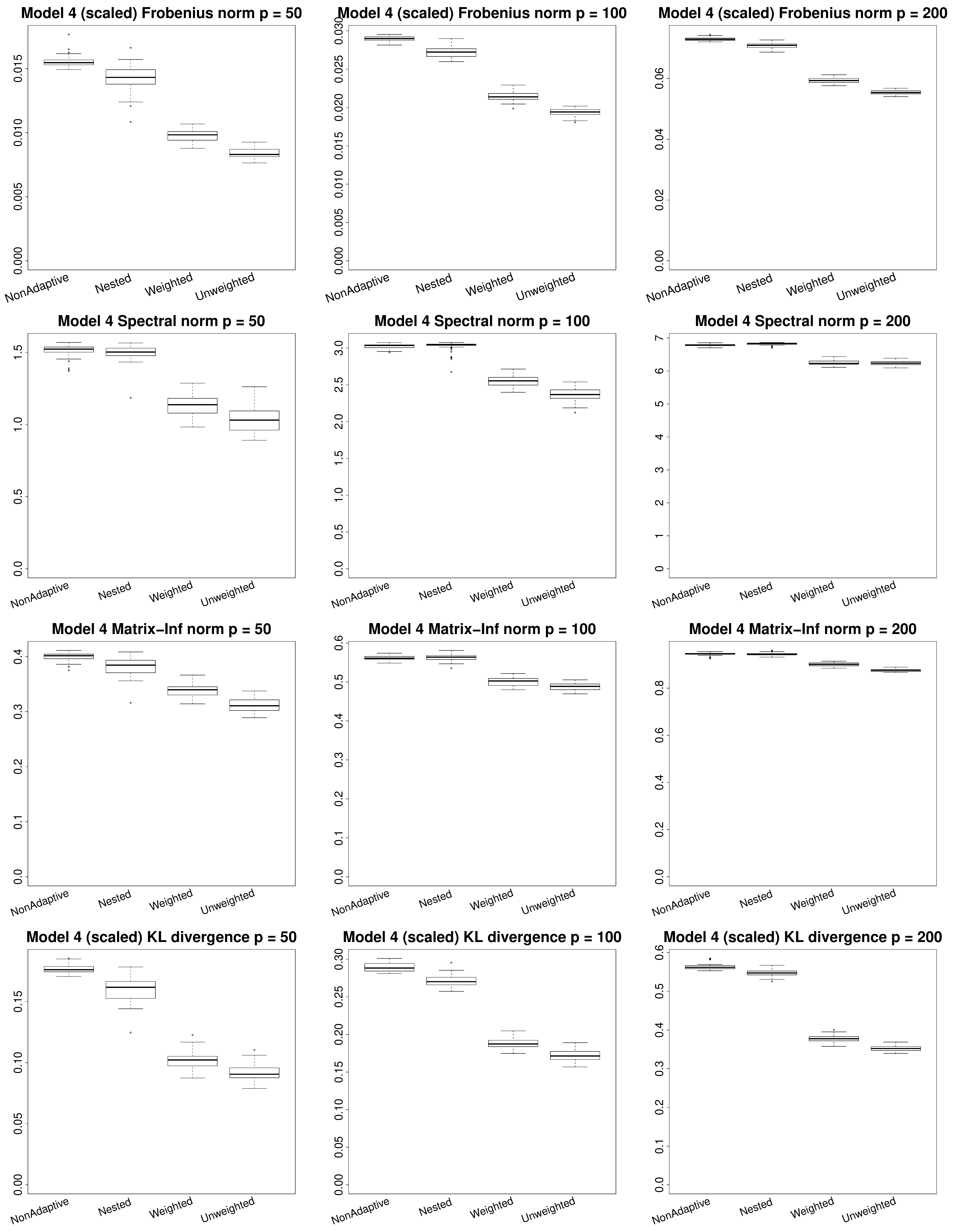}
  \caption{Estimation accuracy when data are generated from Model 4, which is block-diagonal.}
  \label{fig:Model_4}
\end{figure}

As expected, the non-adaptive banding estimator does better than the other
estimators in Model 1.  In Models 2, 3, and 4, where bandwidths vary with row, 
our estimators and the nested lasso outperform non-adaptive banding.

A similar pattern is observed as in support recovery.
As the model becomes more complex and $p$ gets larger, the performance of the nested lasso
degrades and gradually becomes worse than non-adaptive banding. By contrast,
as the estimation problem becomes more difficult, the advantage in performance of our methods
becomes more obvious.

We again observe that the unweighted estimator performs better than the weighted one.
As shown in Section \ref{sec:theory}, 
the overall performance of our method hinges on the underlying model complexity (measured in terms of $\max_r K_r$) as well as 
the relative size of $n$ and $p$. 
When $n$ is relatively small, usually a more constrained method (like the unweighted estimator)
is preferred over a more flexible method (like the weighted estimator). So in our simulation setting, it is reasonable to observe
that the unweighted method works better.
Note that as the underlying $L$ becomes denser (from Model 1 to Model 4), the performance difference between
the weighted and the unweighted estimator diminishes. This corroborates our discussion in the end of Section \ref{sec:theory} that the 
performance of the unweighted estimator becomes worse when the underlying model is dense.

\section{Applications to Data Examples} \label{sec:realdata}
In this section, we illustrate the practical merits of our proposed method by applying it to two data examples.
We start with an application to genomic data where our method can help model the local
correlations along the genome. In Section \ref{sec:phoneme} we compare our method
with other estimators within the context of a sound recording classification problem.
\subsection{An Application to Genomic Data} \label{sec:snp}
We consider an application of our estimator to
modeling correlation along the genome.
Genetic mutations that occur close together on a chromosome are more likely to
be co-inherited than mutations that are located far apart (or on
separate chromosomes).  This leads to local correlations between genetic
variants in a population.  Biologists refer to this local dependence
as {\em linkage disequilibrium} (LD).  The width of this dependence is
known to vary along the genome due to the variable locations of recombination hotspots, 
which suggests that adaptively
banded estimators may be quite suitable in these contexts.

We study HapMap phase 3 data from the International HapMap project
\citep{international2010integrating}.
The data consist of $n = 167$ humans from the YRI (Yoruba in Ibadan,
Nigeria) population,
and we focus on $p = 201$ consecutive tag SNPs 
on chromosome 22 (after filtering out infrequent sites with minor allele frequency $\leq 10 \%$).

While tag SNP data, which take discrete values $\{0, 1, 2\}$, are non-Gaussian, we argue that
our estimator is still sensible to use in this case. First, the parameterization $\Omega = L^T L$ does not depend on the Gaussian assumption.
Moreover the estimator corresponds to minimizing a penalized 
Bregman divergence of the log-determinant function \citep{ravikumar2011high}. 
Furthermore, the least-squares term in \eqref{est:estimator} can be interpreted as minimizing the prediction error in the linear models \eqref{eq:regression}
while the log terms act as log-barrier functions to impose
positive diagonal entries (which ensures that the resulting $\hat{L}$ is a valid Cholesky factor).

To gauge the performance of our estimator on modeling LD,
we randomly split the $167$ samples into training and testing sets of
sizes $84$ and $83$, respectively. Along a path
of tuning parameters with decreasing values, estimators $\hat{L}$ are computed
on the training data.
To evaluate $\hat{L}$ on a vector $\tilde{x}$ from the test data set, we can compute the error in predicting
$\hat{L}_{rr} \tilde{x}_r$ using $- \sum_{k = 1}^{r - 1}\hat{L}_{r,k} \tilde{x}_k$ via \eqref{eq:regression} for each $r$, giving
the error
\begin{align}
  \mathrm{err}(\tilde{x}) = \frac{1}{p - 1} \sum_{r = 2}^p\left( \hat{L}_{rr}\tilde{x}_r +  \sum_{k = 1}^{r - 1} \hat{L}_{rj} \tilde{x}_k \right)^2.
  \label{snp:predictionerror}
\end{align}

This quantity (with mean and the standard deviation over test samples) is reported in Figure~\ref{fig:prederr} for our estimator
under the two weighting schemes. Recall that the quadratically decaying weights \eqref{est:generalweight} act essentially like the $\ell_1$ penalty.
For numerical comparison, we also include the result of the estimator with $\ell_1$ penalty, which is the \textit{CSCS (Convex Sparse Cholesky Selection)} method proposed in \cite{2016arXiv161002436K}.
For both the non-adaptive banding and the nested lasso methods, we found that their implementations fail to work due to the collinearity of the columns of $\mathbf{X}$.
\begin{figure}
  \centering
  \includegraphics[width=\linewidth]{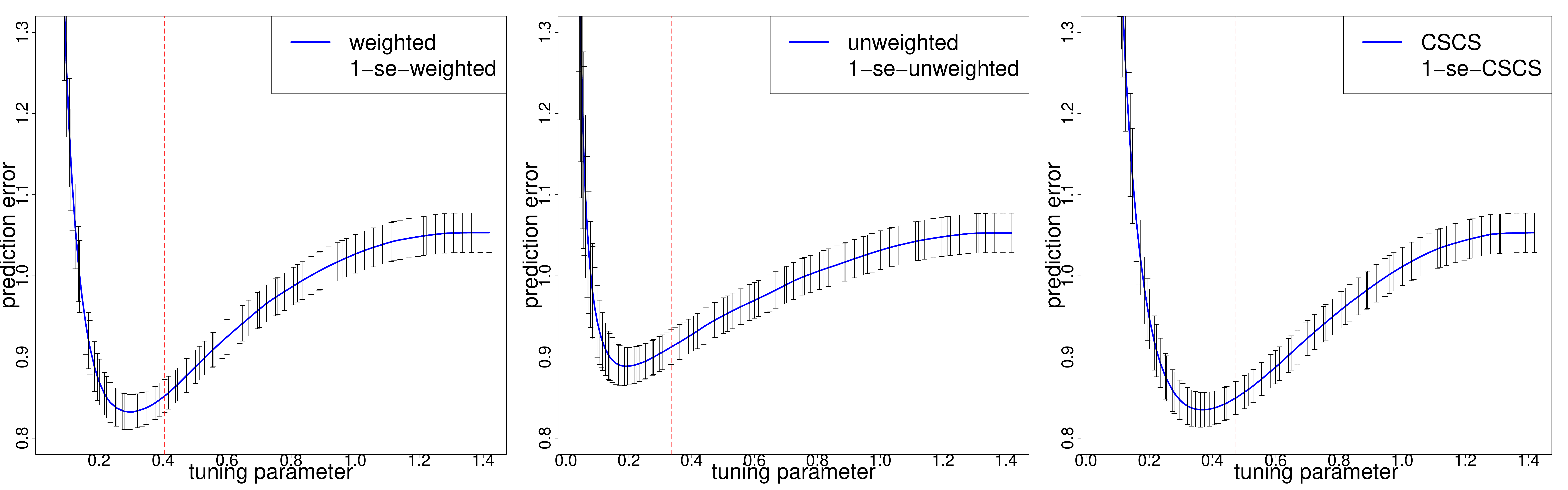}
  \caption{Prediction error (computed on an independent test set) of the weighted (left), unweighted (middle), and CSCS (right) estimators.}
  \label{fig:prederr}
\end{figure}

Figure~\ref{fig:prederr} shows that our estimators are effective in improving modeling
performance over
a diagonal estimator (attained when $\lambda$ is sufficiently large) and strongly outperform the plain MLE (as
evidenced by the sharp increase in prediction error as $\lambda \to 0$).
As expected, the weighted estimator performs very similarly to the CSCS estimator, which uses the $\ell_1$ penalty.
Both of these perform better than the unweighted one.
However, the sparsity pattern obtained by the two penalties are different (as shown in Figure~\ref{fig:bestL}).
\begin{figure}
  \centering
  \includegraphics[width=\linewidth]{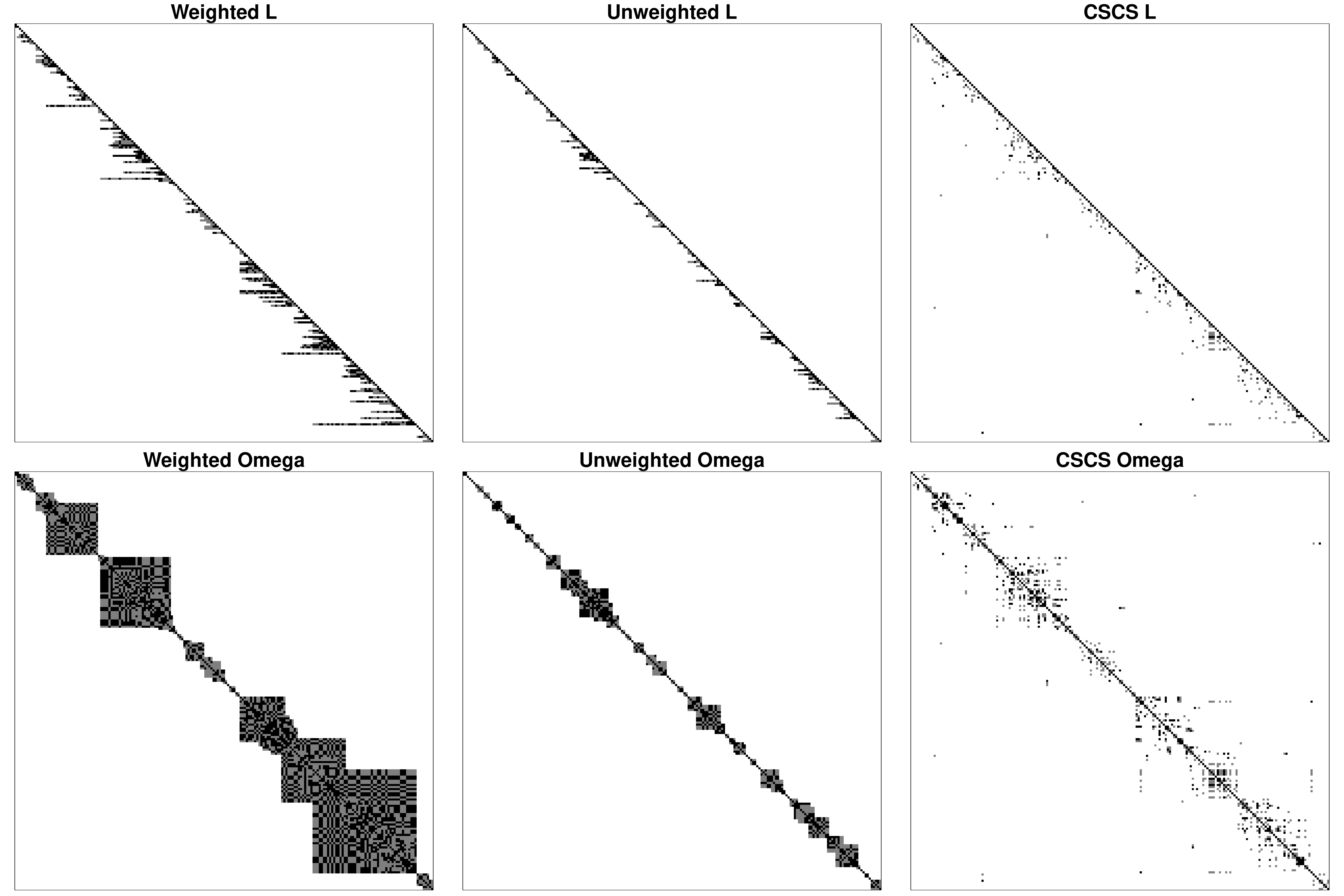}
  \caption{Estimates of linkage disequilibrium with tuning parameters
  selected by the one-standard-error rule and their corresponding precision matrix estimates.}
  \label{fig:bestL}
\end{figure}

In Figure~\ref{fig:bestL} we show the 
recovered signed support of the weighted, unweighted, and CSCS estimators and their corresponding precision matrices.
Black, gray, and white stand for positive, negative, and zero entries, respectively.
Tuning parameters
are chosen using the one-standard-error rule \citep[see, e.g.,][]{ESL}.
The $r$-th row of the estimated matrix $\hat L$ reveals the number of neighboring
SNPs necessary for reliably predicting the state of the $r$-th SNP.
Interestingly, 
we see some evidence of small block-like structures in $\hat{L}$,
consistent with the hotspot model of recombination as previously described.
This regression-based perspective to modeling LD may be a useful
complement to the more standard approach, which focuses on raw marginal correlations.
Finally, the sparsity recovered by the CSCS estimator, which uses the $\ell_1$ penalty, is less easily interpretable, since some entries far from the diagonal
are non-zero, losing the notion of `local'.

\subsection{An Application to Phoneme Classification} \label{sec:phoneme}
In this section, we develop an application of our method to a classification problem described in \citet{ESL}.
The data contain $n = 1717$ continuous speech recordings, which are
categorized %by phoneme frequency 
into two vowel sounds: `aa' ($n_1 = 695$) and `ao' ($n_2 = 1022$).
Each observation $(x_i, y_i)$ has a predictor $x_i \in \real^p$ representing the (log) intensity of the sound across $p = 256$ frequencies
and a class label $y_i \in \{1, -1\}$. 
It may be reasonable to apply our method in this problem since the
features are frequencies, which come with a natural ordering 
%sound recordings.  %, represented by (log) periodograms, remain a natural ordering.

In linear discriminant analysis (LDA), one models the features
as multivariate Gaussian conditional on the class: $x_i | y_i = k \sim N_p(\mu^{(k)}, \Sigma)$ for $k \in \{1, -1\}$; in quadratic discriminant analysis (QDA), one allows each class to
have its own covariance matrix: $x_i | y_i = k \sim N_p(\mu^{(k)}, \Sigma^{(k)})$.
The LDA/QDA classification rules assign an observation $x \in \real^p$ to class $k$ that 
maximizes $\hat{P}(y = k | x) \propto \hat{P}(x |y = k) \hat{P}(y = k)$, where the estimated probability $\hat{P}(x |y = k)$ is calculated
using maximum likelihood estimates $\hat{\mu}^{(k)}$, $\hat{\Sigma}$, and $\hat{\Sigma}^{(k)}$.
More precisely, in the ordered case, the resulting class $k$ maximizes the LDA/QDA scores:
\begin{align}
  \delta^{(k)}_{\mathrm{LDA}}(x) &= x^T \hat{\Omega} \hat{\mu}^{(k)} - \frac{1}{2} (\hat{\mu}^{(k)})^T \hat{\Omega} \hat{\mu}^{(k)} + \log \hat{\pi}^{(k)}
  \nonumber \\
  &= (\hat{L}x)^T \hat{L}\hat{\mu}^{(k)} - \frac{1}{2} \norm{\hat{L}\hat{\mu}^{(k)}}_2^2 + \log \hat{\pi}^{(k)}
  \label{eq:LDA} \\
  \delta^{(k)}_{\mathrm{QDA}}(x) &= x^T \hat{\Omega}^{(k)} \hat{\mu}^{(k)} - \frac{1}{2} (\hat{\mu}^{(k)})^T \hat{\Omega}^{(k)} \hat{\mu}^{(k)} + \log \hat{\pi}^{(k)}
  \nonumber\\
  &= (\hat{L}^{(k)} x)^T \hat{L}^{(k)} \hat{\mu}^{(k)} - \frac{1}{2} \norm{\hat{L}^{(k)}\hat{\mu}^{(k)}}_2^2 + \log \hat{\pi}^{(k)}.
  \label{eq:QDA}
\end{align}
Note that it is the precision matrix, not the covariance matrix, that is used in the above scores.
In the setting where $p > n$, the MLE of $\Omega$ or $\Omega^{(k)}$ does not exist. A regularized estimate of precision matrix
that exploits the natural ordering information can be helpful in this setting.

To demonstrate the use of our estimator in the high-dimensional setting, we randomly split the data into two parts, 
with $10\%$ of the data assigned to the training set and the remaining $90\%$ of the data assigned to the test set.
On the training set,
we use 5-fold cross-validation to select the tuning parameter minimizing misclassification error on the validation data.
The estimates $\hat{L}$ and $\hat{L}^{(k)}$ are then
plugged into 
\eqref{eq:LDA} and \eqref{eq:QDA}
along with 
$\hat{\mu}^{(k)} = \sum_{i \in {\mathrm{class}\, k}} x_i / n^{(k)}$ and $\hat{\pi}^{(k)} = n^{(k)} / n_{\mathrm{train}}$
to calculate the misclassification error in the test set.
For comparison, we also include non-adaptive banding, the nested lasso, and CSCS.
We compute the classification error (summarized in
Table~\ref{tab:phoneme}), averaged over 10 random train-test splits.

We first observe that, in general, the adaptive methods perform better than the non-adaptive one (which assumes a fixed bandwidth).
It is again found that the performance of the weighted estimator is very similar to the one using $\ell_1$ penalty (i.e., the CSCS method).
And our results are comparable to the nested lasso both in LDA and QDA.
Interestingly, we find that the weighted estimator does better in LDA while the unweighted estimator performs better in QDA.
The reason, we suspect, is that QDA requires the estimation of more parameters than LDA and
therefore favors more constrained methods like the unweighted estimator, which more strongly discourages non-zeros from being far from the diagonal
than the weighted one.

\begin{table}
  \centering
  \begin{tabular}{rrrrrr}
    \hline
    & Unweighted & Weighted & Nested Lasso & Non-adaptive & CSCS \\ 
    \hline
    LDA & 0.271 & 0.246 & 0.250 & 0.268 & 0.245 \\ 
    QDA & 0.232 & 0.256 & 0.221 & 0.246 & 0.267 \\ 
    \hline
  \end{tabular}
  \caption{Average test data classification error rate of discriminant analysis of phoneme data}
  \label{tab:phoneme}
\end{table}

%\begin{figure}[h]
%  \centering
%  \includegraphics[width=\linewidth]{phoneme_with_lasso.pdf}
%  \caption{Test data classification error rates of discriminant anlysis of phoneme data}
%  \label{fig:phoneme}
%\end{figure}

An \texttt{R} \citep{citeR} package, named \texttt{varband}, is available on \texttt{CRAN}, implementing our estimator.
The estimation is very fast with core functions coded in {\tt C++},
allowing us to solve large-scale problems in substantially less time than is possible with the \texttt{R}-based implementation of the nested lasso.

\section{Conclusion} \label{sec:conclusion}
We have presented a new flexible method for learning local dependence
in the setting where the elements of a random vector have a known ordering.
The model amounts to sparse estimation of the inverse of the Cholesky factor of the covariance matrix
with variable bandwidth. 
Our method is based
on a convex formulation that allows it to simultaneously
yield a flexible adaptively-banded sparsity pattern, enjoy efficient
computational algorithms, and be studied theoretically.  
To our knowledge, no previous method has all these properties.
We show how the
matrix estimation problem can be decomposed into 
independent row estimation problems, each of which can be solved via
an ADMM algorithm having efficient updates. 
We prove that our method recovers the
signed support of the true Cholesky factor and attains estimation consistency rates 
in several matrix norms under assumptions as mild as those in linear regression problems.
Simulation studies show that our method compares favorably to two pre-existing
estimators in the ordered setting, both in terms of support recovery
and in terms of estimation accuracy. Through a genetic data example, we illustrate how our method
may be applied to model the local dependence of genetic variations in
genes along a chromosome. Finally, we illustrate that our method has favorable
performance in a sound recording classification problem.

\section*{Acknowledgement}
We thank Kshitij Khare for a useful discussion in which he
pointed us to the parametrization in terms of $L$.
We thank Adam Rothman for providing \texttt{R} code for the non-adaptive
banding and the nested lasso methods and Amy Williams for useful
discussions about linkage disequilibrium.
We also thank three referees and an action editor for helpful comments on an earlier manuscript.
This work was supported by NSF DMS-1405746.


\newpage

\appendix
\section{Decoupling Property} \label{online:decouple}
Let $S = \frac{1}{n} \mathbf{X}^T \mathbf{X} \in \real^{p \times p}$ be the sample covariance matrix.
Then the estimator \eqref{est:estimator}
is the solution to the following minimization problem:
\begin{align}
  \min_{\substack{L: L_{rr}>0 \\ L_{rk}=0\text{ for }r < k }}
  \left\{-2\sum_{r=1}^p
  \log L_{rr} + \tr(SL^TL) + \lambda 
  \sum_{r=2}^p \sum_{\ell = 1}^ {r-1} \sqrt{\sum_{m=1}^\ell w_{\ell m}^2 L_{rm}^2} 
\right\}.
\nonumber
\end{align}
First note that under the lower-triangular constraint 
\begin{align}
  \tr \left( S L^T L \right) =
  \frac{1}{n} \sum_{r=1}^p \tr\left( \mathbf{X} 
  L_{\cdot r}^T L_{r\cdot} \mathbf{X}^T \right)
  = \frac{1}{n} \sum_{r=1}^p \norm{\mathbf{X} L_{\cdot r}^T}_2^2
  = \frac{1}{n} \sum_{r=1}^p \norm{\mathbf{X}_{1:r} L_{1:r,r}^T}_2^2,
  \nonumber
\end{align}
where $\mathbf{X}_{1:r}$ is a matrix of the first $r$ columns of $\mathbf{X}$. Thus
\begin{align}
  & -2\sum_{r=1}^p
  \log L_{rr} + \tr(SL^TL) + \lambda \sum_{r=2}^p \sum_{\ell = 1}^{r-1} 
  \sqrt{\sum_{m=1}^\ell w_{\ell m}^2 L_{rm}^2} 
  \nonumber\\
  = & -2 \log L_{11} +
  \frac{1}{n} \norm{\mathbf{X}_1 L_{11}}_2^2 + \sum_{r=2}^p \left( 
  -2 \log L_{rr} + \frac{1}{n} \norm{\mathbf{X}_{1:r}L_{1:r,r}^T}_2^2 + 
  \lambda \sum_{\ell = 1}^{r-1} 
  \sqrt{\sum_{m=1}^\ell w_{\ell m}^2 L_{rm}^2} \right).
  \nonumber 
\end{align}

Therefore the original problem can be decoupled into $p$ separate problems.
In particular, a solution $\hat{L}$ can be written in a row-wise form with
\[
  \hat{L}_{11} = \argmin_{L_{11}>0} \left\{ -2\log L_{11} + \frac{1}{n} 
  \norm{\mathbf{X}_1L_{11}}_2^2 \right\} = \frac{1}{\sqrt{S_{11}}},
\]
and for $r=2,\ldots,p$,
\begin{align}
  \hat{L}^T_{1:r,r} = \argmin_{\beta \in \real^r: \beta_r>0} 
  \left\{ -2\log \beta_r + \frac{1}{n} \norm{\mathbf{X}_{1:r} \beta }_2^2 + \lambda \sum_{\ell = 1}^{r-1} \sqrt{\sum_{m=1}^\ell w_{\ell m}^2 \beta_{m}^2}\right\}.
  \nonumber
\end{align}

\section{A Closed-Form Solution to \eqref{comp:updatebeta}} \label{online:closeformupdatebeta}
The objective function in \eqref{comp:updatebeta} is a smooth function. 
Taking the derivative with respect to $\beta$ and setting to zero gives the following system of equations:
\begin{align}
  -2 \frac{1}{\beta_r} \mathbf{e}_r + \frac{2}{n} \mathbf{X}_{1:r}^T 
  \mathbf{X}_{1:r} \beta
  + u^{(t-1)} + \rho \left( \beta - \gamma^{(t-1)} \right) = \mathbf{0}.
  \nonumber
\end{align}
Letting $S^{(r)} = \frac{1}{n} \mathbf{X}_{1:r}^T \mathbf{X}_{1:r}$, then the equations above can be further decomposed into
\begin{align}
  &- \frac{2}{\beta_r} + \left( 2 S^{(r)}_{rr} + \rho \right) \beta_r + 
  2 S^{(r)}_{r,-r} \beta_{-r} +  u^{(t-1)}_r - \rho \gamma^{(t-1)} _r = 0,
  \nonumber\\
  & \left( 2 S^{(r)}_{-r,-r} + \rho I \right) \beta_{-r} + 
  2 S^{(r)}_{-r,r} \beta_r +  u^{(t-1)}_{-r} - \rho \gamma^{(t-1)} _{-r}  = 
  \mathbf{0}.
  \nonumber
\end{align}
Solving for $\beta_{-r}$ in the second system of equations gives
\begin{align}
  \beta_{-r} = - \left( 2 S^{(r)}_{-r,-r} + \rho I \right)^{-1} \left( 2 S^{(r)}_{-r,r} \beta_r
  + u^{(t-1)}_{-r} - \rho \gamma^{(t-1)}_{-r}\right),
  \nonumber
\end{align}
which is then plugged back in the first equation to give 
\begin{align}
  2 \frac{1}{ \beta_r } + A \beta_r + B = 0,
  \nonumber
\end{align}
where
\begin{align}
  &A = 4 S^{(r)}_{r, -r} \left( 2 S^{(r)}_{-r,-r} + \rho I \right)^{-1} S^{(r)}_{-r,r} 
  - 2 S^{(r)}_{r,r} - \rho,
  \nonumber\\
  &B = 2 S^{(r)}_{r,-r} \left( 2 S^{(r)}_{-r,-r} + \rho I \right)^{-1}\left( u^{(t-1)}_{-r} 
  - \rho \gamma^{(t-1)}_{-r} \right) - u^{(t-1)}_r + \rho \gamma^{(t-1)}_r.
  \nonumber
\end{align}
Solving for $\beta_r$ gives the closed-form update.

\section{Dual Problem of \eqref{comp:updategamma}} \label{online:dualproblem}
\begin{lemma}   
  A dual problem of \eqref{comp:updategamma} is
  \begin{align}
    \min_{a^{(\ell)}\in \real^r} \left\{ \norm{ y^{(t)} - \frac{\lambda}{\rho} 
    \sum_{\ell=1}^{r-1} W^{(\ell)} \ast a^{(\ell)}}_2^2 
    \st \norm{ \left( a^{(\ell)} \right)_{g_{r,\ell}} }_2 \leq 1,
    \quad\left( a^{(\ell)} \right)_{g_{r,\ell}^c} = 0  \right\},
    \label{comp:dual}
  \end{align}
  where $y^{(t)} = \beta^{(t)} + \frac{1}{\rho}u^{(t-1)}$.
  Also, given a solution $\hat{a}^{(1)},\dots,\hat{a}^{(r-1)}$, the solution
  to \eqref{comp:updategamma} can be written as
  \begin{align}
    \gamma^{(t)} = y^{(t)} -
    \frac{\lambda}{\rho} \sum_{\ell=1}^{r-1} W^{(\ell)} \ast 
    \hat{a}^{(\ell)}.
    \label{comp:primaldualrelation}
  \end{align}
\end{lemma}
\begin{proof}
  Note that 
  \begin{align}
    \sqrt{\sum_{m=1}^\ell w_{\ell m}^2 \gamma_m^2} &= 
    \norm{ \left( W^{(\ell)}\ast \gamma  \right)_{g_{r,\ell}} }_2 \nonumber\\
    &= \max \left\{ \left< W^{(\ell)} \ast a^{(\ell)}, \gamma \right> ,
    \st \norm{ \left( a^{(\ell)} \right)_{g_{r,\ell}} }_2 \leq 1, 
    \quad \left( a^{(\ell)} \right)_{g_{r,\ell}^c} = 0   \right\}.
    \nonumber
  \end{align}

  Thus, the minimization problem in \eqref{comp:updategamma} becomes 
  \begin{align}
    &\min_{\gamma} \left\{ \half \norm{ \gamma -  y^{(t)} }_2^2 + 
    \frac{\lambda}{\rho} \sum_{\ell=1}^{r-1} 
    \norm{ \left( W^{(\ell)}\ast \gamma  \right)_{g_{r,\ell}}  }_2 \right\} 
    \nonumber\\
    =&\min_{\gamma} \left\{ \max_{a^{(\ell)}} \left\{ \half 
      \norm{ \gamma - y^{(t)} }_2^2 + \frac{\lambda}{\rho} \sum_{\ell=1}^{r-1}  
      \left< W^{(\ell)} \ast a^{(\ell)}, \gamma \right> ,
      \norm{ \left( a^{(\ell)} \right)_{g_{r,\ell}} }_2 \leq 1, 
      \left( a^{(\ell)} \right)_{g_{r,\ell}^c} = 0 \right\} \right\} 
      \nonumber\\
      =&\max_{a^{(\ell)}}\left\{ \min_{\gamma} \left\{ \half\norm{ \gamma - 
        y^{(t)} }_2^2 + \frac{\lambda}{\rho} \sum_{\ell=1}^{r-1}  
        \left< W^{(\ell)} \ast a^{(\ell)}, \gamma \right> ,
        \norm{ \left( a^{(\ell)} \right)_{g_{r,\ell}} }_2 \leq 1, 
        \left( a^{(\ell)} \right)_{g_{r,\ell}^c} = 0 \right\} \right\},
        \nonumber
      \end{align}
      where $y^{(t)} = \beta^{(t)} + \frac{1}{\rho}u^{(t-1)}$.
      We solve the inner minimization problem by setting the derivative to zero,
      \[
        \gamma - y^{(t)} + \frac{\lambda}{\rho} \sum_{\ell=1}^{r-1}  
        W^{(\ell)} \ast a^{(\ell)} = 0,
      \]
      which gives the primal-dual relation,
      \[
        \gamma = -\frac{\lambda}{\rho} \sum_{\ell=1}^{r-1} W^{(\ell)} 
        \ast a^{(\ell)} + y^{(t)}.
      \]

      Using this gives
      \begin{align}
        \min_{\gamma} &\left\{ \half\norm{ \gamma - y^{(t)} }_2^2 + 
        \frac{\lambda}{\rho} \sum_{\ell=1}^{r-1} 
        \norm{ \left( W^{(\ell)}\ast \gamma  \right)_{g_{r,\ell}}  }_2 \right\} 
        \nonumber\\
        =\max_{a^{(\ell)}} &\left\{ \half \norm{ - \frac{\lambda}{\rho} 
        \sum_{\ell=1}^{r-1} W^{(\ell)} \ast a^{(\ell)} }_2^2 + 
        \frac{\lambda}{\rho} \sum_{\ell=1}^{r-1}  \left< W^{(\ell)} 
        \ast a^{(\ell)}, - \frac{\lambda}{\rho} \sum_{\ell=1}^{r-1} 
        W^{(\ell)} \ast a^{(\ell)} + y^{(t)} \right> \right.
        \nonumber\\
        &\left. \st \norm{ \left( a^{(\ell)} \right)_{g_{r,\ell}} }_2 \leq 1,
        \quad\left( a^{(\ell)} \right)_{g_{r,\ell}^c} = 0 \right\}
        \nonumber\\
        =\min_{a^{(\ell)}} & \left\{\norm{ y^{(t)} -
        \frac{\lambda}{\rho} \sum_{\ell=1}^{r-1} 
        W^{(\ell)} \ast a^{(\ell)}}_2^2 \st 
        \norm{ \left( a^{(\ell)} \right)_{g_{r,\ell}} }_2 \leq 1,
        \quad\left( a^{(\ell)} \right)_{g_{r,\ell}^c} = 0 \right\}.
        \nonumber
      \end{align}
    \end{proof}
    \begin{algorithm}
      \caption{BCD on the dual problem \eqref{comp:dual}}
      \begin{algorithmic}[1]
        \State Let $y^{(t)} = \beta^{(t)} + \frac{1}{\rho}u^{(t-1)}$ 
        \State Initialize $\hat{a}^{(\ell)} \gets 0$ for all $\ell = 1,\cdots,r-1$ \;
        \For{$\ell = 1,\cdots,r-1$}
        \State
        $\hat{z}^{(\ell)} \gets y^{(t)} - \frac{\lambda}{\rho}\sum_{k=1}^{r-1} 
        W^{(k)} \ast \hat{a}^{(k)}$\;
        Find a root $\hat{\nu}_\ell$ that satisfies
        \begin{align}
          h_\ell(\nu) := \sum_{m=1}^\ell \frac{w_{\ell m}^2}
          {\left( w_{\ell m}^2 + \nu \right)^2} \left( \hat{z}^{(\ell)}_m \right)^2
          =\frac{\lambda^2}{\rho^2} 
          \label{comp:rootfunction}
        \end{align}
        \For{$m = 1,\cdots,\ell$}
        \State
        $\hat{a}^{(\ell)}_m \gets \frac{w_{\ell m}}{\frac{\lambda}{\rho}\left( w_{\ell m}^2
        + \left[ \hat{\nu}_\ell \right]_+\right)} \hat{z}^{(\ell)}_m$ \;
      \EndFor
    \EndFor
    \State \Return{$\left\{ \hat{a}^{(\ell)} \right\}$ as a solution 
    to \eqref{comp:dual}} \;
    \State \Return{$ \gamma^{(t)}
    = y^{(t)}-\frac{\lambda}{\rho}\sum_{\ell=1}^{r-1} W^{(\ell)} \ast \hat{a}^{(\ell)}$
    as a solution to \eqref{comp:updategamma}} \;
  \end{algorithmic}
  \label{online:alg:BCDondual}
\end{algorithm}

\section{Elliptical Projection} \label{online:ellipticalprojection}
We adapt the same procedure as in Appendix B of \citet{bien2015convex} to update one $a^{(\ell)}$ in Algorithm \eqref{online:alg:BCDondual}. 
By \eqref{comp:updategamma} we need to solve a problem of the form
\begin{align}
  \min_{a \in \real^{\ell}} \norm{\hat{z}^{(\ell)} - \tau D a}_2^2 
  \st \norm{a}_2 \leq 1,
  \nonumber
\end{align}
where $\tau = \frac{\lambda}{\rho}$ and $D = \mathrm{diag}
(w_{\ell m})_{ m\leq \ell} \in \real^{\ell\times \ell}$.
If $\norm{D^{-1} \hat{z}^{(\ell)}}_2 \leq \tau$, then clearly
$\hat{a} = \frac{1}{\tau}D^{-1} \hat{z}^{(\ell)}$.
Otherwise, we use the Lagrangian multiplier method to solve the constrained
minimization problem above. Specifically, we find a stationary point of 
\begin{align}
  \mathcal{L}\left( a,\nu \right) = \norm{\hat{z}^{(\ell)} - \tau D a}_2^2 + 
  \nu \tau^2 \left( \norm{a}_2^2 -1 \right).
  \nonumber
\end{align}
Taking the derivative with respect to $a$ and set it equal to zero, we have
\begin{align}
  \hat{a}_m = \frac{w_{\ell m}}{\tau(w_{\ell m}^2 + \hat{\nu})} \hat{z}^{(\ell)}_m,
  \nonumber
\end{align}
for each $m \leq \ell$, and $\hat{\nu}$ is such that $\norm{\hat{a}}_2 = 1$, which means it 
satisfies \eqref{comp:rootfunction}.
By observing that $h_\ell(\nu)$ is a decreasing function of $\nu$ and 
$w_{\ell\ell} = \max_{m \leq \ell}w_{\ell m}$, following Appendix B of \citet{bien2015convex},
we obtain lower and upper bounds for $\hat{\nu}$:
\begin{align}
  \left[ \frac{1}{\tau} \norm{D\hat{z}^{(\ell)}}_2 - w_{\ell\ell}^2 \right]_+
  \leq \hat{\nu} \leq \frac{1}{\tau} \norm{D\hat{z}^{(\ell)}}_2,
  \nonumber
\end{align}
which can be used as an initial interval for finding $\hat{\nu}$ using Newton's method.
In practice, we usually find $\hat{\nu}$ from the equation
$\frac{1}{h(\nu)} = \tau^{-2}$ for better numerical stability.

We end this section with a characterization of the solution to \eqref{comp:updategamma},
which says that the solution can be written as 
$\gamma^{(t)} = y^{(t)} \ast \hat{t}$,
where $\hat{t}$ is some data-dependent vector in $\real^{r}$.

\begin{theorem} \label{thm:comp:tapering}
  A solution to \eqref{comp:updategamma} can be written as
  $ \gamma^{(t)}= y^{(t)} \ast \hat{g}$, 
  where the data-dependent vector $\hat{g} \in \real^{r}$ is given by
  \begin{align}
    \hat{g}_m = \prod_{\ell = m}^{r-1} \frac{\left[ \hat{\nu}_\ell \right]_+}
    {w_{\ell m}^2 + \left[ \hat{\nu}_\ell \right]_+} 
    \nonumber
  \end{align}
  and $\hat{g}_r = 1$,
  where $\hat{\nu}_\ell$ satisfies $\tau^2 
  = \sum_{m=1}^\ell \frac{w_{\ell m}^2}
  {\left( w_{\ell m}^2 + \nu \right)^2} \left( \hat{z}^{(\ell)}_m \right)^2$.
\end{theorem}

\begin{proof} \label{proof:comp:tapering}
  By \cite{jenatton2011structured}, we can get a solution to 
  \eqref{comp:updategamma} in a single pass as described in Algorithm 
  \ref{online:alg:BCDondual}. If we start from
  $\hat{z}^{(1)} = y^{(t)}$, then for $\ell = 1,\cdots,r-1$ and each $m \leq \ell$, 
  \begin{align}
    \hat{z}^{(\ell+1)}_m = \hat{z}^{(\ell)}_m - \tau 
    w_{\ell m} \hat{a}^{(\ell)}_m  = \frac{\left[ \hat{\nu}_\ell \right]_+}
    {w_{\ell m}^2 + \left[ \hat{\nu}_\ell \right]_+} \hat{z}^{(\ell)}_m.
    \nonumber
  \end{align}
  By \eqref{comp:primaldualrelation}, $ \gamma^{(t)} = \hat{z}^{(r-1)}$,
  and the result follows.
\end{proof}

A key observation from this characterization is that a banded sparsity pattern is induced
in solving \eqref{comp:updategamma}, which in turn implies the same property of the output of Algorithm \ref{alg:ADMM}.
\begin{corollary}
  A solution $\gamma^{(t)}$ to \eqref{comp:updategamma} has banded sparsity, i.e.,
  $\left( \gamma^{(t)} \right)_{1:\hat{J}} = 0$ for $\hat{J} = 
  \max\left\{ \ell:\hat{\nu}_\ell \leq 0 \right\}$.
\end{corollary}

\section{Uniqueness of the Sparse Row Estimator} \label{online:sparserow}
\begin{lemma}\label{lem:theory:opt_cond}(Optimality condition)
  For any $\lambda>0$ and a $n$-by-$p$ sample matrix $\mathbf{X}$, 
  $\hat{\beta}$ is a solution to the problem
  \[ 
    \min_{\beta \in \real^r} \left\{ -2\log \beta_r + 
      \frac{1}{n} \norm{\mathbf{X}_{1:r} \beta }_2^2 + \lambda \sum_{\ell = 1}^{r-1} 
      \sqrt{\sum_{m=1}^\ell w_{\ell m}^2 \beta_{m}^2}\right\}
    \]
    if and only if there 
    exist $\hat{a}^{(\ell)} \in \real^{r}$ for $\ell = 1,\dots,r-1$ such that
    \begin{align}
      -\frac{2}{\hat{\beta}_r} \mathbf{e}_r + \frac{2}{n} \mathbf{X}_{1:r}^T
      \mathbf{X}_{1:r} \hat{\beta} + 
      \lambda \sum_{\ell = 1}^{r-1} W^{(\ell)} \ast \hat{a}^{(\ell)} = 0 
      \label{theory:opt_cond}
    \end{align}
    with $\left( \hat{a}^{(\ell)} \right)_{g_{r,\ell}^c}=0$, 
    $\left( \hat{a}^{(\ell)} \right)_{g_{r,\ell}} = 
    \frac{\left( W^{(\ell)}\ast \hat \beta  
    \right)_{g_{r,\ell}}}{\left\| \left( W^{(\ell)}\ast \hat \beta \right)_{g_{r,\ell}}
  \right\|_2}$ for $\hat \beta _{g_{r,\ell}} \neq 0 $ and 
  $\left\| \left( \hat{a}^{(\ell)} \right)_{g_{r,\ell}} \right\|_2 \leq 1$ 
  for $\hat{\beta}_{g_{r,\ell}} = 0$.
\end{lemma}

\begin{lemma} \label{lem:theory:moresparse}
  Take $\hat{\beta}$ and $\hat{a}^{(\ell)}$ as in the previous lemma. 
  Suppose that
  \[
    \left\| \left(\hat{a}^{(\ell)}\right)_{g_{r,\ell}} \right\|_2 < 1 \quad 
    \quad \text{ for } \quad \ell = 1,\dots,J( \hat{\beta} )
  \]
  then for any other solution $\tilde{\beta}$ to \eqref{est:subproblem}, 
  it is as sparse as $\hat{\beta}$ if not more. 
  In other words,
  \[
    K (\tilde{\beta}) \leq \hat{K}_r.
  \]
\end{lemma}

\begin{lemma} \label{lem:theory:uniqueness}(Uniqueness)
  Under the conditions of the previous lemma, 
  let $\hat{\mathcal S} = \left\{ i: \hat{\beta}_{i} \neq 0 \right\}$.
  If $\mathbf{X}_{\hat{\mathcal S}}$ 
  has full column rank (i.e., 
  $ \rank\left( \mathbf{X}_{\hat{\mathcal S}} \right) =  
  |\hat{\mathcal{S}}|$) 
  then $\hat{\beta}$ is unique. 
\end{lemma}
\begin{proof}
  See Appendices \ref{proof:opt_cond}, \ref{proof:moresparse}, and \ref{proof:uniqueness}.
\end{proof}

\section{Proof of Theorem \ref{thm:row}} \label{online:proofofrow}
We start with introducing notation. 
\textbf{From now on we suppress the dependence on $r$ in notation for simplicity}. 
We denote the group structure
$g_\ell = \left\{ 1,\cdots,\ell \right\}$ for $\ell \leq r$ for each $r = 1,\dots,p$. For any vector $\beta \in \real^r$,
we let $\beta_{g_\ell} \in \real^\ell$ be the vector with elements 
$\left\{ \beta_m: m \leq \ell \right\}$.
We also introduce the weight vector $W^{(\ell)} \in \real^p$ with
$\left( W^{(\ell)} \right)_m = w_{\ell m}$ where $w_{\ell m}$ can be defined
as in \eqref{est:generalweight} or $w_{\ell m} = 1$.
Finally recalling from Section \ref{sec:theory} the definition of
$\mathcal I$, we denote $\mathcal S = \mathcal I \cup \{r\} = 
\left\{ J+1,\dots,r \right\}$ and $\mathcal S^c = \left\{ 1, 2, \ldots, J \right\}$.

The general idea of the proof depends on the primal-dual witness procedure in 
\citet{wainwright2009sharp} and \citet{ravikumar2011high}. 
Considering the original problem \eqref{est:subproblem} for any
$r=2,\dots,p$, we construct the primal-dual 
witness solution pairs 
$\left( \tilde{ \beta }, \sum_{\ell=1}^{r-1} 
W^{(\ell)} \ast \tilde{a}^{(\ell)} \right)$ 
as follows: 

\begin{enumerate} \label{online:Primal-Dual}
  \item[(a)] 
    Solve the restricted subproblem with the true bandwidth
    $K = r-1- J$:
    \[
      \tilde{ \beta } = \argmin_{\substack {\beta_{r} > 0\\\beta _{\mathcal S^c}=0}}
      \left\{  -2 \log \beta_{r} + \frac{1}{n} \norm{\mathbf{X}_{1:r} \beta}_2^2 + 
      \lambda \sum_{\ell=1}^{r-1} 
      \left\| \left(W^{(\ell)} \ast \beta \right)_{g_{\ell}}  \right\|_2 \right\}.
    \]
    The solution above can be written as
    \[
      \tilde{\beta} = \begin{pmatrix}
        \mathbf{0}_{J}\\
        \tilde{\gamma}
      \end{pmatrix},
    \]
    where 
    \[
      \tilde{\gamma} = \argmin_{\gamma \in \real^{K+1}}
      \left\{ -2 \log \gamma_{K+1} + \frac{1}{n} 
      \norm{ \mathbf{X}_{\mathcal{S}}
    \gamma}_2^2 + \lambda \sum_{\ell = 1}^{K} 
    \norm{\left( \tilde{W}^{(\ell)} \ast \gamma \right)_{g_\ell}}_2\right\},
  \]
  with
  \[
    \tilde{W}^{(\ell)} = \left( W^{(\ell+ J)} \right)_{\mathcal{S}}
    \quad \iff \quad \sum_{\ell = 1}^{K} 
    \norm{\left( \tilde{W}^{(\ell)} \ast \gamma \right)_{g_\ell}}_2 =
    \sum_{\ell = J+1}^{r-1} \sqrt{\sum_{m=J+1}^{r-1} 
    w_{\ell m}^2 \gamma_{m-J}^2}.
  \]
\item [(b)] 
  By Lemma \ref{lem:theory:opt_cond}, there exist 
  $\tilde{b}^{(\ell)} \in \real^{K+1}$ for $\ell =1,\dots,K$,
  such that 
  $\left( \tilde{b}^{(\ell)} \right)_{g_{\ell}^c} = 0$ and 
  \[
    \left(\tilde{b}^{(\ell)}\right)_{g_{r\ell}} = 
    \frac{\left( \tilde{W}^{(\ell)} \ast \tilde{\gamma} \right)_{g_{\ell}} }
    {\left\| \left( \tilde{W}^{(\ell)} \ast \tilde{\gamma} \right)_{g_{\ell}} 
  \right\|_2},
\]
satisfying 
\[
  -\frac{2}{\tilde{\gamma}_{K+1}} \mathbf{e}_{K+1} + 
  \frac{2}{n} \mathbf{X}_{\mathcal{S}}^T \mathbf{X}_{\mathcal{S}}
  \tilde{\gamma} + 
  \lambda \sum_{\ell = 1}^{K} \tilde{W}^{(\ell)} \ast 
  \tilde{b}^{(\ell)} = 0.
\]

\item [(c)]
  For $\ell = J+1,\dots,r-1$, we let
  \[
    \tilde{a}^{(\ell)} = \begin{pmatrix}
      \mathbf{0}_{J} \\
      \tilde{b}^{(\ell-J)}
    \end{pmatrix}.
  \]
  Then we have $\left(\tilde{a}^{(\ell)}\right)_{g_{\ell}^c} = 0$,
  $\left\| \left(\tilde{a}^{(\ell)}\right)_{g_{\ell}} \right\|_2 \leq 1$,
  $\left(\tilde{a}^{(\ell)}\right)_{g_{\ell}} = 
  \frac{\left( W^{(\ell)} \ast \tilde{\beta} \right)_{g_{\ell}}}
  {\left\| \left( W^{(\ell)}\ast \tilde{\beta} \right)_{g_{\ell}} \right\|_2}$ 
  for $\tilde{\beta}_{g_{\ell}} \neq 0$.
\item[(d)]
  For each $\ell = 1,..., J$, we choose $\tilde{a}^{(\ell)} \in \real^{r}$ satisfying 
  \[
    \left(\tilde{a}^{(\ell)}\right)_{\ell'}= 0 \quad \text{ for any }
    \ell' \neq \ell \quad \text {and} \quad 
    \left(\tilde{a}^{(\ell)}\right)_{\ell} = 
    - \frac{2}{\lambda w_{\ell \ell}} 
    \left( S \tilde{\beta} \right)_{\ell} 
    = -\frac{2}{n \lambda} 
    \mathbf{X}^T_{\ell} \mathbf{X}_{\mathcal{S}} \tilde{\beta}_{\mathcal{S}} .
  \]

  By construction and the fact that $w_{\ell\ell} = 1$,
  \[
    \lambda \left( W^{(\ell)} \ast \tilde{a}^{(\ell)} \right)_{\ell} =
    \lambda w_{\ell \ell} \left( \tilde{a}^{(\ell)}\right)_{\ell} = 
    -2 \left( S \tilde{\beta} \right)_{\ell}.
  \]
  By Lemma \ref{lem:theory:opt_cond},
  $\left\{ \tilde{a}^{(\ell)} \right\}$ satisfies the optimality condition 
  \eqref{theory:opt_cond}:
  \begin{align}	
    -\frac{2}{\tilde{\beta}_{r}} \mathbf{e}_{r} + 
    \frac{2}{n} \mathbf{X}_{1:r}^T \mathbf{X}_{1:r} \tilde{\beta} + 
    \lambda \sum_{\ell = 1}^{r-1} W^{(\ell)} \ast 
    \tilde{a}^{(\ell)} = 0 
    \label{eq:optimality_pd}
  \end{align}

\item [(e)] 
  Verify the strict dual feasibility condition for  
  $\ell = 1,...,J$
  \begin{align}
    \left| \frac{2}{n \lambda} 
    \mathbf{X}^T_{\ell} \mathbf{X}_{\mathcal{S}} \tilde{\beta}_{\mathcal{S}}
    \right| = 
    \left| \left(\tilde{a}^{(\ell)}\right)_{\ell} \right| = 
    \left\| \left(\tilde{a}^{(\ell)}\right)_{g_{\ell}} \right\|_2 < 1.
    \label{SDF} 
  \end{align}
\end{enumerate}

At a high level, steps (a) through (d) construct a pair 
$\left( \tilde{ \beta }, \left\{ \tilde{a}^{(\ell)} \right\} \right)$ 
that satisfies the optimality condition \eqref{theory:opt_cond},
but the $ \left\{ \tilde{a}^{(\ell)} \right\}$ is not necessarily
guaranteed to be a member of 
$\partial \left( P ( \tilde{ \beta } ) \right)$. 
Step (e) does more than verifying the necessary conditions for it to belong to 
$\partial \left( P ( \tilde{ \beta })\right)$. The strict dual feasibility condition, once verified, 
ensures the uniqueness of the solution.
Note that by construction in Step (b), 
$\left\{ \tilde{a}^{(\ell)} \right\}$ satisfies dual feasibility conditions for 
$\ell = J+1,...,r-1$ since 
$\left\{ \tilde{b}^{(\ell)} \right\}$ does,
so it remains to verify for $\ell = 1,...,J$ (see Step (c)).

For each $\ell = 1,...,J$, by the construction in Step (d), 
$\left( \tilde{a}^{(\ell)}\right)_{g_{\ell}^c} = 0$. 
Note that $\tilde{\beta}_{g_{J}} = 0$ implies $\tilde{\beta}_{g_{\ell}} = 0$.
Thus, for $\tilde{a}^{(\ell)}$ to satisfy conditions in Lemma \ref{lem:theory:opt_cond},
it suffices to show \eqref{SDF}.

If the primal-dual witness procedure succeeds, 
then by construction, the solution $\tilde{\beta}$,
whose support is contained in the support of the true $L_{r\cdot}$, 
is a solution to \eqref{est:subproblem}. 
Moreover, by strict dual feasibility and Lemma \ref{lem:theory:uniqueness}, we know 
that $\tilde{\beta}$ is the unique solution $\hat{\beta}$ to the 
unconstrained problem \eqref{est:subproblem}. 
Therefore, the support of $\hat{\beta}$ is contained in the 
support of $L_{r\cdot}$. 

In the following we adapt the same proof technique as \cite{wainwright2009sharp}
to show that the primal-dual witness succeeds with high probability, 
from which we first conclude that $K(\hat{\beta}) \leq K$.

\subsection{Proof of Property 1 in Theorem \ref{thm:row}}
\begin{proof}
  We need to verify the strict dual feasibility \eqref{SDF}.
  By \eqref{eq:optimality_pd},
  \begin{align}
    & - \frac{2}{ \tilde{\beta}_r } + \frac{2}{n} \mathbf{X}_r^T \mathbf{X}_r
    \tilde{\beta}_r + \frac{2}{n} \mathbf{X}_r^T \mathbf{X}_{\mathcal{I}}
    \tilde{\beta}_{\mathcal{I}} = 0,
    \label{eq:i} \\
    & \frac{2}{n} \mathbf{X}_{\mathcal{I}}^T \mathbf{X}_r \tilde{\beta}_r + 
    \frac{2}{n} \mathbf{X}_{\mathcal{I}}^T \mathbf{X}_{\mathcal{I}}
    \tilde{\beta}_{\mathcal{I}} + \lambda \left( \sum_{\ell=1}^{r-1} W^{(\ell)} \ast
    \tilde{a}^{(\ell)}\right)_{\mathcal{I}} = 0.
    \label{eq:ii}
  \end{align}

  From \eqref{eq:ii},
  \begin{align}
    \tilde{\beta}_{\mathcal{I}} = - \left( \mathbf{X}_{\mathcal{I}}^T 
    \mathbf{X}_{\mathcal{I}}\right)^{-1} \left[ \mathbf{X}_{\mathcal{I}}^T
      \mathbf{X}_r \tilde{\beta}_r + \frac{\lambda n}{2}
      \left( \sum_{\ell=1}^{r-1} W^{(\ell)} \ast \tilde{a}^{(\ell)}\right)
    _{\mathcal{I}}\right].
    \label{eq:iii}
  \end{align}

  Plugging \eqref{eq:iii} back into \eqref{eq:i} and denoting
  $\mathbf{C}_{\mathcal{I}} = \mathbf{X}_{\mathcal{I}}
  \left( \mathbf{X}_{\mathcal{I}}^T \mathbf{X}_{\mathcal{I}}\right)^{-1} 
  \left( \sum_{\ell=1}^{r-1} W^{(\ell)} \ast 
  \tilde{a}^{(\ell)}\right)_{\mathcal{I}}
  $
  and $\mathbf{O}_{\mathcal{I}} = 
  \mathbf{I} -  \mathbf{X}_{\mathcal{I}} \left( \mathbf{X}_{\mathcal{I}}^T 
  \mathbf{X}_{\mathcal{I}}\right)^{-1}  \mathbf{X}_{\mathcal{I}}^T$
  as the orthogonal projection matrix onto the orthogonal complement of the column space of 
  $\mathbf{X}_{\mathcal{I}}$, 
  we have
  \[
    - \frac{2}{\tilde{\beta}_r} + \frac{2}{n} \mathbf{X}_r^T 
    \mathbf{O}_{\mathcal{I}}
    \mathbf{X}_r \tilde{\beta}_r - 
    \lambda \mathbf{X}_r^T \mathbf{C}_{\mathcal{I}} = 0,
  \]
  which implies that
  \begin{align}
    \tilde{\beta}_r = \frac{\frac{\lambda}{2} 
    \mathbf{X}^T_r \mathbf{C}_{\mathcal{I}}+ \sqrt{
      \frac{\lambda^2}{4} \left( \mathbf{X}^T_r \mathbf{C}_{\mathcal{I}} \right)^2 + 
      \frac{4}{n} \mathbf{X}_r^T\mathbf{O}_{\mathcal{I}} \mathbf{X}_r}}
      {\frac{2}{n} \mathbf{X}_{r}^T \mathbf{O}_{\mathcal{I}} \mathbf{X}_r}
      \label{eq:betar}
    \end{align}
    and that
    \begin{align}
      \left( \tilde{a}^{(\ell)} \right)_\ell & =  -\frac{2}{n \lambda} 
      \mathbf{X}^T_{\ell} \mathbf{X}_{\mathcal{S}} \tilde{\beta}_{\mathcal{S}}
      =  -\frac{2}{n \lambda} 
      \mathbf{X}^T_{\ell} \mathbf{X}_r \tilde{\beta}_{r}
      -\frac{2}{n \lambda } 
      \mathbf{X}^T_{\ell} \mathbf{X}_{\mathcal{I}} \tilde{\beta}_{\mathcal{I}}
      \nonumber\\
      & = -\frac{2}{n \lambda } 
      \mathbf{X}^T_{\ell} \mathbf{X}_r \tilde{\beta}_{r}
      + \frac{2}{n \lambda } 
      \mathbf{X}^T_{\ell} \mathbf{X}_{\mathcal{I}} 
      \left( \mathbf{X}_{\mathcal{I}}^T 
      \mathbf{X}_{\mathcal{I}}\right)^{-1} \left[ \mathbf{X}_{\mathcal{I}}^T
        \mathbf{X}_r \tilde{\beta}_r + \frac{\lambda n}{2}
        \left( \sum_{\ell=1}^{r-1} W^{(\ell)} \ast \tilde{a}^{(\ell)} \right)
      _{\mathcal{I}}\right]
      \nonumber\\
      & = -\frac{2}{n \lambda } \mathbf{X}^T_{\ell}
      \left[ \mathbf{I} -  \mathbf{X}_{\mathcal{I}} 
        \left( \mathbf{X}_{\mathcal{I}}^T 
      \mathbf{X}_{\mathcal{I}}\right)^{-1}  \mathbf{X}_{\mathcal{I}}^T \right]
      \mathbf{X}_r \tilde{\beta}_r + 
      \mathbf{X}^T_{\ell} \mathbf{X}_{\mathcal{I}} \left( \mathbf{X}_{\mathcal{I}}^T 
      \mathbf{X}_{\mathcal{I}}\right)^{-1} 
      \left( \sum_{\ell=1}^{r-1} W^{(\ell)} \ast \tilde{a}^{(\ell)} \right)
      _{\mathcal{I}}
      \nonumber\\
      & = \mathbf{X}_{\ell}^T \left[ \mathbf{C}_{\mathcal{I}} - 
        \mathbf{O}_{\mathcal{I}}\left( \frac{2}{n\lambda} \mathbf{X}_r 
      \tilde{\beta}_r\right)\right].
      \label{eq:al}
    \end{align}

    Conditioning on $\mathbf{X}_{\mathcal{I}}$, we can decompose $\mathbf{X}_r$
    and $\mathbf{X}_{\ell}$ as
    \begin{align}
      & \mathbf{X}_r^T = \Sigma_{r \mathcal{I}} \left( 
      \Sigma_{\mathcal{I} \mathcal{I}} \right)^{-1} \mathbf{X}^T_{\mathcal{I}}
      + E^T_r,
      \label{eq:decomposition1}\\
      & \mathbf{X}_{\ell}^T = \Sigma_{\ell \mathcal{I}} \left( 
      \Sigma_{\mathcal{I} \mathcal{I}} \right)^{-1} \mathbf{X}^T_{\mathcal{I}}
      + E^T_{\ell},
      \nonumber
    \end{align}
    where $E_r \sim N \left( \mathbf{0}_n, \theta^{(r)}_r \mathbf{I}_{n\times n} \right)$ and 
    $E_{\ell} \sim N \left( \mathbf{0}_n, \theta^{(\ell)}_r\mathbf{I}_{n\times n} \right)$,
    and $\theta_r^{(\ell)}$ and $\theta_r^{(r)}$ are defined in Section 4.
    Then 
    \begin{align}
      \mathbf{X}^T_{\ell} \mathbf{O}_{\mathcal{I}} 
      = E_{\ell}^T \mathbf{O}_{\mathcal{I}} 
      \quad \text{  and  } \quad 
      \mathbf{O}_{\mathcal{I}} \mathbf{X}_r 
      = \mathbf{O}_{\mathcal{I}} E_r,
      \nonumber
    \end{align}
    and from \eqref{eq:al}
    \begin{align}
      \left( \tilde{a}^{(\ell)} \right)_\ell & = 
      E_{\ell}^T \left[ \mathbf{C}_{\mathcal{I}} - 
        \mathbf{O}_{\mathcal{I}}\left( \frac{2}{n\lambda} E_r 
      \tilde{\beta}_r\right)\right]
      + \Sigma_{\ell \mathcal{I}} \left( \Sigma_{\mathcal{I} \mathcal{I}} 
      \right)^{-1}
      \left( \sum_{\ell=1}^{r-1} W^{(\ell)}\ast \tilde{a}^{(\ell)} \right)
      _{\mathcal{I}} 
      \nonumber\\
      &:= R^{(\ell)} + F^{(\ell)}.
      \label{eq:decompofdual}
    \end{align}

    We first bound $\max_{\ell} \left|F^{(\ell)}\right|$.
    Note that 
    \begin{align}
      & \left \| \left( 
      \sum_{\ell=1}^{r-1} W^{(\ell)} \ast 
      \tilde{a}^{(\ell)}\right)_{\mathcal I} 
      \right \|_\infty = \left \| \left( 
      \sum_{\ell=J+1}^{r-1} W^{(\ell)} \ast 
      \tilde{a}^{(\ell)}\right)_{\mathcal I} 
      \right \|_\infty 
      = \max_{m\in \mathcal{I}} \left| \sum_{\ell=m}^{r-1}
      w_{\ell m} \left( \tilde{a}^{(\ell)} \right)_m\right| 
      \nonumber\\	
      \leq & \max_{m \in \mathcal{I}} \sum_{\ell = m}^{r-1} w_{\ell m} 
      \left|\left( \tilde{a}^{(\ell)} \right)_m\right|
      \leq \max_{m \in \mathcal{I}} \sum_{\ell = m}^{r-1} 
      \frac{1}{\left( \ell-m+1 \right)^2} \leq \sum_{k = 1}^\infty \frac{1}{k^2}= \frac{\pi^2}{6},
      \label{eq:boundingdualvariable}
    \end{align}
    where we used $\norm{\tilde{a}^{(\ell)}}_\infty \leq \norm{\tilde{a}^{(\ell)}}_2 \leq 1$.
    Therefore, by Assumption \ref{airrepre}, 
    \[
      \max_{1 \leq \ell \leq J} \left| 
      \Sigma_{\ell \mathcal{I}} \left( \Sigma_{\mathcal{I} \mathcal{I}} 
      \right)^{-1}\left( \sum_{\ell=1}^{r-1} 
      W^{(\ell)}\ast \tilde{a}^{(\ell)} \right)_{\mathcal{I}} \right| 
      \leq  1-\alpha.
    \]

    To give a bound on the random quantity $\left|R^{(\ell)}\right|$, we first state a general result that will be used
    multiple times later in the proof.
    \begin{lemma}
      Consider the term $E_j^T \eta$ where $\eta \in \real^n$ is a random vector 
      depending on $\mathbf{X}_{\mathcal I}$ and $\mathbf{X}_{r}$ and 
      $E_{j} \sim N \left( \mathbf{0}_n, 
      \theta_r^{(j)}\mathbf{I}_{n\times n} \right)$ for $j=1,\dots,J,r$. 
      If for some $\bar{Q}\geq 0$
      \[
        \Prob\left[ \Var\left( E_j^T\eta \Big| \mathbf{X}_{\mathcal{I}}, \mathbf{X}_r \right)
        \geq \bar{Q} \right] \leq \bar{p} 
      \]
      then for any $a > 0$, 
      \[
        \Prob\left[ \left| E_j^T\eta\right| \geq a \right] \leq 
        2 \exp\left( -\frac{a^2}{2 \bar{Q}} \right) + \bar{p} 
      \]
      \label{lem:probargu}
    \end{lemma}
    \begin{proof}
      Define the event 
      \[
        \bar{\mathcal{B}} = \left\{ \Var\left( E_j^T \eta \Big|
          \mathbf{X}_{\mathcal{I}}  \right) \geq \bar{Q} \right\}.
        \]

        Now for any $a$ and conditioned on $\mathbf{X}_{\mathcal{I}}$ and $\mathbf{X}_r$,
        \[
          \Prob \left[ E_j^T \eta \geq a \right] \leq 
          \Prob \left[ E_j^T \eta \geq a
          \Big| \bar{ \mathcal{B} }^c \right]
          + \Prob \left[ \bar{ \mathcal{B} } \right]
          \leq \Prob \left[ E_j^T \eta \geq a
          \Big| \bar{\mathcal{B}}^c \right]
          + \bar{p}.
        \]
        Conditioned on $\bar{\mathcal B}^c$, the variance of $E_j^T \eta$ is at most 
        $\bar{Q}$. 
        So by standard Gaussian tail bounds, we have
        \[
          \Prob \left[ E_j^T \eta \geq a \Big| 
          \bar{\mathcal B}^c\right] 
          = \E \left[ \Prob\left( E_j^T \eta \geq a \Big| 
          \mathbf{X}_{\mathcal{I}}, \mathbf{X}_r\right) \Big| \bar{\mathcal{B}}^c \right]
          \leq \E \left[ 2 \exp
            \left( -\frac{a^2}{2 \bar{Q}} \right)
          \Big| \bar{\mathcal{B}}^c \right] \leq 2 \exp
          \left( -\frac{a^2}{2 \bar{Q}} \right).
          \nonumber
        \]
      \end{proof}

      Then note that 
      $\Var\left( E_{i\ell} \right) = \theta_r^{(\ell)} \leq \theta_r$
      for $i=1,\dots,n$.
      Now conditioned on both $\mathbf{X}_{\mathcal{I}}$ and $\mathbf{X}_r$, 
      $R^{(\ell)}$ is zero-mean with variance at most 
      \begin{align}
        \Var&\left( R^{(\ell)} \Big| \mathbf{X}_{\mathcal{I}} \right) \nonumber\\ 
        \leq &\theta_r \norm{\mathbf{C}_{\mathcal{I}} - 
        \mathbf{O}_{\mathcal{I}}\left( \frac{2}{n\lambda} E_r 
        \tilde{\beta}_r\right)}^2_2 
        = \theta_r \left\{ \mathbf{C}_{\mathcal{I}}^T \mathbf{C}_{\mathcal{I}} +
        \norm{\mathbf{O}_{\mathcal{I}}\left( \frac{2}{n\lambda} 
        E_r\tilde{\beta}_r \right)}^2_2 \right\}
        \nonumber\\
        =& \theta_r \left\{ \frac{1}{n} 
        \left( \sum_{\ell=1}^{r-1} W^{(\ell)}\ast \tilde{a}^{(\ell)} 
        \right)^T_{\mathcal{I}} \left( \frac{1}{n} \mathbf{X}_{\mathcal{I}}^T
        \mathbf{X}_{\mathcal{I}}\right)^{-1} 
        \left(\sum_{\ell=1}^{r-1} W^{(\ell)}\ast \tilde{a}^{(\ell)} 
        \right)_{\mathcal{I}} + \frac{4\tilde{\beta}^2_r
        \norm{\mathbf{O}_{\mathcal{I}} E_r}^2_2}{n^2 \lambda^2} \right\}
        \nonumber\\
        :=& \theta_r M_n,
        \nonumber
      \end{align}
      where the first equality holds from Pythagorean identity.
      The next lemma bounds the random scaling $M_n$.

      \begin{lemma}
        For $\varepsilon \in \left( 0,\frac{1}{2} \right)$, denote
        \[
          \bar{M}_n\left( \varepsilon \right) := 
          \frac{3 \kappa^2 \pi^2 }{2} \frac{K}{n} +
          \frac{1}{\theta_r^{(r)} \left( n-K \right)\left( 1-\varepsilon \right)}
          +\frac{16}{n\lambda^2},
        \]
        then 
        \[
          \Prob\left[ M_n \geq \bar{M}_n\left( \varepsilon \right) \Big| \mathbf{X}_{\mathcal{I}}\right] \leq 7 
          \exp\left( -n \min\left\{ \frac{\alpha^2}{3 \theta_r^{(r)} \kappa^2 \pi^2 K} 
          ,\frac{\varepsilon^2}{4}\left( 1-\frac{K}{n}\right) \right\}\right).
        \]
        \label{lem:boundingvarA}
      \end{lemma}
      \begin{proof}
        See Appendix \ref{proof:boundingvarA}.
      \end{proof}

      Now by Lemma \ref{lem:probargu} and the union bound,
      \begin{align}
        \Prob\left[ \max_{1\leq\ell \leq J} \left|R^{(\ell)}\right| 
        \geq \alpha \right]
        \leq 2 J \exp\left( -\frac{\alpha^2}{ 2 \theta_r \bar{M}_n
      \left( \varepsilon \right)} \right) + 
      7 \exp\left( -c_3 n\right),
      \label{eq:condprob}
    \end{align}
    for some constant $c_3$ independent of $n$ and $J$.
    By the assumption that $\frac{K}{n} = o(1)$, we have that $\frac{K}{n} \leq 1-\varepsilon$ for $n$ large enough, thus 
    \[
      \bar{M}_n\left( \varepsilon \right) \leq 
      \frac{K}{n} \left( \frac{3 \kappa^2 \pi^2 }{2} +
      \frac{1}{K \theta_r^{(r)} \left( 1-\varepsilon \right)^2}
      +\frac{16}{K\lambda^2} \right)
      \leq  \frac{K}{n} \left( \frac{3 \kappa^2 \pi^2 }{2} +
      \frac{4}{K \theta_r^{(r)} }
      +\frac{16}{K\lambda^2} \right).
    \]
    For the exponential term in \eqref{eq:condprob}
    to have faster decaying rate than the $J$ term, we need
    \[
      \frac{n}{K \log J} > \frac{\theta_r}{\alpha^2}
      \left( 3\kappa^2 \pi^2 + \frac{8}{K \theta_r^{(r)}} 
      + \frac{32}{K \lambda^2} \right).
    \]
  \end{proof}

  \subsection{Proof of Property 2 in Theorem \ref{thm:row}}
  Next we study the $\ell_{\infty}$ error bound. 
  The following theorem gives an $\ell_\infty$ error bound of $\tilde{\beta}$.
  \begin{proof}
    \label{proof:thmerrorbound}
    Let $\delta = \tilde{\beta} - \beta^\ast = \tilde{\beta} - 
    \left( {L}^T \right)_{1:r,r}$ and 
    $\mathcal{W} = S {L}^T - \left( L \right)^{-1}$,
    then from \eqref{eq:ii} and the fact that $L^{-1}$ is 
    lower-triangular,
    \begin{align}
      \delta_{\mathcal{I}}  =& - \left( \mathbf{X}_{\mathcal{I}}^T 
      \mathbf{X}_{\mathcal{I}} \right)^{-1} \left[ \mathbf{X}_{\mathcal{I}}^T
        \mathbf{X}_r \tilde{\beta}_r 
        + \left( \mathbf{X}_{\mathcal{I}}^T 
        \mathbf{X}_{\mathcal{I}} \right) 
      \left( L \right)^T_{\mathcal{I},r} \right]
      - \frac{n \lambda}{2} \left( \mathbf{X}_{\mathcal{I}}^T 
      \mathbf{X}_{\mathcal{I}} \right)^{-1} 
      \left( \sum_{\ell=1}^{r-1} W^{(\ell)}\ast \tilde{a}^{(\ell)} \right)_
      {\mathcal{I}} 
      \nonumber\\
      =& - \left(\frac{1}{n} \mathbf{X}_{\mathcal{I}}^T 
      \mathbf{X}_{\mathcal{I}} \right)^{-1} \left[ \frac{1}{n}
        \mathbf{X}_{\mathcal{I}}^T
        \mathbf{X}_r \left( \delta_r + \beta^\ast_r \right)
        + \left(\frac{1}{n} \mathbf{X}_{\mathcal{I}}^T 
        \mathbf{X}_{\mathcal{I}} \right) 
      \left( L \right)^T_{\mathcal{I},r} \right]
      \nonumber\\
      -& \frac{\lambda}{2} \left( \frac{1}{n} \mathbf{X}_{\mathcal{I}}^T 
      \mathbf{X}_{\mathcal{I}} \right)^{-1} 
      \left( \sum_{\ell=1}^{r-1} W^{(\ell)}\ast \tilde{a}^{(\ell)} \right)_
      {\mathcal{I}} 
      \nonumber \\
      =& - \left(\mathbf{X}_{\mathcal{I}}^T 
      \mathbf{X}_{\mathcal{I}} \right)^{-1} 
      \mathbf{X}_{\mathcal{I}}^T \mathbf{X}_r \delta_r
      - \left(\frac{1}{n} \mathbf{X}_{\mathcal{I}}^T 
      \mathbf{X}_{\mathcal{I}} \right)^{-1} 
      \left( S {L}^T \right)_{\mathcal{I},r}
      \nonumber\\
      -& \frac{\lambda}{2} \left( \frac{1}{n} \mathbf{X}_{\mathcal{I}}^T 
      \mathbf{X}_{\mathcal{I}} \right)^{-1} 
      \left( \sum_{\ell=1}^{r-1} W^{(\ell)}\ast \tilde{a}^{(\ell)} \right)_
      {\mathcal{I}} 
      \nonumber\\
      =& - \left(\mathbf{X}_{\mathcal{I}}^T 
      \mathbf{X}_{\mathcal{I}} \right)^{-1} 
      \mathbf{X}_{\mathcal{I}}^T \mathbf{X}_r \delta_r
      - \left(\frac{1}{n} \mathbf{X}_{\mathcal{I}}^T 
      \mathbf{X}_{\mathcal{I}} \right)^{-1} 
      \mathcal{W}_{\mathcal{I},r}
      - \frac{\lambda}{2} \left( \frac{1}{n} \mathbf{X}_{\mathcal{I}}^T 
      \mathbf{X}_{\mathcal{I}} \right)^{-1} 
      \left( \sum_{\ell=1}^{r-1} W^{(\ell)}\ast \tilde{a}^{(\ell)} \right)_
      {\mathcal{I}}.
      \label{eq:deltaI}
    \end{align}

    From \eqref{eq:i} and the fact that $\left( L^{-1} \right)_{rr} = 
    \frac{1}{L_{rr}}$,
    \begin{align}
      &-\frac{1}{\tilde{\beta}_r} + \frac{1}{n} \mathbf{X}_r^T \mathbf{X}_r
      \delta_r + \frac{1}{n} \mathbf{X}_r^T\mathbf{X}_{\mathcal{I}} 
      \delta_{\mathcal{I}} + \frac{1}{n} \mathbf{X}_r^T\mathbf{X}_r \beta^\ast_r
      + \frac{1}{n} \mathbf{X}_r^T\mathbf{X}_{\mathcal{I}} \beta^\ast_{\mathcal{I}}
      \nonumber\\
      &=-\frac{1}{\tilde{\beta}_r} + \frac{1}{n} \mathbf{X}_r^T \mathbf{X}_r
      \delta_r + \frac{1}{n} \mathbf{X}_r^T\mathbf{X}_{\mathcal{I}} 
      \delta_{\mathcal{I}} + \left( S {L}^T \right)_{rr}
      \nonumber\\
      &=\left( L^{-1} \right)_{rr} - \frac{1}{\tilde{\beta}_r} + 
      \frac{1}{n} \mathbf{X}_r^T \mathbf{X}_r
      \delta_r + \frac{1}{n} \mathbf{X}_r^T\mathbf{X}_{\mathcal{I}} 
      \delta_{\mathcal{I}} + \mathcal{W}_{rr}
      \nonumber\\
      &=\frac{\delta_r}{L_{rr} \tilde{\beta}_r}+ 
      \frac{1}{n} \mathbf{X}_r^T \mathbf{X}_r
      \delta_r + \frac{1}{n} \mathbf{X}_r^T\mathbf{X}_{\mathcal{I}} 
      \delta_{\mathcal{I}} + \mathcal{W}_{rr} = 0.
      \label{eq:deltar} 
    \end{align}

    Plugging \eqref{eq:deltaI} into \eqref{eq:deltar}, we have
    \begin{align}
      \frac{\delta_r}{L_{rr} \tilde{\beta_r}} + \frac{1}{n} 
      \mathbf{X}_r^T \mathbf{O}_{\mathcal{I}} \mathbf{X}_r \delta_r = 
      \mathbf{X}_r^T\mathbf{X}_{\mathcal{I}}\left( \mathbf{X}_{\mathcal{I}}^T
      \mathbf{X}_{\mathcal{I}}\right)^{-1} \mathcal{W}_{\mathcal{I},r}
      +\frac{\lambda}{2} \mathbf{X}_r \mathbf{C}_{\mathcal{I}} - \mathcal{W}_{rr},
      \nonumber
    \end{align}
    which implies
    \begin{align}
      \delta_r = \left( \frac{1}{L_{rr}\tilde{\beta}_r} + 
      \frac{1}{n}\mathbf{X}_r^T\mathbf{O}_{\mathcal{I}}\mathbf{X}_r \right)^{-1}
      \left[  \mathbf{X}_r^T\mathbf{X}_{\mathcal{I}}\left( \mathbf{X}_{\mathcal{I}}^T
        \mathbf{X}_{\mathcal{I}}\right)^{-1} \mathcal{W}_{\mathcal{I},r}
        +\frac{\lambda}{2} \mathbf{X}_r \mathbf{C}_{\mathcal{I}} - 
      \mathcal{W}_{rr}\right].
      \nonumber
    \end{align}
    Since $L_{rr}> 0$ and $\tilde{\beta}_r >0$,
    \begin{align}
      |\delta_r| \leq& \left| \left( \frac{1}{L_{rr}\tilde{\beta}_r} + 
      \frac{1}{n}
      {\mathbf{X}_r^T\mathbf{O}_{\mathcal{I}}\mathbf{X}_r} \right)^{-1}\right|
      \left( \left| \mathbf{X}_r^T\mathbf{X}_{\mathcal{I}}
      \left( \mathbf{X}_{\mathcal{I}}^T\mathbf{X}_{\mathcal{I}}\right)^{-1} 
      \mathcal{W}_{\mathcal{I},r} \right|
      +\left| \frac{\lambda}{2} \mathbf{X}_r \mathbf{C}_{\mathcal{I}}\right| + 
      \left| \mathcal{W}_{rr}\right| \right)
      \nonumber\\
      \leq &
      \left| \left(\frac{1}{n}
      {\mathbf{X}_r^T\mathbf{O}_{\mathcal{I}}\mathbf{X}_r} \right)^{-1}\right|
      \left( \left| \mathbf{X}_r^T\mathbf{X}_{\mathcal{I}}
      \left( \mathbf{X}_{\mathcal{I}}^T\mathbf{X}_{\mathcal{I}}\right)^{-1} 
      \mathcal{W}_{\mathcal{I},r} \right|
      +\left| \frac{\lambda}{2} \mathbf{X}_r \mathbf{C}_{\mathcal{I}}\right| + 
      \left| \mathcal{W}_{rr}\right| \right).
      \nonumber
    \end{align}

    Now conditioned on $\mathbf{X}_{\mathcal{I}}$, by the decomposition 
    \eqref{eq:decomposition1}, $\left( \frac{1}{n} \mathbf{X}_r^T 
    \mathbf{O}_{\mathcal{I}} \mathbf{X}_r  \right)^{-1}= 
    \left( \frac{1}{n} E_r^T \mathbf{O}_{\mathcal{I}} E_r  \right)^{-1}= 
    \frac{n}{\norm{\mathbf{O}_{\mathcal{I}}E_r}_2^2}$.
    From Lemma \ref{lem:chisqconcentration}, it follows that
    \begin{align}
      \Prob\left[  \left( \frac{1}{n} \mathbf{X}_r^T 
        \mathbf{O}_{\mathcal{I}} \mathbf{X}_r  \right)^{-1} \geq  
      \frac{1}{\theta_r^{(r)}} \frac{n}{n-K} \frac{1}{1-\varepsilon} \right]
      \leq \exp \left( -\frac{1}{4}\left( n-K \right)\varepsilon^2 \right).
      \nonumber
    \end{align}

    Also, by Lemma \ref{lem:boundingXrCI},
    \begin{align}
      \Prob \left[ 
        \left| \mathbf{X}_r^T \mathbf{C}_{\mathcal{I}}\right| \geq 
      1 \right] \leq  
      2\exp\left( -\frac{n\alpha^2}{3 \theta_r^{(r)} \kappa^2 \pi^2 K} \right)
      + 2 \exp \left( -\frac{n}{2} \right).
      \nonumber
    \end{align}

    To deal with the rest of terms in \eqref{eq:deltar} that involve $\mathcal{W}$, 
    we introduce the following concentration inequality to control its
    element-wise infinity norm.
    \begin{lemma} \label{lem:concentration}
      Let $\mathcal{W} = S {L}^T - L^{-1}$.
      Under Assumptions \ref{abddsval} and \ref{asubgaussian},
      there exist constants $C_1, C_2,C_3>0$ such that
      for any $0 < t \leq 2 \kappa$,
      \[
        \Prob \left[ \norm{\mathcal W}_\infty > t \right]
        \leq 2p^2 \exp \left( -\frac{C_3nt^2}{\kappa^2} \right)
        + 4p \exp \left( - \frac{C_1nt}{\kappa^2} \right) 
        + 4p \exp \left( - C_2nt \right).
      \]
    \end{lemma}
    \begin{proof}
      See Appendix \ref{proof:concentration}.
    \end{proof}

    In terms of the event
    \[
      \mathcal{A} = \left\{ \norm{ \mathcal{W}}_\infty 
    \leq \lambda \right\},
  \]
  Lemma \ref{lem:concentration} states that
  \[
    \Prob\left[ \mathcal{A}^c \right] \leq  
    2p^2 \exp \left( -\frac{C_3n \lambda^2}{\kappa^2} \right) +
    4p \exp \left( - \frac{C_1n\lambda}{\kappa^2} \right)
    + 4p \exp \left( - C_2n \lambda \right).
  \]

  The next lemma shows that,
  on the event $\mathcal{A}$ and with the assumption that $\frac{\lambda^2}{n} = o(1)$, the term
  $\left| \mathbf{X}_r^T\mathbf{X}_{\mathcal{I}}
  \left( \mathbf{X}_{\mathcal{I}}^T\mathbf{X}_{\mathcal{I}}\right)^{-1} 
  \mathcal{W}_{\mathcal{I},r} \right|$ can be bounded by $\lambda$ with high probability.

  \begin{lemma}
    Using the general weigthing scheme \eqref{est:generalweight}, 
    we have
    \begin{align}
      \Prob \left[ \left| \mathbf{X}_r^T\mathbf{X}_{\mathcal{I}}
        \left( \mathbf{X}_{\mathcal{I}}^T\mathbf{X}_{\mathcal{I}}\right)^{-1} 
        \mathcal{W}_{\mathcal{I},r} \right| \geq 
      \lambda \Big| \mathcal{A} \right] \leq  
      2\exp\left( -\frac{2 n\alpha^2}{9 \theta_r^{(r)} \kappa^2 K \lambda^2} \right)
      + 2 \exp \left( -\frac{n}{2} \right).
      \nonumber
    \end{align}
    \label{lem:boundingXrWIr}
  \end{lemma}
  \begin{proof}
    Recall that by conditioning on $\mathbf{X}_{\mathcal{I}}$, the decomposition \eqref{eq:decomposition1} gives
    \begin{align}
      \mathbf{X}_r^T\mathbf{X}_{\mathcal{I}}
      \left( \mathbf{X}_{\mathcal{I}}^T\mathbf{X}_{\mathcal{I}}\right)^{-1} 
      \mathcal{W}_{\mathcal{I},r} = \Sigma_{r \mathcal{I}} \left( 
      \Sigma_{\mathcal{I} \mathcal{I}} \right)^{-1} \mathcal{W}_{\mathcal{I},r} 
      + E^T_r \mathbf{X}_{\mathcal{I}}
      \left( \mathbf{X}_{\mathcal{I}}^T\mathbf{X}_{\mathcal{I}}\right)^{-1} 
      \mathcal{W}_{\mathcal{I},r}.
      \nonumber
    \end{align}
    On the event $\mathcal{A}$, by \ref{airrepre} and 
    \eqref{eq:boundingdualvariable},
    \begin{align}
      \left| \Sigma_{r \mathcal{I}} \left( 
      \Sigma_{\mathcal{I} \mathcal{I}} \right)^{-1} \mathcal{W}_{\mathcal{I},r} 
      \right| \leq  \vertiii{ \Sigma_{r \mathcal{I}} \left( 
      \Sigma_{\mathcal{I} \mathcal{I}} \right)^{-1} }_\infty
      \norm{\mathcal{W}_{\mathcal{I},r}}_\infty \leq \lambda.
      \nonumber
    \end{align}

    Note that $\Var\left( E_{ir} \right) = \theta_r^{(r)}$ for $i=1,\dots,n$.
    Let
    $B^{(r)}:= E_r^T \mathbf{X}_{\mathcal{I}} \left( \mathbf{X}^T_{\mathcal{I}}
    \mathbf{X}_{\mathcal{I}}\right)^{-1} \mathcal{W}_{\mathcal{I},r}$, 
    then $B^{(r)}$ has mean zero and variance at most
    \[
      \Var\left( B^{(r)} \Big| \mathbf{X}_{\mathcal{I}} \right) = 
      \frac{\theta_r^{(r)}}{n} \mathcal{W}_{\mathcal{I},r}^T \left( \frac{1}{n}
      \mathbf{X}_{\mathcal{I}}^T \mathbf{X}_{\mathcal{I}}\right)^{-1} 
      \mathcal{W}_{\mathcal{I},r}
      \leq \frac{9\theta_r^{(r)} \kappa^2 K \lambda^2}{n},
    \]
    with probability greater than $1-2\exp\left( \frac{n}{2} \right)$. The result
    follows from Lemma \ref{lem:probargu}.
  \end{proof}

  Putting everything together and choosing the tuning parameter from 
  \eqref{theory:lambdar},
  with a union bound argument and 
  some algebra, we have shown that conditioned on $\mathbf{X}_{\mathcal{I}}$, 
  \begin{align}
    &\Prob \left[ |\delta_r| \geq  \frac{1}{\theta_r^{(r)}} \frac{n}{n-K} \frac{1}{1-\varepsilon}
    \frac{5}{2} \lambda \right] \leq  
    \Prob \left[ |\delta_r| \geq  \frac{5}{2\theta_r^{(r)}} \lambda \right]
    \leq 
    \Prob \left[ |\delta_r| \geq  \frac{5}{2\theta_r^{(r)}} \lambda \Big| \mathcal{A} \right] + 
    \Prob \left[ \mathcal{A}^c \right]
    \nonumber\\
    &\leq \exp \left( -\frac{1}{4n}\left( 1-\frac{K}{n} \right)\varepsilon^2 \right)+
    2\exp\left( -\frac{n\alpha^2}{3 \theta_r \kappa^2 \pi^2 K} \right)+
    2\exp\left( -\frac{2 n\alpha^2}{9 \theta_r \kappa^2 K \lambda^2} \right)+
    4 \exp \left( -\frac{n}{2} \right)
    \nonumber\\
    &+2p^2 \exp \left( -\frac{C_3n \lambda^2}{\kappa^2} \right)
    + 4p \exp \left( - \frac{C_1n\lambda}{\kappa^2} \right)
    + 4p \exp \left( - C_2n\lambda \right)
    \nonumber\\
    &\leq c_4 \exp\left( -c_5 n \right) + \frac{c_6}{p},
    \label{eq:boundingdeltar}
  \end{align}
  for some constants $c_4,c_5,c_6,x>0$ that do not depend on $n$ and $p$.

  We now consider a bound 
  for $\delta_{\mathcal{I}}$. Recall from \eqref{eq:deltaI} that
  \begin{align}
    \delta_{\mathcal{I}}= F_1 + F_2  
    \nonumber
  \end{align}
  where
  \begin{align}
    F_1 = & - \left(\mathbf{X}_{\mathcal{I}}^T 
    \mathbf{X}_{\mathcal{I}} \right)^{-1} 
    \mathbf{X}_{\mathcal{I}}^T \mathbf{X}_r \delta_r,
    \nonumber\\
    F_2 = & - \left(\frac{1}{n} \mathbf{X}_{\mathcal{I}}^T 
    \mathbf{X}_{\mathcal{I}} \right)^{-1} 
    \left( \mathcal{W}_{\mathcal{I},r} + \frac{\lambda}{2} \mathbf{D} \right)
    \quad \text{  with  } \quad \mathbf{D} = 
    \left( \sum_{\ell=1}^{r-1} W^{(\ell)}\ast \tilde{a}^{(\ell)} \right)_
    {\mathcal{I}}.
    \nonumber
  \end{align}
  An $\ell_\infty$ bound of $F_2$ is given by
  \begin{align}
    \norm{F_2}_\infty  \leq
    \norm{ \left(  \left(\frac{1}{n} \mathbf{X}_{\mathcal{I}}^T 
    \mathbf{X}_{\mathcal{I}} \right)^{-1} - 
    \left( \Sigma_{\mathcal{I}\mathcal{I}} \right)^{-1}
    \right) 
    \left( \mathcal{W}_{\mathcal{I},r} + \frac{\lambda}{2} \mathbf{D}\right)}_\infty
    + \norm{   \left( \Sigma_{\mathcal{I}\mathcal{I}} \right)^{-1}
    \left( \mathcal{W}_{\mathcal{I},r} + \frac{\lambda}{2} \mathbf{D}\right)}_\infty.
    \label{bound:F2}
  \end{align}
  On the event $\mathcal{A}$, by \eqref{eq:boundingdualvariable},
  \begin{align}
    &\norm{ \left( \Sigma_{\mathcal{I}\mathcal{I}} \right)^{-1}
    \left( \mathcal{W}_{\mathcal{I},r} + \frac{\lambda}{2} \mathbf{D}\right)}_\infty
    \leq \vertiii{\left( \Sigma_{\mathcal{I}\mathcal{I}} \right)^{-1}}_\infty
    \left( \norm{\mathcal{W}_{\mathcal{I},r}}_\infty + 
    \frac{\lambda}{2} \norm{\mathbf{D}}_\infty\right)
    \nonumber\\
    &\leq \vertiii{\left( \Sigma_{\mathcal{I}\mathcal{I}} \right)^{-1}}_\infty
    \left( 1+ \frac{\pi^2}{12} \right) \lambda 
    \leq 2 \lambda \vertiii{\left( \Sigma_{\mathcal{I}\mathcal{I}} \right)^{-1/2}}^2_\infty
    \nonumber 
  \end{align}

  To deal with the first term in \eqref{bound:F2}, note that $X_{\mathcal{I}} = 
  W_{\mathcal{I}} \left( \Sigma_{\mathcal{I}\mathcal{I}} \right)^{1/2}$,
  where $W_{\mathcal{I}} \in \real^{n \times K}$ is a standard Gaussian random 
  matrix, i.e., $\left(W_{\mathcal{I}}\right)_{ij} \sim N(0,1)$. Thus we can write
  it as
  \begin{align}
    \norm{ \left( \Sigma_{\mathcal{I}\mathcal{I}} \right)^{-1/2} 
    \left[ \left( \frac{1}{n} W_{\mathcal{I}}^T W_{\mathcal{I}} \right)^{-1} -
    \mathbf{I}_{K}\right] 
    \left( \Sigma_{\mathcal{I}\mathcal{I}} \right)^{-1/2} 
    \left( \mathcal{W}_{\mathcal{I},r} + \frac{\lambda}{2} \mathbf{D}\right)}_\infty
    \leq\vertiii{\left(\Sigma_{\mathcal{I}\mathcal{I}}\right)^{-1/2}}_\infty
    G,
    \nonumber
  \end{align}
  where
  \begin{align}
    G = 
    \norm{ \left[ \left( \frac{1}{n} W_{\mathcal{I}}^T W_{\mathcal{I}} \right)^{-1} -
    \mathbf{I}_{K}\right] 
    \left( \Sigma_{\mathcal{I}\mathcal{I}} \right)^{-1/2} 
    \left( \mathcal{W}_{\mathcal{I},r} + \frac{\lambda}{2} \mathbf{D}\right)}_\infty.
    \nonumber
  \end{align}

  By Lemma 5 in \cite{wainwright2009sharp}, we have, for some constant $c_7>0$.
  \begin{align}
    \Prob \left[ G \geq 
      \norm{\left( \Sigma_{\mathcal{I}\mathcal{I}} \right)^{-1/2} 
      \left( \mathcal{W}_{\mathcal{I},r} + \frac{\lambda}{2} \mathbf{D}\right)}_\infty
    \Big| \mathbf{X}_{\mathcal{I}} \right]
    \leq 4 \exp \left( -c_7 \min\left\{ K, 
    \log J \right\} \right)
    \nonumber
  \end{align}

  Note that conditioning on $\mathcal{A}$, 
  $\norm{\left( \Sigma_{\mathcal{I}\mathcal{I}} \right)^{-1/2} 
  \left( \mathcal{W}_{\mathcal{I},r} + \frac{\lambda}{2} \mathbf{D}\right)}_\infty$ 
  is upper bounded by \newline
  $2 \lambda \vertiii{\left( \Sigma_{\mathcal{I}\mathcal{I}} \right)^{-1/2}}_\infty$.
  Thus,
  \begin{align}
    \Prob \left[ G \geq 2\lambda
      \vertiii{\left( \Sigma_{\mathcal{I}\mathcal{I}} \right)^{-1/2}}^2_\infty
    \Big| \mathcal{A} \right]
    \leq 4 \exp \left( -c_7 \min\left\{ K, 
    \log J \right\} \right),
    \nonumber
  \end{align}
  and
  \begin{align}
    &\Prob \left[ \norm{F_2}_\infty \geq  4 \lambda \vertiii{\left( 
    \Sigma_{\mathcal{I}\mathcal{I}}\right)^{-1/2}}^2_\infty \right]
    \leq \Prob \left[ \norm{F_2}_\infty \geq  4\lambda \vertiii{\left( 
      \Sigma_{\mathcal{I}\mathcal{I}}\right)^{-1/2}}^2_\infty 
    \Big| \mathcal{A}\right] + \Prob \left[ \mathcal{A}^c \right]
    \nonumber\\
    &\leq 4 \exp \left( -c_7 \min\left\{ K, \log J \right\} \right) + \frac{c_6}{p}.
    \label{eq:probF2}
  \end{align}

  Turning to $F_1$, conditioned on $\mathbf{X}_{\mathcal{I}}$, by decomposition
  \eqref{eq:decomposition1}, we have that
  \begin{align}
    \norm{ F_1 }_\infty \leq \norm{\left( \Sigma_{\mathcal{I} \mathcal{I}} \right)^{-1} 
    \Sigma_{\mathcal{I} r}}_\infty |\delta_r| +
    \norm{\left(\mathbf{X}_{\mathcal{I}}^T 
    \mathbf{X}_{\mathcal{I}} \right)^{-1} 
    \mathbf{X}_{\mathcal{I}}^T E_r\delta_r}_\infty.
    \nonumber
  \end{align}

  By \eqref{eq:boundingdeltar} and \ref{airrepre},
  \begin{align}
    \Prob\left[  \norm{\left( \Sigma_{\mathcal{I} \mathcal{I}} \right)^{-1} 
      \Sigma_{\mathcal{I} r}}_\infty |\delta_r| \geq \frac{5}{2\theta_r^{(r)}} \lambda
    \right] \leq c_4 \exp\left( -c_5 n \right) + \frac{c_6}{p}.
    \nonumber
  \end{align}
  Consider each coordinate $j \in \mathcal{I}$ of the random term whose variance is bounded by
  \begin{align}
    \Var \left[ \mathbf{e}_j^T \left(\mathbf{X}_{\mathcal{I}}^T 
      \mathbf{X}_{\mathcal{I}} \right)^{-1} 
    \mathbf{X}_{\mathcal{I}}^T E_r\delta_r \Big| \mathbf{X}_{\mathcal{I}}\right] 
    \leq \theta_r \vertiii{\left( \frac{1}{n} \mathbf{X}_{\mathcal{I}}^T
    \mathbf{X}_{\mathcal{I}} \right)^{-1}}_2 \frac{\delta_r^2}{n}.
    \nonumber
  \end{align}
  By Lemma \ref{lem:mineigenvalue} and \eqref{eq:boundingdeltar},
  \begin{align}
    \Prob\left[  \Var \left[ \mathbf{e}_j^T \left(\mathbf{X}_{\mathcal{I}}^T 
      \mathbf{X}_{\mathcal{I}} \right)^{-1} 
    \mathbf{X}_{\mathcal{I}}^T E_r\delta_r \Big| \mathbf{X}_{\mathcal{I}}\right] 
  \geq \frac{235}{4} \frac{\kappa^2}{\theta_r} \frac{\lambda^2}{n} \right] 
  \leq 2 \exp\left( -\frac{n}{2}  \right) + c_4 \exp\left( -c_5 n \right) + \frac{c_6}{p}.
  \nonumber
\end{align}
Thus by Lemma \ref{lem:probargu},
\begin{align}
  \Prob\left[ \norm{ \left(\mathbf{X}_{\mathcal{I}}^T \mathbf{X}_{\mathcal{I}} \right)^{-1} 
  \mathbf{X}_{\mathcal{I}}^T E_r\delta_r}_\infty \geq \frac{5}{2\theta_r^{(r)}} \lambda  \right] 
  \leq &2 \exp\left( - \frac{n}{18 \theta_r \kappa^2} \right) + 2 \exp\left( -\frac{n}{2}  \right) \nonumber\\
  + &c_4 \exp\left( -c_5 n \right) + \frac{c_6}{p},
  \nonumber
\end{align}
and
\begin{align}
  \Prob\left[ \norm{F_1}_\infty \geq \frac{5}{\theta_r^{(r)}} \lambda  \right] 
  \leq 2 \exp\left( - \frac{n}{18 \theta_r \kappa^2} \right) + 
  2 \exp\left( -\frac{n}{2}  \right) + c_4 \exp\left( -c_5 n \right) + \frac{c_6}{p}.
  \nonumber
\end{align}

Combining with \eqref{eq:boundingdeltar} and \eqref{eq:probF2}, we have
\begin{align}
  \Prob\left[ \norm{\delta}_\infty \geq 4 \lambda 
    \vertiii{\left(\Sigma_{\mathcal{I}\mathcal{I}}\right)^{-1/2}}^2_\infty +
  \frac{5}{\theta_r^{(r)}} \lambda \right] \leq c_8 \exp\left( -c_9 n \right) + 2 \frac{c_6}{p} +
  4 \exp \left( -c_7 \min\left\{ K, \log J \right\} \right),
  \nonumber
\end{align}
for some constants $c_8,c_9 > 0$ that do not depend on $n$ and $J$.
\end{proof}

\subsection{Proof of Property 3 in Theorem \ref{thm:row}}
Finally we establish a $\beta_{\text{min}}$ condition, which, combined
with the $\ell_\infty$ rate, gives the other direction of the support recovery, 
i.e., $K(\hat{\beta}) \geq K$.

By the triangle inequality
\begin{align}
  \left| \tilde{\beta}_j\right| \geq \left|\beta_j\right| - 
  \left| \tilde{\beta}_j - \beta_j\right|.
  \nonumber
\end{align}
So if we have 
\begin{align}
  \max_{j \geq J+1}\left\{ \left|\beta_j\right| - 
    \left| \tilde{\beta}_j - \beta_j\right| \right\} > 0,
    \nonumber
  \end{align}
  then $K(\tilde{\beta}) \geq K$.

  \section{Proof of Theorem \ref{thm:matrix}} \label{online:proofofmatrix}
  \begin{proof}
    The overall proof techniques are the same as the proof of Theorem \ref{thm:row}.
    The first part of the theorem holds if 
    $\max_{2\leq r \leq p} \max_{1\leq \ell \leq J_r} |\tilde{a}^{(r\ell)}|<1 $.
    Now for each $r = 2,\dots,p$ we proceed with the same primal-dual 
    witness procedure and end up with the same decomposition \eqref{eq:decompofdual}.

    Assumption \ref{airrepre} ensures that 
    $\max_{2\leq r \leq p} \max_{1\leq \ell \leq J_r}
    |F^{(r\ell)}|\leq 1-\alpha$.
    Following the same line of proof to deal with random term $R^{(r\ell)}$, 
    we have that $R^{(r\ell)}$ is zero-mean Gaussian with conditional
    variance bounded above by the scaling
    \begin{align}
      \theta_r \bar{M}^{(r)}_n\left( \varepsilon \right) &= 
      \frac{3 \kappa^2 \pi^2 \theta_r}{2} \frac{K_r^\ast}{n} +
      \frac{\theta_r}{\theta^{(r)}_r}
      \frac{1}{\left( n-K_r^\ast \right)\left( 1-\varepsilon \right)}
      +\frac{16\theta_r}{n\lambda^2}
      \nonumber\\
      \leq& 
      \frac{3 \kappa^2 \pi^2 \theta_r}{2} \left( \frac{K}{n} +
      \frac{\kappa^2}{n\theta_r^{(r)}\left( 1-\varepsilon \right)^2}
      +\frac{16}{n\lambda^2}\right),
      \nonumber
    \end{align}
    for $\varepsilon \in \left( 0,\frac{1}{2} \right)$ with high probability,
    where we use the fact that 
    $K = o(n)$ implies that $\frac{K}{n} \leq \varepsilon$ for $n$ large.
    And
    \[
      \Prob\left[ \left|R^{(r\ell)}\right| 
      \geq \alpha \right]
      \leq 2 \exp\left( -\frac{\alpha^2}{ 2 \theta_r \bar{M}^{(r)}_n
    \left( \varepsilon \right)} \right) + 
    7 \exp\left( -c_3 n\right).
  \]
  Thus,
  \begin{align}
    \Prob\left[ \max_{2\leq r\leq p} \max_{1\leq \ell \leq J_r}
    \left|R^{(r\ell)}\right| \geq \alpha \right]
    &\leq 2 \sum_{r=2}^p J_r \exp\left( -\frac{\alpha^2}
    { 2 \theta_r \bar{M}^{(r)}_n
  \left( \varepsilon \right)} \right) + 
  7 \sum_{r=2}^p J_r \exp\left( -c_3 n\right)
  \nonumber\\
  &\leq p^2 \exp\left( -\frac{\alpha^2}{3 \kappa^2 \pi^2 \theta \frac{K}{n}
  + \frac{8 \theta \kappa^2}{n}
  +\frac{32\theta}{n\lambda^2}} \right) + 
  \frac{7}{2} p^2 \exp\left( -c_3 n\right).
  \nonumber
\end{align}

For the exponential term to decay faster than $p^2$, we need
\begin{align}
  \frac{n}{\log p} > \max\left\{ \frac{2}{\alpha^2}
  \left( 3\kappa^2 \pi^2 \theta K + 8 \kappa^2 \theta
  + \frac{32\theta}{\lambda^2} \right), \frac{2}{c_3} \right\}.
  \nonumber
\end{align}
\end{proof}

\section{Proof of Theorem \ref{thm:Omegabounds}} \label{online:proofofrates}
\begin{lemma} \label{cor:L2Fbounds}
  Using the notation and conditions in Theorem \ref{thm:Omegabounds}, the following deviation bounds hold with high probability:
  \begin{align}
    &\vertiii{\hat{L} - L}_\infty \leq
    \zeta_\Gamma \left( K + 1 \right) \sqrt{\frac{\log p}{n}},
    \nonumber\\
    &\vertiii{\hat{L} - L}_1 \leq
    \zeta_\Gamma \left( K + 1 \right) \sqrt{\frac{\log p}{n}},
    \nonumber\\
    &\vertiii{\hat{L} - L}_2 \leq
    \zeta_\Gamma \left( K + 1 \right) \sqrt{\frac{\log p}{n}},
    \nonumber\\
    &\norm{\hat{L} - L}_F \leq
    \zeta_\Gamma \sqrt{\frac{\left( s + p \right) \log p}{n}}.
    \nonumber
  \end{align}
\end{lemma}
\begin{proof}
  By Theorem \ref{thm:matrix}, with high probability, the support of $\hat{L}$ is contained in the
  true support and
  \begin{align}
    \norm{\hat{L} - L}_\infty \leq \zeta_\Gamma \sqrt{\frac{\log p}{n}}.
    \nonumber
  \end{align}
  Note that 
  \begin{align}
    &\vertiii{\hat{L} - L}_\infty = \max_{2\leq r \leq p} \sum_{c=1}^r
    \left|\hat{L}_{rc} - L_{rc}\right|
    \leq \max_{2\leq r \leq p} \left( K_r+1 \right) \norm{\hat{L} - L}_\infty
    \leq \left( K+1 \right) \norm{\hat{L} - L}_\infty.
    \nonumber
  \end{align}
  Denote $D = \max_{1 \leq c \leq p-1} D_c$ where 
  $D_c = \left|\left\{ r = c,\dots,p: L_{rc} \neq 0 \right\} \right|$.
  Observing that $D \leq K$, we have
  \begin{align}
    \vertiii{\hat{L} - L}_1 &= \max_{1\leq c \leq p-1} \sum_{r=1}^c
    \left|\hat{L}_{rc} - L_{rc}\right|
    \leq \max_{1\leq c \leq p-1} 
    \left( D_c+1 \right) \norm{\hat{L} - L}_\infty
    \nonumber\\
    &\leq \left( D+1 \right) \norm{\hat{L} - L}_\infty
    \leq \left( K+1 \right) \norm{\hat{L} - L}_\infty.
    \nonumber
  \end{align}

  By H\"{o}lder's inequality
  \begin{align}
    \vertiii{\hat{L} - L}_2 \leq \sqrt{ \vertiii{\hat{L} - L}_1
    \vertiii{\hat{L} - L}_\infty}.
    \nonumber	
  \end{align}

  Finally for Frobenius norm,
  \begin{align}
    \norm{\hat{L} - L}_F^2 = \sum_{r=2}^p \sum_{ c = J_r+1}^r
    \left( \hat{L}_{rc} - L_{rc} \right)^2
    \leq \sum_{r=2}^p \sum_{ c = J_r+1}^r
    \norm{\hat{L} - L}_\infty^2 
    \leq \zeta_\Gamma^2 \left( \sum_r K_r + p \right) \frac{\log p}{n}.
    \nonumber
  \end{align}
\end{proof}

\begin{proof}[of Theorem \ref{thm:Omegabounds}]
  First note that 
  \begin{align}
    \hat{L}^T \hat{L} - L^TL &= \left( \hat{L} - L \right)^T \left( \hat{L} - L \right) + \hat{L}^T L + L^T \hat{L} - 2 L^TL \nonumber\\
    &= \left( \hat{L} - L \right)^T \left( \hat{L} - L \right) + \left( \hat{L} - L \right)^T L + L^T \left( \hat{L} - L \right).
    \nonumber
  \end{align}
  Thus,
  \begin{align}
    \norm{\hat{L}^T\hat{L} - {L}^T L}_\infty &\leq
    \vertiii{\hat{L}- L}_\infty \norm{\hat{L} - L}_\infty + 
    2\vertiii{L}_\infty \norm{\hat{L} - L}_\infty,
    \nonumber
    \\
    \vertiii{\hat{L}^T\hat{L} - {L}^T L}_1 &= 
    \vertiii{\hat{L}^T\hat{L} - {L}^T L}_\infty \leq
    2 \vertiii{L}_\infty \vertiii{\hat{L} - L}_\infty + 
    \vertiii{\hat{L} - L}_\infty^2.
    \nonumber
  \end{align}
  By H\"{o}lder's inequality
  \begin{align}
    \vertiii{\hat{L}^T\hat{L} - L^TL}_2 \leq \sqrt{ \vertiii{\hat{L}^T\hat{L} - L^TL}_1
    \vertiii{\hat{L}^T\hat{L} - L^TL}_\infty}.
    \nonumber	
  \end{align}
  Finally, for Frobenius norm, observe that 
  \begin{align}
    \norm{L^T \left( \hat{L} - L \right)}_F &= \norm{\vec \left( L^T \left( \hat{L} - L \right) \right) }_2
    = \norm{\left( I_p \otimes L^T \right) \vec\left( \hat{L} - L \right)}_2
    \nonumber \\
    &\leq \vertiii{I_p \otimes L^T}_2 \norm{\hat{L} - L}_F = \vertiii{L}_2 \norm{\hat{L} - L}_F.
    \nonumber
  \end{align}
  Applying the same strategy to $\norm{\left( \hat{L} - L \right)\left( \hat{L} - L \right)}_F$,
  we have
  \begin{align}
    \norm{\hat{L}^T\hat{L} - L^TL}_F \leq \left( \vertiii{\hat{L} - L}_2 + 2\vertiii{L}_2 \right) \norm{\hat{L} - L}_F,
    \nonumber
  \end{align}
  then the results follow from Corollary \ref{cor:L2Fbounds}.
\end{proof}

\section{Proof of Theorem \ref{thm:Fbound}}\label{proof:Fbound}
\begin{proof}
  We adapt the proof technique of \cite{rothman2008sparse}. Let \begin{align}
    G(\Delta) = & -2 \log\det\left(L + \Delta\right) + \operatorname{tr}\left(S \left( L + \Delta \right)^T \left( L+ \Delta \right)\right)+ \lambda\left\|\left(\Delta+L\right)\right\|_{2,1}^\ast \nonumber\\
    & + 2\log\det L-\operatorname{tr}\left(S L^T L\right)- \lambda\left\|L\right\|_{2,1}^\ast,
    \label{eq:Gdelta}
  \end{align}
  where $L$ is the inverse of the Cholesky factor of the true covariance matrix, and the penalty is defined above as
  \begin{align}
    \|L \|^{\ast}_{2,1}=\sum_{r = 2}^p \sum_{\ell=1}^{r-1}\sqrt{\sum_{m=1}^{\ell}w_{\ell m}^2L_{rm}^2}.
    \nonumber
  \end{align}
  Since the estimator $\hat{L}$ is defined as
  \begin{align}
    \hat{L} = \argmin_{L_{jk} = 0: j<k} \left\{ -2 \log\det L + \tr\left( SL^T L \right) + \lambda \norm{L}_{2,1}^\ast \right\},
    \nonumber
  \end{align}
  it follows that $G(\Delta)$ is minimized at $\hat{\Delta}=\hat{L} - L$.
  %We suppress the dependence on $\lambda$ in $\hat{L}$ and $\hat{\Delta}$.
  Consider the value of $G(\Delta)$ on the set defined as
  \begin{align}
    \Theta_n(M)=\left\{\Delta:\,\,\Delta_{jk} = 0 \text{ for all}\,\, k > j,\,\,\,\, \left( \Delta + L \right)_{jj} > 0 \,\, \text{for all}\,\, j \,\,, \left\|\Delta\right\|_F=Mr_n\right\},
    \nonumber
  \end{align}
  where $M > 0$ and
  \begin{align}
    r_n=\sqrt{\frac{\left( \sum_{r=2}^p K_r + p \right) \log p}{n}}.
    \nonumber
  \end{align}
  The assumed scaling implies that $r_n\to 0$.
  We aim at showing that $\inf\left\{G(\Delta):\,\,\Delta\in \Theta_n(M)\right\}>0$. If it holds, then the convexity
  of $G\left( \Delta \right)$ and the fact that $G(\hat{\Delta})\leq G(\mathbf{0})=0$
  implies
  \begin{align}
    \|\hat{\Delta}\|_F = 
    \|\hat{L} - L\|_F\leq Mr_n.
    \nonumber
  \end{align}

  We start with analyzing the logarithm terms in \eqref{eq:Gdelta}. First let $f(t)=\log\det(L+t\Delta)$.
  Using a Taylor expansion of $f(t)$ at $t = 0$ with 
  $f'(t)=\operatorname{tr}[(L+t\Delta)^{-1}\Delta]$ and 
  $f''(t)=-{\operatorname{vec}\Delta}^T(L+t\Delta)^{-1}\otimes (L+t\Delta)^{-1}\operatorname{vec}\Delta$,
  we have
  \begin{align}
    &\log\det(L+\Delta)-\log\det(L)
    \nonumber\\
    =&\operatorname{tr}(L^{-1} \Delta) -
    (\operatorname{vec}\Delta)^T\left[\int_0^1(1-\nu)(L+\nu\Delta)^{-1}\otimes (L+\nu\Delta)^{-1}d\nu\right](\operatorname{vec}\Delta).
    \nonumber
  \end{align}
  The trace term in \eqref{eq:Gdelta} can be written as
  \begin{align}
    \tr\left( S\left( L+\Delta \right)^T \left( L+ \Delta \right) \right) - \tr\left( SL^TL \right) &= \tr\left( SL^T\Delta + S\Delta^TL + S\Delta^T\Delta \right) \nonumber\\
    & = 2\tr\left( SL^T\Delta \right) + \tr\left( S\Delta^T\Delta \right) \nonumber\\
    & \geq 2 \tr\left( SL^T\Delta \right),
    \nonumber
  \end{align}
  where the last inequality comes from the fact that the sample covariance matrix $S$ is positive semidefinite.
  Combining these with \eqref{eq:Gdelta} gives
  \begin{align}
    G(\Delta) \geq & 2 (\operatorname{vec}\Delta)^T \left[\int_0^1(1-\nu)(L+\nu\Delta)^{-1}\otimes (L+\nu\Delta)^{-1}d\nu\right](\operatorname{vec}\Delta) \nonumber\\
    & + 2\operatorname{tr}[(SL^T - L^{-1})\Delta] + \lambda\left(\left\|L + \Delta\right\|_{2,1}^\ast-\left\|L\right\|_{2,1}^\ast\right) \nonumber\\
    \equiv &(a)+(b)+(c) \label{eq:Gdeltanew}.
  \end{align}
  The integral term $(a)$ above has a positive lower bound.
  Recalling that \newline $\sigma_{\text{min}}(M)=\min_{\|x\|=1}x^TMx$ is a concave function of $M$ 
  (the minimum of linear functions of $M$ is concave), we have
  \begin{align}
    (a)&= 2 \|\operatorname{vec}\Delta\|^2 \frac{{\operatorname{vec}\Delta}^T}{\|\operatorname{vec}\Delta\|}\left[\int_0^1(1-\nu)(L+\nu\Delta)^{-1}\otimes (L+\nu\Delta)^{-1}d\nu\right]\frac{\operatorname{vec}\Delta}{\|\operatorname{vec}\Delta\|} \nonumber\\
    &\geq 2\|\Delta\|_F^2 \sigma_{\text{min}}\left[\int_0^1(1-\nu)(L+\nu\Delta)^{-1}\otimes (L+\nu\Delta)^{-1}d\nu\right] \nonumber\\
    &\geq 2\|\Delta\|_F^2 \left[\int_0^1(1-\nu)\sigma_{\text{min}}\left( (L+\nu\Delta)^{-1}\otimes (L+\nu\Delta)^{-1} \right)d\nu\right] \nonumber\\
    &\geq 2\|\Delta\|_F^2 \int_0^1(1-\nu)\sigma_{\text{min}}^2(L+\nu\Delta)^{-1}d\nu \nonumber\\
    &\geq \|\Delta\|_F^2\min_{0\leq\nu\leq1}\sigma_{\text{min}}^2(L+\nu\Delta)^{-1} \nonumber\\
    &\geq \|\Delta\|_F^2\min\left\{\sigma_{\text{min}}^2(L+\tilde{\Delta})^{-1}:\,\,\|\tilde{\Delta}\|_F\leq Mr_n\right\}.
    \label{eq:terma}
  \end{align}
  The second inequality uses Jenson's inequality of the concave function $\sigma_{\text{min}}(\cdot)$,
  and the third inequality uses the fact that $\sigma_{\min}\left( A \otimes A \right) = \sigma_{\min}\left( A \right)^2$ for any positive (semi)definite matrix $A$.
  Using triangle inequality on the matrix operator norm, we have
  \begin{align}
    \sigma_{\text{min}}^2(L+\tilde{\Delta})^{-1}=\sigma_{\text{max}}^{-2}(L+\tilde{\Delta})\geq \left(\vertiii{L}_2+\vertiii{\tilde{\Delta}}_2\right)^{-2}
    \geq\frac{1}{2\vertiii{L}_2^2}\geq \frac{\kappa^2}{2},
    \nonumber
  \end{align}
  where the second inequality holds with high probability since
  $\vertiii{\tilde{\Delta}}_2 \leq \|\tilde{\Delta}\|_F \le Mr_n \le \vertiii{L}_2$ as $r_n \rightarrow 0$ and the last inequality follows from Assumption \ref{abddsval}.
  This gives the lower bound for the first term in (\ref{eq:Gdeltanew}):
  \begin{align}
    (a) \geq \frac{1}{2}\kappa^2\|\Delta\|_F^2=\frac{1}{2}\kappa^2 M^2r_n^2.
    \label{eq:boundona}
  \end{align}
  To deal with $(b)$, we start by recalling some notation. We let $\mathcal S = \left\{ \left( r,j \right): L_{rj} \neq 0 \right\}$
  denote the support of $L$,
  and $s = \sum_{r= 2}^p K_r$ be the number of non-zero off-diagonal elements. We also define
  \begin{align}
    \norm{L}_{2,1} = \sum_{r=2}^p \sum_{\ell = 1}^{r-1} w_{\ell\ell} |L_{r\ell}|
    = \sum_{r=2}^p \sum_{\ell = 1}^{r-1} |L_{r\ell}|,
    \nonumber
  \end{align}
  where the last equality holds since $w_{\ell\ell} = 1$ by \eqref{est:generalweight}.
  Then, by the Cauchy-Schwarz inequality,
  \begin{align}
    \left|\operatorname{tr}[(SL^T - L^{-1})\Delta]\right| 
    &= \left| \sum_{r=1}^p \sum_{j = 1}^r \left( SL^T - L^{-1} \right)_{rj} \Delta _{rj}\right|
    \nonumber\\
    &\leq \left|\sum_{r = 1}^p \sum_{j \in \mathcal{I}_r}(SL^T - L^{-1})_{rj} \Delta_{rj}\right| +
    \left|\sum_{r= 1}^p \sum_{j \notin \mathcal{I}_r} (SL^T - L^{-1})_{rj} \Delta_{rj}\right| 
    \nonumber \\
    &\leq \sqrt{s + p} \norm{SL^T - L^{-1}}_\infty \norm{\Delta_{\mathcal S}}_F + \norm{SL^T - L^{-1}}_\infty \norm{\Delta_{\mathcal S^c}}_{2,1}
    \nonumber \\
    &\leq C_1 \sqrt{s + p} \sqrt{\frac{\log p}{n}} \norm{\Delta_{\mathcal S}}_F + C_1 \sqrt{\frac{\log p}{n}} \norm{\Delta_{\mathcal S^c}}_{2,1},
    \label{eq:boundonb}
  \end{align}
  where the last inequality comes from Lemma \ref{lem:concentration} with probability tending to 1.
  To bound the penalty terms, we note that 
  \begin{align}
    &\norm{L + \Delta}_{2,1}^\ast - \norm{L}_{2,1}^\ast
    \nonumber\\
    =& \sum_{r= 2}^p \sum_{\ell = 1}^{r-1}\sqrt{\sum_{m = 1}^\ell w_{\ell m}^2 (L_{rm} + \Delta_{rm})^2} - \norm{L_{\mathcal S}}_{2,1}^\ast
    \nonumber\\
    =& \sum_{r=2}^p\sum_{\ell = 1}^{r-1} \sqrt{\sum_{m: (r,m) \in \mathcal S} w_{\ell m}^2 (L_{rm} + \Delta_{rm})^2 + \sum_{m: (r,m) \notin \mathcal S}w_{\ell m}^2 (L_{rm} + \Delta_{rm})^2}
    - \norm{L_{\mathcal S}}_{2,1}^\ast
    \nonumber\\
    \geq& \sum_{r=2}^p \sum_{\ell = 1}^{r-1} \sqrt{\sum_{m:(r,m)\in \mathcal S} w_{\ell m}^2 (L_{rm} + \Delta_{rm})^2} + \sum_{r=2}^p \sum_{\ell: (r,\ell) \notin \mathcal S} |L_{r\ell} + \Delta_{r\ell}| - \norm{L_{\mathcal S}}_{2,1}^\ast
    \nonumber\\
    =& \norm{L_{\mathcal S} + \Delta_{\mathcal S}}_{2,1}^\ast + \norm{L_{\mathcal S^c} + \Delta_{\mathcal S^c}}_{2,1} - \norm{L_{\mathcal S}}_{2,1}^\ast
    \nonumber\\
    =& \norm{L_{\mathcal S} + \Delta_{\mathcal S}}_{2,1}^\ast + \norm{\Delta_{\mathcal S^c}}_{2,1} - \norm{L_{\mathcal S}}_{2,1}^\ast
    \nonumber\\
    \geq& \norm{\Delta_{\mathcal S^c}}_{2,1} - \norm{\Delta_{\mathcal S}}_{2,1}^\ast,
    \nonumber
  \end{align}
  where the last inequality comes from triangle inequality.
  To give an upper bound on $\norm{L_{\mathcal S}}_{2,1}^\ast$, we observe that
  $2\lambda b \leq a\lambda^2+b^2/a$ holds for any $a>0$, and obtain
  \begin{align}
    2\lambda \left\|\Delta_{\mathcal S}\right\|_{2,1}^\ast&=\sum_{r = 2}^p 2\lambda \sum_{\ell=J_r+1}^{r-1}\sqrt{\sum_{m=J_r+1}^\ell w_{\ell m}^2\Delta_{rm}^2} \nonumber \\
    &\leq \left( \sum_{r=2}^p K_r \right) \lambda^2 a +  \sum_{r=2}^p \sum_{\ell=J_r+1}^{r-1}\sum_{m=J_r+1}^\ell w_{\ell m}^2\Delta_{rm}^2/a \nonumber\\
    &= \left( \sum_{r=2}^p K_r \right) \lambda^2 a + \sum_{r=2}^p \sum_{m=J_r+1}^{r-1}\left(\sum_{\ell=m}^{r-1} w_{\ell m}^2\right) \Delta_{rm}^2/a.
    \nonumber
  \end{align}
  Now let 
  \begin{align}
    a &=\frac{4}{\kappa^2} \max_r\max_{J_r+1 \leq m \leq r-1}\sum_{\ell = m}^{r-1}w_{\ell m}^2
    \nonumber\\
    &= \frac{4}{\kappa^2} \max_r\max_{J_r+1\leq m\leq r-1}\sum_{\ell = m}^{r-1} \frac{1}{\left( \ell - m + 1 \right)^4}
    \leq \sum_{k = 1}^\infty \frac{4}{k^4 \kappa^2} \leq \frac{C_2}{\kappa^2},
    \nonumber
  \end{align} 
  for some constant $C_2 >0$, it follows that
  \begin{align}
    \lambda \left\|\Delta_{\mathcal S}\right\|_{2,1}^\ast& \leq \frac{C_2}{\kappa^2} s \lambda^2 + \norm{\Delta_{\mathcal S}}_F^2 \frac{\kappa^2}{4}
    \leq \frac{C_2}{\kappa^2} s \lambda^2 + \norm{\Delta}_F^2 \frac{\kappa^2}{4}.
    \nonumber
  \end{align}
  Therefore,
  \begin{align}
    \lambda\left( \norm{L + \Delta}_{2,1}^\ast - \norm{L}_{2,1}^\ast \right)
    \geq \lambda \norm{\Delta_{\mathcal{S}^c}}_{2,1} - \frac{C_2}{\kappa^2}s\lambda^2 - \frac{\kappa^2}{4}\norm{\Delta}_F^2.
    \label{eq:boundonc}
  \end{align}
  Finally, combining \eqref{eq:boundona}, \eqref{eq:boundonb}, and \eqref{eq:boundonc}, we have
  \begin{align}
    G(\Delta)&\geq \frac{\kappa^2}{4} \left\|\Delta\right\|_F^2 - C_1\sqrt{\frac{(s+p)\log p}{n}}\left\|\Delta\right\|_F +
    \left( \lambda - C_1 \sqrt{\frac{\log p}{n}}\right) \norm{\Delta_{\mathcal S^c}}_{2,1} -  \frac{C_2}{\kappa^2}s \lambda^2.
    \nonumber
  \end{align}
  For any $\varepsilon<1$, choose
  \begin{align}
    \lambda= \frac{C_1}{\varepsilon}\sqrt{\frac{\log p}{n}}.
    \nonumber
  \end{align}
  Since $\norm{\Delta}_F = Mr_n$, we have
  \begin{align}
    G(\Delta) \geq & \frac{\kappa^2}{4} M^2r_n^2 -C_1Mr_n^2 
    + C_1\sqrt{\frac{\log p}{n}}\left(\frac{1}{\varepsilon}-1\right)\left\|\Delta_{\mathcal S^c}\right\|_{2,1} 
    - \frac{C_2C_1^2}{\kappa^2\varepsilon^2}\frac{s\log p}{n} \nonumber \\
    \geq &\left(\frac{\kappa^2}{4}M^2 -C_1M-\frac{C_2C_1^2}{\kappa^2\varepsilon^2}\right)r_n^2 >0,
    \nonumber
  \end{align}
  for $M$ sufficiently large.
\end{proof}

\section{Proof of Lemma \ref{lem:theory:opt_cond}} \label{proof:opt_cond}
\begin{proof}
  Denote 
  \begin{align}
    &\mathcal L \left( \tau, z, \beta; \nu, \phi,  a^{(\ell)} \right)
    \nonumber\\
    =&
    -2 \log \tau + \frac{1}{n} \left\| z \right\|_2^2 + 
    \nu \left( \tau - \beta_r \right)  \nonumber 
    + \frac{1}{n} \left\langle \phi, z - \mathbf{X}_{1:r}\beta \right\rangle +
    \lambda \sum_{\ell=1}^{r-1} 
    \left\langle W^{(\ell)} \ast a^{(\ell)}, \beta \right\rangle. \nonumber
  \end{align}

  Then the primal \eqref{est:subproblem}
  can be written equivalently as
  \begin{align}
    \min_{\tau,z,\beta}  \left\{ \max_{\nu, \phi, a^{(\ell)}} \left\{
      \mathcal L\left(\tau, z, \beta; \nu, \phi, a^{(\ell)}\right):
      \left\| \left( a^{(\ell)} \right)_{g_{r,\ell}}\right\|_2 \leq 1,
      \left( a^{(\ell)} \right)_{g_{r,\ell}^c} = 0
    \right\} \right\}. \nonumber
  \end{align}

  The dual function can then be written as
  \begin{align}
    g\left( \nu, \phi, a^{(\ell)} \right) =& \inf_{\tau,z,\beta}
    \mathcal L\left(\tau, z, \beta; \nu, \phi, a^{(\ell)} \right) \nonumber\\
    =& \inf_{\tau} \left\{ -2 \log \tau + \nu \tau \right\}
    + \inf_{z} \left\{ \frac{1}{n} \left\| z \right\|_2^2 + 
    \frac{1}{n} \left\langle \phi, z \right\rangle  \right\} \nonumber\\
    &+ \inf_{\beta} \left\{ - \nu \beta_r 
      -\frac{1}{n} \left\langle \mathbf{X}_{1:r}^T \phi, \beta \right\rangle + 
      \lambda \sum_{\ell=1}^{r-1}
      \left\langle W^{(\ell)}\ast a^{(\ell)}, \beta \right\rangle\right\} 
      \nonumber \\
      =& 2 \log \nu - 2\log 2 +
      2 -\mathbb{1}_{\infty} \left\{ \nu > 0 \right\}
      -\frac{1}{4n} \left\| \phi \right\|_2^2
      \nonumber\\
      &-\mathbb{1}_{\infty} \left\{ -\nu \mathbf{e}_r
      - \frac{1}{n}\mathbf{X}_{1:r}^T \phi  
      + \lambda  \sum_{\ell=1}^{r-1} 
      W^{(\ell)}\ast a^{(\ell)} = 0 \right\},
      \nonumber
    \end{align}
    where $\mathbf{e}_r \in \real^r$ is such that $(\mathbf{e}_r)_r = 1$ and 
    $(\mathbf{e}_r)_j = 0$ for all $j \neq r$.
    Thus the dual problem (up to a constant) is 
    \begin{align}
      \max_{\nu, \phi, a^{(\ell)}}& g\left( \nu, \phi, a^{(\ell)} \right) 
      \nonumber \\
      = \min_{\nu, \phi, a^{(\ell)}}& \left\{ -2 \log \nu +
        \frac{1}{4n} \left\| \phi \right\|_2^2 \st \nu >0, \quad
        \left\| \left( a^{(\ell)} \right)_{g_{r,\ell}}\right\|_2 \leq 1,
        \left( a^{(\ell)} \right)_{g_{r,\ell}^c} = 0,  \right. \nonumber \\
        &\left. 
        \nu \mathbf{e}_r + \frac{1}{n} \mathbf{X}_{1:r}^T \phi = \lambda 
        \sum_{\ell=1}^{r-1} W^{(\ell)} \ast a^{(\ell)} \right\}.
        &&
        \nonumber
      \end{align}
      The primal-dual relation is 
      \begin{align}
        \hat{\beta}_{r} = \hat{\tau} = \frac{2}{\hat{\nu}} \quad \quad
        \hat{\phi} = -2 \hat{z} = -2 \mathbf{X}_{1:r}\hat{\beta}. \nonumber
      \end{align} 

      This implies that at optimal points
      \begin{align}
        -\frac{2}{\hat{\beta}_r} \mathbf{e}_r + 2 S_{1:r,1:r} \hat{\beta}  + 
        \lambda  \sum_{\ell=1}^{r-1} W^{(\ell)} \ast 
        \hat{a}^{(\ell)} = 0,\nonumber
      \end{align}
      with $  \left\| \left( \hat{a}^{(\ell)} \right)_{g_{r,\ell}}\right\|_2 \leq 1,
      \left( \hat{a}^{(\ell)} \right)_{g_{r,\ell}^c} = 0$.

      If we denote the objective function as 
      \begin{align}
        f\left( \beta \right) = -2 \log \beta_r + \left\langle S_{1:r,1:r}, 
        \beta \beta^T
        \right\rangle + 
        \lambda P(\beta), \nonumber
      \end{align}
      then from the equality $f( \hat{\beta} ) 
      = \mathcal L \left( \hat{\tau}, \hat{z},
      \hat{\beta}; \hat{\nu}, \hat{\phi}, \hat{a}^{(\ell)} \right)$
      together with the primal-dual relation, we have
      \begin{align}
        P(\hat{\beta}) =  \sum_{\ell=1}^{r-1}
        \left\langle W^{(\ell)} \ast \hat{a}^{(\ell)}, \hat{\beta} \right\rangle
        =   \sum_{\ell=1}^{r-1}
        \left\langle W^{(\ell)} \ast \hat{\beta}, \hat{a}^{(\ell)}\right\rangle.
        \nonumber
      \end{align}
      Suppose there exists some $\ell$ 
      with $\hat{\beta}_{g_{r,\ell}} \neq 0$
      but $\left( \hat{a}^{(\ell)} \right)_{g_{r,\ell}} \neq 
      \frac{\left( W^{(\ell)}\ast \hat{\beta} \right)_{g_{r,\ell}}}
      {\left\| \left( W^{(\ell)}\ast \hat{\beta} \right)_{g_{r,\ell}} \right\|_2}$,
      \newline then $\left\langle W^{(\ell)} \ast \hat{\beta}, 
      \hat{a}^{(\ell)} \right\rangle < 
      \left\| \left( W^{(\ell)}\ast \hat{\beta} \right)_{g_{r,\ell}} \right\|_2$
      while for other $\ell'$ by Cauchy-Schwarz inequality
      we have $\left\langle W^{(\ell')} \ast \hat{\beta}, 
      \hat{a}^{(\ell')} \right\rangle 
      \leq \left\| \left( W^{(\ell')}\ast \hat{\beta} 
      \right)_{g_{r, \ell'}} \right\|_2$.
      Therefore, summing over all $\ell = 1,\dots,r-1$ would give 
      \begin{align}
        P(\hat{\beta}) =  \sum_{\ell=1}^{r-1}
        \left\| \left( W^{(\ell)} \ast \hat{\beta} \right)_{g_{r,\ell}} \right\|_2 
        >  \sum_{r=2}^p \sum_{\ell=1}^{r-1}
        \left\langle W^{(\ell)} \ast \hat{\beta}, \hat{a}^{(\ell)} \right\rangle,
        \nonumber
      \end{align}
      which leads to a contradiction. Thus
      $\left( \hat{a}^{(\ell)} \right)_{g_{r,\ell}} = \frac{\left( W^{(\ell)} \ast 
      \hat \beta \right)_{g_{r,\ell}}}
      {\left\| \left( W^{(\ell)}\ast \hat \beta  \right)_{g_{r,\ell}} \right\|_2}$ for
      $\hat \beta _{g_{r,\ell}} \neq 0 $ and 
      $\left\| \hat{a}^{(\ell)}_{g_{r,\ell}} \right\|_2 \leq 1$ 
      for $\hat{\beta}_{g_{r,\ell}} = 0$.
    \end{proof}

    \section{Proof of Lemma \ref{lem:theory:moresparse}} \label{proof:moresparse}
    \begin{proof}
      In this proof, we continue to use the notation in Appendix \ref{proof:opt_cond}. Observe that 
      $\mathcal L \left( \tau, z, \beta; \nu, \phi, a^{(\ell)} \right)$
      is jointly convex in $\tau$, $z$ and $\beta$, 
      and it is strictly convex in $\tau$ 
      and $z$. Thus, the minimizers $\hat{z}$ and $\hat{\tau}$ are unique. 

      To see this in a more general setting, without loss of generality, suppose
      $f(x,y)$ is convex in $y$ and is strictly convex in $x$. Then for
      $x_1 \neq x_2$ and $\theta \in \left( 0,1 \right)$ we have
      \[
        f\left( \theta x_1 + \left( 1-\theta \right) x_2, y \right) <
        \theta f\left( x_1,y \right) + \left( 1-\theta \right) f\left( x_2,y \right)
      \]
      Now suppose $\left( \hat{x}_1, \hat{y} \right)$ and 
      $\left( \hat{x}_2, \hat{y}_2 \right)$ are both minima of $f$, then taking
      $\theta = 1/2$ we have $f\left( \frac{\hat{x}_1+\hat{x}_2}{2},\hat{y} \right)
      < f\left( \hat{x}_1, \hat{y} \right) = f\left( \hat{x}_2, \hat{y} \right)$, 
      which leads to a contradiction.

      By the 
      primal-dual relation, we know that if $\hat{\beta}$ and $\tilde{\beta}$ are two
      solutions to \eqref{est:subproblem},
      then $\hat{\beta}_{r} = \tilde{\beta}_{r}$ and 
      $\mathbf{X}_{1:r} \hat{\beta} = \mathbf{X}_{1:r}\tilde{\beta}$.
      So from the equality $f(\hat{\beta}) = f(\tilde{\beta})$ we know that 
      $P(\tilde{\beta}) = 
      P(\hat{\beta})$. Also
      by
      \begin{align}
        f\left( \hat{\beta} \right) = \mathcal L 
        \left( \hat{\tau}, \hat{z}, \hat{\beta};
        \hat{\nu}, \hat{\phi}, \hat{a}^{(\ell)} \right) 
        \leq \mathcal L \left( \hat{\tau}, \hat{z}, \tilde{\beta};
        \hat{\nu}, \hat{\phi}, \hat{a}^{(\ell)} \right) 
        \leq \mathcal L \left( \tilde{\tau}, \tilde{z}, \tilde{\beta};
        \tilde{\nu}, \tilde{\phi}, \tilde{a}^{(\ell)} \right) 
        = f\left( \tilde{\beta} \right), \nonumber
      \end{align}
      we have 
      \begin{align}
        \mathcal L \left( \hat{\tau}, \hat{z}, \hat{\beta};
        \hat{\nu}, \hat{\phi}, \hat{a}^{(\ell)} \right) 
        = \mathcal L \left( \hat{\tau}, \hat{z}, \tilde{\beta};
        \hat{\nu}, \hat{\phi}, \hat{a}^{(\ell)} \right) ,
        \nonumber
      \end{align}
      and thus
      \begin{align}
        \sum_{\ell=1}^{r-1} 
        \left\langle W^{(\ell)}\ast \hat{a}^{(\ell)}
        , \tilde{\beta} \right\rangle = 
        \sum_{\ell=1}^{r-1} 
        \left\langle W^{(\ell)}\ast \hat{a}^{(\ell)}
        , \hat{\beta} \right\rangle 
        = P(\hat{\beta}) 
        = P(\tilde{\beta})
        =  \sum_{\ell=1}^{r-1} 
        \left\| \left( W^{(\ell)}\ast \tilde{\beta} \right)_{g_{r,\ell}} \right\|_2.
        \nonumber
      \end{align}

      Now for any $\ell \leq r-1$ suppose 
      $\left\| \left( \hat{a}^{(\ell)} \right)_{g_{r,\ell}} \right\|_2 < 1$,
      then for the equality above to hold, we must have 
      $\tilde{\beta}_{g_{r,\ell}} = 0$.
      Therefore, by Lemma \ref{lem:theory:opt_cond},
      $\hat{\beta}_{g_{r,\ell}} = 0 \implies \tilde{\beta}_{g_{r,\ell}} = 0 $,
      so any other solutions to \eqref{est:subproblem} cannot be
      less sparse than $\hat{\beta}$.
    \end{proof}

    \section{Proof of Lemma \ref{lem:theory:uniqueness}} \label{proof:uniqueness}
    \begin{proof}
      By Lemma \ref{lem:theory:moresparse}, any other solution $\beta$ to 
      \eqref{est:subproblem} must have 
      $\beta_{g_{J(\hat{\beta})}} = 0$.
      Recall that $J( \hat \beta ) = r - 1 - K( \hat \beta )$.
      The original problem \eqref{est:subproblem} can thus be written equivalently as 
      \begin{align}
        \min_{\gamma \in \real^{K(\hat\beta)+1}} 
        -2 \log \gamma_{K(\hat \beta)+1} + \frac{1}{n} 
        \norm{\mathbf{X}_{\hat{\mathcal{S}}}\gamma}_2^2 + \lambda
        \sum_{\ell = 1}^{K(\hat{\beta})}
        \norm{\left( \hat{W}^{(\ell)} \ast \gamma \right)_{g_{r,\ell}}}_2,
        \nonumber
      \end{align}
      where $\hat{W}^{(\ell)} = \left( W^{(\ell+\hat{J})} \right)_{\hat{S}}$.

      Note that the penalty term is a convex function of $\gamma$. 
      The Hessian matrix of the first term is a diagonal matrix of 
      dimension $|\hat{\mathcal S}|= K( \hat\beta )+1$ 
      with non-negative entries in the diagonal. 
      The Hessian matrix of the second term is 
      $2 S_{\hat{\mathcal{S}}\hat{\mathcal{S}}}$. 
      Then by Assumption \ref{agaussian}, the uniqueness follows from strict convexity. 
    \end{proof}

    \section{Proof of Lemma \ref{lem:boundingvarA}} 
    \label{proof:boundingvarA}
    \begin{proof}
      Recall that 
      \[
        M_n = \frac{1}{n} 
        \left( \sum_{\ell=1}^{r-1} W^{(\ell)}\ast \tilde{a}^{(\ell)} 
        \right)^T_{\mathcal{I}} \left( \frac{1}{n} \mathbf{X}_{\mathcal{I}}^T
        \mathbf{X}_{\mathcal{I}}\right)^{-1} 
        \left(\sum_{\ell=1}^{r-1} W^{(\ell)}\ast \tilde{a}^{(\ell)} 
        \right)_{\mathcal{I}} + \frac{4}{n^2 \lambda^2} \tilde{\beta}^2_r
        \norm{\mathbf{O}_{\mathcal{I}} E_r}^2_2.
      \]
      We cite Lemma $9$ (specifically in the form (60)) in 
      \cite{wainwright2009sharp} here for completeness.
      \begin{lemma}[\citealt{wainwright2009sharp}]
        \label{lem:mineigenvalue}
        For $k \leq n$, let $\mathbf{X}_{\mathcal{I}} \in \real^{n\times k}$ have i.i.d. rows 
        from a multivariate Gaussian distribution with mean $\mathbf{0}$ and covariance matrix $\Sigma$. If
        $\Sigma$ has minimum eigenvalue $\kappa >0$, then
        \[
          \Prob\left[ \vertiii{\left( \frac{1}{n} 
            X_{\mathcal{I}}^T X_{\mathcal{I}} \right)^{-1}}_2 \geq \frac{9}{\kappa}
          \right] \leq 2 \exp \left( -\frac{n}{2} \right).
        \]
      \end{lemma}

      By the lemma above, Assumption \ref{abddsval}, and 
      \eqref{eq:boundingdualvariable}
      \begin{align}
        \frac{1}{n}
        \left( \sum_{\ell=1}^{r-1} W^{(\ell)} \ast 
        \tilde{a}^{(\ell)}\right)_{\mathcal I} ^T
        \left(\frac{1}{n} \mathbf{X}_{\mathcal I}^T 
        \mathbf{X}_{\mathcal I}\right)^{-1} 
        \left( \sum_{\ell=1}^{r-1} W^{(\ell)} \ast 
        \tilde{a}^{(\ell)}\right)_{\mathcal I} 
        &\leq \frac{9 \kappa^2}{n}  
        \norm{ \left( \sum_{\ell=1}^{r-1} W^{(\ell)} 
        \ast\tilde{a}^{(\ell)}\right)_{\mathcal I}}^2
        \nonumber\\
        &\leq \frac{3 \pi^2 \kappa^2}{2} \frac{K}{n},
        \nonumber
      \end{align}
      with probability greater than $1-2\exp\left( -\frac{n}{2} \right)$.

      Next we deal with the second term in $M_n$. Recall from \eqref{eq:betar} that
      \begin{align}
        \frac{4}{n^2 \lambda^2} \tilde{\beta}^2_r
        \norm{\mathbf{O}_{\mathcal{I}} E_r}^2_2 & =
        \frac{4}{n^2} \left( \frac{\frac{1}{2} 
        \mathbf{X}_r^T \mathbf{C}_{\mathcal{I}} + \sqrt{\frac{1}{4} 
        \left( \mathbf{X}^T_r \mathbf{C}_{\mathcal{I}} \right)^2 +
        \frac{4}{\lambda^2 n}\norm{\mathbf{O}_{\mathcal{I}}E_r}^2_2}}
        {\frac{2}{n} \norm{\mathbf{O}_{\mathcal{I}}E_r}^2_2} \right)^2
        \norm{\mathbf{O}_{\mathcal{I}}E_r}^2_2
        \nonumber\\
        &\leq \frac{4}{n^2} \frac{\frac{1}{4}\left( \mathbf{X}^T_r 
        \mathbf{C}_{\mathcal{I}} \right)^2 + \frac{4}{\lambda^2 n}
        \norm{\mathbf{O}_{\mathcal{I}}E_r}^2_2}{\frac{1}{n^2}
        \norm{\mathbf{O}_{\mathcal{I}}E_r}^4_2} \norm{\mathbf{O}_{\mathcal{I}}
      E_r}^2_2
      \nonumber\\
      &=  \frac{\left( \mathbf{X}^T_r 
      \mathbf{C}_{\mathcal{I}} \right)^2 }
      {\norm{\mathbf{O}_{\mathcal{I}}E_r}^2_2} + \frac{16}{\lambda^2 n}.
      \nonumber
    \end{align}

    The next lemma gives us a handle on the numerator of the first term.
    \begin{lemma}
      Using the general weight \eqref{est:generalweight}, 
      we have
      \begin{align}
        \Prob \left[ 
          \left| \mathbf{X}_r^T \mathbf{C}_{\mathcal{I}}\right| \geq 
        1 \right] \leq  
        2\exp\left( -\frac{n\alpha^2}{3 \theta \kappa^2 \pi^2 K} \right)
        + 2 \exp \left( -\frac{n}{2} \right).
        \nonumber
      \end{align}
      \label{lem:boundingXrCI} 
    \end{lemma}
    \begin{proof}
      Conditioned on $\mathbf{X}_{\mathcal{I}}$, from the decomposition 
      \eqref{eq:decomposition1} and the definition of $\mathbf{C}_{\mathcal{I}}$
      \[
        \mathbf{X}^T_r \mathbf{C}_{\mathcal{I}} = 
        \Sigma_{r\mathcal{I}} 
        \left( \Sigma_{\mathcal{I}\mathcal{I}} \right)^{-1}
        \left( \sum_{\ell=1}^{r-1} W^{(\ell)} \ast 
        \tilde{a}^{(\ell)} \right)_{\mathcal{I}} + 
        E_r^T \mathbf{X}_{\mathcal{I}} \left( \mathbf{X}^T_{\mathcal{I}}
        \mathbf{X}_{\mathcal{I}}\right)^{-1} \left( \sum_{\ell=1}^{r-1}
        W^{(\ell)}\ast \tilde{a}^{(\ell)}\right)_{\mathcal{I}}.
      \]
      By the irrepresentable assumption \eqref{airrepre} and \eqref{eq:boundingdualvariable},
      \[
        \Sigma_{r\mathcal{I}} 
        \left( \Sigma_{\mathcal{I}\mathcal{I}} \right)^{-1}
        \left( \sum_{\ell=1}^{r-1} W^{(\ell)} \ast 
        \tilde{a}^{(\ell)} \right)_{\mathcal{I}} \leq 1-\alpha.
      \]
      Note that $\Var\left( E_{ir} \right) = \theta_r^{(r)}$ for $i=1,\dots,n$.
      Let $B^{(r)}= E_r^T \mathbf{X}_{\mathcal{I}} \left( \mathbf{X}^T_{\mathcal{I}}
      \mathbf{X}_{\mathcal{I}}\right)^{-1} \left( \sum_{\ell=1}^{r-1}
      W^{(\ell)}\ast \tilde{a}^{(\ell)}\right)_{\mathcal{I}}$.
      By Lemma \ref{lem:mineigenvalue}, $B^{(r)}$ has mean zero and variance at most
      \[
        \Var\left( B^{(r)} \Big| \mathbf{X}_{\mathcal{I}} \right)  = 
        \frac{\theta^{(r)}_r}{n} \left( \sum_{\ell=1}^{r-1} W^{(\ell)} \ast 
        \tilde{a}^{(\ell)} \right)_{\mathcal{I}}^T \left( \frac{1}{n}
        \mathbf{X}_{\mathcal{I}}^T \mathbf{X}_{\mathcal{I}}\right)^{-1} 
        \left( \sum_{\ell=1}^{r-1} W^{(\ell)} \ast 
        \tilde{a}^{(\ell)} \right)_{\mathcal{I}}
        \leq \frac{3\theta_r^{(r)} \kappa^2 \pi^2 K}{2 n},
      \]
      with probability greater than $1-2\exp\left( \frac{n}{2} \right)$. By
      Lemma \ref{lem:probargu}, we have that
      \[
        \Prob \left[ B^{(r)} \geq \alpha \right] \leq 2 \exp
        \left( -\frac{n\alpha^2}{3 \theta_r^{(r)} \kappa^2 \pi^2 K} \right)
        + 2 \exp\left( -\frac{n}{2} \right).
        \nonumber
      \]
    \end{proof}

    Since $\frac{\norm{\mathbf{O}_{\mathcal{I}}E_r}^2_2}{\theta^{(r)}_r} \sim 
    \chi^2 \left( n-K \right)$. To bound it, 
    we cite a concentration inequality from \cite{wainwright2009sharp} 
    (specifically (54b)) as the following lemma:
    \begin{lemma}[Tail Bounds for $\chi^2$-variates, \citealt{wainwright2009sharp}]
      \label{lem:chisqconcentration}
      For a centralized $\chi^2$-variate $X$ with $d$ degrees of freedom, for all
      $\varepsilon \in \left( 0, 1/2 \right)$, we have
      \begin{align}
        \Prob \left[ X \leq d(1-\varepsilon) \right] \leq \exp \left( 
        -\frac{1}{4}d\varepsilon^2\right).
        \nonumber
      \end{align}
    \end{lemma}

    From Lemma \ref{lem:chisqconcentration} it follows that
    \begin{align}
      \Prob\left[ \norm{\mathbf{O}_{\mathcal{I}} E_r}^2_2
      \leq  \theta_r^{(r)} \left( n-K  \right) \left( 1-\varepsilon \right)\right]
      \leq \exp \left( -\frac{1}{4}\left( n-K \right)\varepsilon^2 \right),
      \nonumber
    \end{align}
    which together with Lemma \ref{lem:boundingXrCI} implies that
    \begin{align}
      % \Prob \left[ \frac{\left( \mathbf{X}^T_r 
      %   \mathbf{C}_{\mathcal{I}} \right)^2 }
      %   {\norm{\mathbf{O}_{\mathcal{I}}E_r}^2_2}\geq
      %   \frac{1}{\theta_r^{(r)} \left( n-K \right)
      % \left( 1-\varepsilon \right)}\right]
      % \leq 2\exp\left( -\frac{n\alpha^2}{3 \theta_r^{(r)} \kappa^2 \pi^2 K} \right)
      % + 2 \exp \left( -\frac{n}{2} \right) + \exp \left( 
      % -\frac{1}{4}\left( n-K \right)\varepsilon^2\right).
      % \nonumber
      &\Prob \left[ \frac{\left( \mathbf{X}^T_r 
        \mathbf{C}_{\mathcal{I}} \right)^2 }
        {\norm{\mathbf{O}_{\mathcal{I}}E_r}^2_2}\geq
        \frac{1}{\theta_r^{(r)} \left( n-K \right)
      \left( 1-\varepsilon \right)}\right] \nonumber\\
      \leq& 2\exp\left( -\frac{n\alpha^2}{3 \theta_r^{(r)} \kappa^2 \pi^2 K} \right)
      + 2 \exp \left( -\frac{n}{2} \right) + \exp \left( 
      -\frac{1}{4}\left( n-K \right)\varepsilon^2\right).
      \nonumber
    \end{align}
    The result follows from a union bound.
  \end{proof}

  \section{Proof of Lemma \ref{lem:concentration}} \label{proof:concentration}
  \begin{proof}
    The proof strategy is based on the proof of Lemma 2 in \citet{bien2015convex}.

    For the design matrix $\mathbf{X}_{n\times p}$ with independent rows, 
    denote $X_i = \left( \mathbf{X}_{i\cdot} \right)^T \in \real^p$. 
    Then $X_i$ are i.i.d with mean 0 and true covariance matrix 
    $\Sigma = \left( {L}^T L \right)^{-1}$ for $i=1,...,n$. 
    And $\bar{X} = \frac{1}{n}\sum_{i=1}^n X_i$ has mean 0 
    and true covariance matrix $\frac{1}{n}\Sigma$.

    Let $Y_i = L X_i \in \real^p$. 
    Then $Y_i$ are i.i.d with mean 0 and true covariance matrix 
    $L \tSig {L}^T = L \left( {L}^T L \right)^{-1}
    {L}^T= \mathbf{I}_p$. 
    And $\bar{Y} = \frac{1}{n}\sum_{i=1}^n Y_i = 
    \frac{1}{n}\sum_{i=1}^n L X_i = L \bar{X}$ 
    has mean zero and covariance matrix $\frac{1}{n}\mathbf{I}_p$. 
    Also the corresponding design matrix $\mathbf{Y} = \mathbf{X} {L}^T$ 
    has independent rows.
    \begin{align}
      S\tLt &= \frac{1}{n} \sum_{i=1}^n \left( X_i - \bar{X} \right) 
      \left( X_i - \bar{X} \right)^T {L}^T  \nonumber\\
      &= \frac{1}{n} \sum_{i=1}^n \left(  X_i -  \bar{X} \right) 
      \left( L X_i - L \bar{X} \right)^T 
      =\frac{1}{n} \sum_{i=1}^n \left(  X_i -  \bar{X} \right) 
      \left(  Y_i - \bar{Y} \right)^T.
      \nonumber
    \end{align}
    So we have
    \[
      \left( S {L}^{T} \right)_{ij} = n^{-1} \sum_{k=1}^p X_{ki}Y_{kj} - 
      \bar{X}_i \bar{Y}_j.
    \]
    Letting
    \[
      \mathcal W = S{L}^T - L^{-1},
    \]
    we have that
    \[ 
      \left| \mathcal{W}_{ij} \right| \leq 
      \left| n^{-1} \sum_{k=1}^p X_{ki}Y_{kj} - 
      \left( L^{-1} \right)_{ij}\right| + \left | \bar{X}_i \bar{Y}_j \right|.
    \]
    \begin{align}
      &\Prob \left[ \max_{ij} \left| \mathcal W\right|_{ij} > t \right]  
      \nonumber\\
      \leq & \Prob \left[ \max_{ij}\left| n^{-1} \sum_{k=1}^p X_{ki}Y_{kj} - 
      \left( \tL ^{-1}\right)_{ij}\right| > \frac{t}{2}\right] + 
      \Prob \left[ \max_{ij} \left | \bar{X}_i \bar{Y}_j \right| > 
      \frac{t}{2}\right] 
      \nonumber\\
      \leq & \Prob \left[ \left| n^{-1} \sum_{k=1}^p X_{ki}Y_{kj} - 
        \left( \tL^{-1} \right)_{ij}\right| > \frac{t}{2} 
      \,\,\text{for some } \, i,j \right] \nonumber\\
      &+ 
      \Prob \left[ \max_{i} \left | \bar{X}_i \right| > 
      \sqrt{\frac{t}{2}}\right] + 
      \Prob \left[ \max_{j} \left | \bar{Y}_j \right| > 
      \sqrt{\frac{t}{2}}\right] 
      \nonumber\\
      \leq & \sum_{ij}\Prob \left[ \left| n^{-1} \sum_{k=1}^p X_{ki}Y_{kj} -  
      \left( \tL^{-1} \right)_{ij}\right| > \frac{t}{2} \right] +  
      \sum_i \Prob \left[\left | \bar{X}_i \right| > 
      \sqrt{\frac{t}{2}}\right] + \sum_{j} 
      \Prob \left[ \left | \bar{Y}_j \right| > \sqrt{\frac{t}{2}}\right] 
      \nonumber\\
      \leq & p^2\max_{ij}\Prob \left[ \left| n^{-1} \sum_{k=1}^p X_{ki}Y_{kj} -
      \left( \tL^{-1} \right)_{ij}\right| > \frac{t}{2} \right] 
      \nonumber\\
      &+ p \max_i \Prob \left[\left | \bar{X}_i \right| > 
      \sqrt{\frac{t}{2}}\right] + 
      p \max_{j} \Prob \left[ \left | \bar{Y}_j \right| > 
      \sqrt{\frac{t}{2}}\right] 
      \nonumber\\
      := & p^2 \max_{ij} I_{ij} + p \max_i I^X_i + p\max_j I^Y_j.
      \nonumber
    \end{align}

    Consider $I^X_{i}$ first. 
    Since $X_{ki}$ are independent sub-Gaussian with variance 
    $\Sigma_{ii}$ for $k=1,..,n$, we have

    \begin{align}
      \E \exp \left( t \frac{\bar{X}_i}{\sqrt{\tSig_{ii}/n}}\right) 
      & = \prod_{k=1}^n \E \exp \left( t \frac{X_{ki}}{\sqrt{n\tSig_{ii}}}\right)
      \quad \text{by independence}
      \nonumber \\
      & \leq \prod_{k=1}^n \exp \left( \tilde{C}_1 t^2/n\right) 
      = \exp(\tilde C_1 t^2)
      \quad \text{by the definition of sub-Gaussian},
      \nonumber
    \end{align}
    so $\bar{X}_i$ is sub-Gaussian with variance $\Sigma_{ii}/n$.

    By Lemma 5.5 in \cite{vershynin2010introduction}, we have
    \[
      \Prob \left[ \left| \bar{X}_i\right|/ \sqrt{\Sigma_{ii}^\ast}
      >t \right) \leq \exp \left( 1- t^2/K_1^2 \right],
    \]
    where $K_1$ is a constant that does not depend on $i$. 

    Following the same argument we have
    \[
      \E \exp\left( t \bar{Y}_i / \sqrt{1/n} \right) = \prod_{k=1}^n
      \E \exp\left( t Y_{ki} / \sqrt{n} \right) 
      \leq \exp \left( \tilde C_2 t^2 \right),
    \]
    thus
    \[
      \Prob \left[ \left| \bar{Y}_i\right|/ \sqrt{1/n}
      >t \right) \leq \exp \left( 1- t^2/K_2^2 \right],
    \]
    where $K_2$ is a constant that does not depend on $i$. 
    And we have 
    \begin{align}
      I_i^{X} + I_i^{Y} &= 
      \Prob \left[ \left| \bar{X}_i\right| > \sqrt{t/2} \right] + 
      \Prob \left[ \left| \bar{Y}_i\right| > \sqrt{t/2} \right] 
      \nonumber\\
      &= \Prob \left[ \frac{\left| \bar{X}_i\right|}{\sqrt{\Sigma_{ii}/n}}
      > \frac{\sqrt{t/2}}{\sqrt{\Sigma_{ii}/n}} \right] + 
      \Prob \left[ \left| \frac{\bar{Y}_i}{\sqrt{1/n}}\right| 
      > \frac{\sqrt{t/2}}{\sqrt{1/n}} \right] \nonumber\\
      &\leq \exp \left( 1- \frac{nt}{2K_1^2 \Sigma_{ii}^\ast} \right) + 
      \exp \left( 1- \frac{nt}{2K_2^2} \right).
      \nonumber 
    \end{align}
    Thus
    \[
      \max_i{\left( I_i^{X}+ I_i^{Y} \right)} \leq 4 
      \exp\left( -\frac{C_1nt}{\max_i \Sigma_{ii}^\ast} \right) + 
      4\exp\left( -C_2nt \right)
    \]
    for some constant $C_1$.

    Now consider the term $I_{ij}$. We have shown that both $\mathbf{X}$ and 
    $\mathbf{Y}$ have independent rows. So for any $i,j$, 
    $Z_k^{(ij)} = X_{ki} Y_{kj}$ are independent for $k=1,\dots,n$.
    Let $X \sim N \left( \textbf{0}, \Sigma \right)$ and $Y \sim N\left( \text{0}, \mathbf{I}_p \right)$, then
    \[
      \E \left( X_{ki}Y_{kj} \right) = \cov \left( X, \tL X \right)_{ij} - 0
      = \left[ \cov \left( X,X \right)\tLt \right]_{ij} 
      = \left( \Sigma \tLt \right)_{ij} = \left( \tL^{-1} \right)_{ij}.
    \]

    If there exist $\nu_{ij}$ and $c_{ij}$ such that
    \begin{align}
      & \sum_{k=1}^n \E \left( X_{ki}^2 Y_{kj}^2 \right) \leq \nu_{ij} 
      \nonumber\\
      & \sum_{k=1}^n \E \left\{ \left( X_{ki}Y_{kj} \right)_+^q \right\}
      \leq \frac{q!}{2}\nu_{ij}c_{ij}^{q-2}
      \quad \text{for some} \quad q \geq 3 \in \mathbb N,
      \nonumber
    \end{align}
    then by Theorem 2.10 (Corollary 2.11) in \cite{boucheron2013concentration},
    $\forall t>0$, we have
    \[
      \Prob \left[ \left| \sum_{k=1}^n \left( X_{ki}Y_{kj} - 
      \left( \tL \right)^{-1}_{ij}\right)\right| >t  \right]\leq 
      2 \exp \left( - \frac{t^2}{2\left( \nu_{ij} + c_{ij}t \right)}\right).
    \]

    The rest of the proof focuses on characterizing $\nu_{ij}$ and $c_{ij}$. 
    First, Lemma 5.5 in \cite{vershynin2010introduction} shows that, for some constant $K_3$ that does not depend on $j$,
    \begin{align}
      \left( \E \left| X_{ij} / \sqrt{\Sigma_{jj}}\right|^q \right)^{1/q}
      \leq K_3 \sqrt{q}
      \nonumber
    \end{align}
    holds for all $q \geq 1$. Thus,
    \begin{align}
      \E \left|X_{ij}\right|^q \leq K_3^q q^{q/2} 
      \left( \Sigma_{jj} \right)^{q/2}.
      \nonumber
    \end{align}
    Following the same argument, there exists some constant $K_4$ that does not depend on $j$ such that
    \begin{align}
      \E \left|Y_{ij}\right|^q \leq K_4^q q^{q/2}
      \nonumber
    \end{align}
    for all $q \geq 1$.

    Therefore,
    \[
      \sum_{k=1}^n \E \left( X_{ki}^2Y_{kj}^2 \right) \leq
      \sum_{k=1}^n \sqrt{\E X_{ki}^4 \E Y_{kj}^4} \leq 
      n \sqrt{K_3^4 2^4 K_4^4 2^4 {\Sigma_{ii}}^2}
      = 16 n K_3^2K_4^2 \Sigma_{ii},
    \]
    and
    \[
      \sum_{k=1}^n \E \left\{ \left( X_{ki}Y_{kj} \right)_+^q \right\} \leq 
      \sum_{k=1}^n \sqrt{\E X_{ki}^{2q} \E Y_{kj}^{2q}} \leq 
      n \sqrt{K_3^{2q} \left( 2q \right)^{2q} K_4^{2q} 
      \left(\Sigma_{ii}\right)^2}
      = n K_3^qK_4^q \left( 2q \right)^q \left( \Sigma_{ii} \right)^{q/2}.
    \]
    So taking
    \begin{align}
      \nu_{ij} = K_5 n \Sigma_{ii}^\ast, \nonumber\\
      c_{ij} = K_5 \sqrt{\Sigma_{ii}^\ast} \nonumber
    \end{align}
    for some $K_5$ large enough and does not depend on $i,j$.

    Now we have 
    \[	  
      I_{ij} \leq 2 \exp \left( -\frac{n^2t^2}
      {4\left( 2\nu_{ij} + c_{ij}tn \right)} \right)
      = 2 \exp \left( -\frac{nt^2}{4\left( 2K_5 \Sigma_{ii}^\ast +
      K_5 \sqrt{\Sigma_{ii}}t\right)} \right).
    \]
    If $t \leq 2 \max_i \sqrt{\Sigma_{ii}^\ast}$, then with 
    $C_3 = \left( 16 K_5 \right)^{-1}$ we have 
    \[
      I_{ij} \leq 2 \exp\left( -\frac{C_2 nt^2}{\max_i \Sigma_{ii}^\ast} \right).
    \]

    To sum up, for any $0 < t \leq 2 \max_i \sqrt{\Sigma_{ii}^\ast}$,
    \[
      \Prob \left[ \max_{ij} \left| \mathcal W_{ij}\right| > t \right]
      \leq 2p^2 \exp \left( -\frac{C_2nt^2}{\max_i \Sigma_{ii}^\ast} \right)
      + 4p \exp \left( - \frac{C_1nt}{\max_i \Sigma_{ii}^\ast} \right) + 
      4p \exp \left( - C_2nt \right).
    \]
  \end{proof}

  \vskip 0.2in
  \bibliographystyle{plainnat}
  \bibliography{citation.bib}
  \end{document}